\documentclass[12pt]{amsart}

\usepackage[paperheight=11in,paperwidth=8.5in,left=1in,top=1.25in,right=1in,bottom=1.25in,]{geometry}
\usepackage[linktocpage=false,colorlinks=true,linkcolor=black,citecolor=black,
filecolor=black,urlcolor=black,pagebackref=false,pdfstartpage={1},
pdfstartview={FitH},pdftitle={},pdfauthor={},pdfsubject={},pdfcreator={},
pdfproducer={},pdfkeywords={}]{hyperref}

\usepackage{rotating}
\usepackage{pdflscape}
\usepackage{manfnt}

\usepackage{chngpage}
\usepackage{diagbox}
\usepackage[lofdepth,lotdepth]{subfig}
\usepackage{comment}
\usepackage{afterpage}
\usepackage{amsfonts}
\usepackage{amsmath}
\allowdisplaybreaks[4]
\usepackage{amssymb}
\usepackage{amsthm}
\usepackage{bbm}
\usepackage[justification=centering,margin=1cm]{caption}
\usepackage[capitalize,nameinlink,noabbrev,nosort]{cleveref}
\usepackage{enumerate}
\usepackage{enumitem}
\usepackage{float}
\restylefloat{figure}
\usepackage{kbordermatrix}
\usepackage{youngtab}
\usepackage{psfrag}
\usepackage[aligntableaux=center]{ytableau}

\usepackage{graphicx}
\usepackage{mathrsfs}
\usepackage{mathtools}
\usepackage[framemethod=tikz,ntheorem,xcolor]{mdframed}
\usepackage{parcolumns}
\usepackage{tabularx}
\usepackage{tikz}\pdfpageattr{/Group <</S /Transparency /I true /CS /DeviceRGB>>}
\usetikzlibrary{arrows,calc,intersections,matrix,positioning,through}
\usepackage{tikz-cd}
\tikzset{commutative diagrams/.cd,every label/.append style = {font = \normalsize}}
\usepackage[all]{xy}
\usepackage{rotating}

\numberwithin{equation}{section}

\newtheorem*{theorem*}{Theorem}
\newtheorem*{corollary*}{Corollary}
\newtheorem{thm}[equation]{Theorem}
\AfterEndEnvironment{thm}{\noindent\ignorespaces}
\newtheorem{theorem}[equation]{Theorem}
\AfterEndEnvironment{theorem}{\noindent\ignorespaces}
\newtheorem{cor}[equation]{Corollary}
\newtheorem{corollary}[equation]{Corollary}
\AfterEndEnvironment{cor}{\noindent\ignorespaces}
\newtheorem{lemma}[equation]{Lemma}
\AfterEndEnvironment{lemma}{\noindent\ignorespaces}
\newtheorem{lem}[equation]{Lemma}
\AfterEndEnvironment{lem}{\noindent\ignorespaces}
\newtheorem{proposition}[equation]{Proposition}
\AfterEndEnvironment{proposition}{\noindent\ignorespaces}
\newtheorem{prop}[equation]{Proposition}
\AfterEndEnvironment{prop}{\noindent\ignorespaces}
\newtheorem{conj}[equation]{Conjecture}
\AfterEndEnvironment{conj}{\noindent\ignorespaces}

\AfterEndEnvironment{prob}{\noindent\ignorespaces}

\AfterEndEnvironment{question}{\noindent\ignorespaces}

\theoremstyle{definition}
\newtheorem{defn}[equation]{Definition}
\newtheorem{definition}[equation]{Definition}
\AfterEndEnvironment{defn}{\noindent\ignorespaces}
\AfterEndEnvironment{definition}{\noindent\ignorespaces}
\newtheorem*{pf_no_qed}{Proof}

\AfterEndEnvironment{pf}{\noindent\ignorespaces}
\newtheorem{eg_no_qed}[equation]{Example}

\AfterEndEnvironment{eg}{\noindent\ignorespaces}
\newenvironment{example}[1][]{\begin{eg_no_qed}[#1]\pushQED{\qed}}{\popQED\end{eg_no_qed}}
\AfterEndEnvironment{example}{\noindent\ignorespaces}
\newtheorem{rmk}[equation]{Remark}
\AfterEndEnvironment{rmk}{\noindent\ignorespaces}
\newtheorem{remark}[equation]{Remark}
\AfterEndEnvironment{remark}{\noindent\ignorespaces}

\theoremstyle{remark}

\AfterEndEnvironment{claim}{\noindent\ignorespaces}
\newtheorem*{claimpf_no_qed}{Proof of Claim}

\AfterEndEnvironment{claimpf}{\noindent\ignorespaces}

\font\pipefont=lcircle10
\def\elbow{\smash{\raise3pt\hbox{\pipefont\rlap{\rlap{\char'014}\char'016}}}}

\def\halfelbow{\smash{\raise2pt\hbox{\pipefont\rlap{\rlap{\rlap{\char'015}\phantom{\char'017}}}}}}
\def\cross{\smash{\lower5pt\hbox{\rlap{\vrule height16pt}}\raise3pt\hbox{\rlap{\hskip-8pt \vrule height0.4pt depth0pt width16pt}}}}

\def\AA{\mathcal{A}}

\def\B{\mathcal{B}}
\def\M{\mathcal{M}}

\def\PP{\mathbb{P}}

\def\T{\mathcal{T}}
\def\C{\mathbb{C}}

\def\mcv{\mathcal{V}}
\def\Aalpha{\boldsymbol\alpha}
\def\Bbeta{\boldsymbol\beta}
\def\Ggamma{\boldsymbol\gamma}
\DeclareMathOperator{\inv}{inv}
\DeclareMathOperator{\flip}{Flip}
\DeclareMathOperator{\Facet}{Facet}
\DeclareMathOperator{\area}{area}
\DeclareMathOperator{\punc}{punc}
\DeclareMathOperator{\Edges}{Edges}
\DeclareMathOperator{\convex}{convex}

\DeclareMathOperator{\Bound}{Bound}
\newcommand{\Bdkn}{\Bound(k, n)}

\DeclareMathOperator{\Gr}{Gr}
\DeclareMathOperator{\Id}{Id}

\DeclareMathOperator{\Mat}{Mat}

\DeclareMathOperator{\codim}{codim}

\newcommand{\Grk}{\Gr_{k,n}^{\ge 0}}
\newcommand{\A}{\mathcal{A}_{n,k,m}}

\newcommand{\R}{\mathbb{R}}
\newcommand{\DD}{\mathbb{D}}
\newcommand{\RR}{\mathbb{R}}
\newcommand{\Z}{\mathbb{Z}}

\DeclareMathOperator{\conv}{conv}

\DeclareMathOperator{\rank}{rank}

\DeclareMathOperator{\rowspan}{rowspan}
\newcommand{\rf}[1]{\hyperref[#1]{(\ref*{#1})}}
\DeclareMathOperator{\sgn}{sgn}
\DeclareMathOperator{\sign}{sign}

\DeclareMathOperator{\spn}{span}

\DeclareMathOperator{\Trop}{Trop}

\DeclareMathOperator{\var}{var}
\DeclareMathOperator{\wt}{wt}

\DeclareMathOperator{\des}{Des_L}
\DeclareMathOperator{\cdes}{cDes_L}

\newcommand{\hatG}{\hat{G}}
\newcommand{\simp}[1]{\Delta_{#1}}
\newcommand{\asimp}[1]{\hat{\Delta}_{#1}(Z)}
\newcommand{\asimpo}[1]{\hat{\Delta}_{#1}^{\circ}(Z)}
\newcommand{\td}[1]{\hat{#1}} 
\newcommand{\gt}[1]{Z_{#1}} 
\newcommand{\gto}[1]{Z_{#1}^\circ} 
\newcommand{\lrangle}[1]{\langle #1 \rangle}
\newcommand{\twMtx}{C^{\text{\normalfont tw}}_\T}


 \title[The $m=2$ amplituhedron and the Hypersimplex]
 {The $m=2$ amplituhedron and the Hypersimplex: \\ \small{Signs, clusters, tilings, Eulerian numbers}}

\author{Matteo Parisi}
\author{Melissa Sherman-Bennett}
\author{Lauren K.\ Williams}
\address{}
\email{\href{mailto:mparisi@cmsa.fas.harvard.edu}{mparisi@cmsa.fas.harvard.edu}}
\email{\href{mailto:msherben@mit.edu}{msherben@mit.edu}}
\email{\href{mailto:williams@math.harvard.edu}{williams@math.harvard.edu}}

\begin{document}

\begin{abstract}

The hypersimplex $\Delta_{k+1,n}$ is the image of the positive
Grassmannian $\Gr^{\geq 0}_{k+1,n}$ under the moment map.  
It is a polytope of dimension $n-1$ in $\R^n$.  
Meanwhile, the amplituhedron
$\mathcal{A}_{n,k,2}(Z)$ is the projection of the
positive Grassmannian $\Gr^{\geq 0}_{k,n}$ into the
Grassmannian $\Gr_{k,k+2}$ under a map $\tilde{Z}$ induced by 
a positive matrix $Z\in \Mat_{n,k+2}^{>0}$.
Introduced in the context of \emph{scattering amplitudes}, it is not a polytope, and 
has full dimension $2k$ inside $\Gr_{k,k+2}$.  
Nevertheless, there seem to be remarkable connections between these two objects via \emph{T-duality},
as was first discovered in \cite{LPW}.
In this paper we use ideas from oriented matroid theory, total positivity, 
and the geometry of the hypersimplex and positroid polytopes
to obtain a deeper understanding of the amplituhedron.
We show that the inequalities cutting out \emph{positroid polytopes} --
images of positroid cells of $\Gr^{\geq 0}_{k+1,n}$ under the moment map --
translate into sign conditions characterizing the T-dual \emph{Grasstopes} -- images
of positroid cells of $\Gr^{\geq 0}_{k,n}$ under 
$\tilde{Z}$.
Moreover, we subdivide the amplituhedron into \emph{chambers}, just as the hypersimplex can be subdivided into simplices - with both chambers and 
simplices enumerated by the Eulerian numbers.
We use these properties to prove the main conjecture of \cite{LPW}: 
a collection of positroid polytopes is a tiling of the hypersimplex if and only if the collection of T-dual Grasstopes is a tiling of the amplituhedron $\mathcal{A}_{n,k,2}(Z)$ for all $Z$.
Moreover, we prove 
Arkani-Hamed--Thomas--Trnka's conjectural sign-flip characterization of 
$\mathcal{A}_{n,k,2}$,
and {\L}ukowski--Parisi--Spradlin--Volovich's conjectures on 
\emph{$m=2$ cluster adjacency} and on \emph{positroid tiles} for $\mathcal{A}_{n,k,2}$ 
(images of $2k$-dimensional positroid cells which map injectively into $\mathcal{A}_{n,k,2}$). 
Finally, we introduce new cluster structures in the amplituhedron.

\end{abstract}

\maketitle
\setcounter{tocdepth}{1}
\tableofcontents


\section{Introduction}\label{sec_intro}

The \emph{positive Grassmannian}\footnote{more formally, 
the \emph{totally nonnegative
Grassmannian}} $\Grk$ 
is the subset of the real Grassmannian $\Gr_{k,n}$ 
where all Pl\"ucker coordinates
are nonnegative \cite{postnikov, rietsch, lusztig}.  
This is a remarkable space 
with connections to cluster algebras, integrable systems,
and high energy physics \cite{ca1, Scott, KodamaWilliams, abcgpt},
and it has a beautiful CW decomposition into \emph{positroid cells} $S_{\pi}$,
which are indexed by various 
combinatorial objects including 
\emph{decorated permutations}
$\pi$ \cite{postnikov}.

There are several interesting maps which one can apply 
to the positive Grassmannian $\Grk$
and its cells.  The first map is the \emph{moment map} $\mu$, 
initially studied by 
Gelfand-Goresky-MacPherson-Serganova \cite{GGMS} in the context of the Grassmannian
and its torus orbits, who showed that the image of the Grassmannian is the
\emph{hypersimplex} $\Delta_{k,n} \subset \R^{n}$, a polytope of dimension
$n-1$.
When one restricts $\mu$ to $\Grk$,
 the image is still the hypersimplex \cite{tsukerman_williams}.


The second map is the \emph{amplituhedron map}, introduced by 
Arkani-Hamed and Trnka 
\cite{arkani-hamed_trnka}
in the context of 
\emph{scattering amplitudes} in $\mathcal{N}=4$ SYM.  In particular,
any $n \times (k+m)$ matrix $Z$ with maximal minors positive induces
a map $\tilde{Z}$ from $\Grk$ to the Grassmannian $\Gr_{k,k+m}$, whose
image has full dimension $mk$ and is called the \emph{amplituhedron}
$\mathcal{A}_{n,k,m}(Z)$. 

Given any surjective map $\phi: \Grk \to X$ where $\dim X=d$,
it is natural to try to decompose $X$ using images
of positroid cells under $\phi$.   
This leads to the following definition.\footnote{There are many reasonable
variations of this definition.  One might want to relax the injectivity
assumption, or to impose further restrictions on how boundaries of the 
images of cells should overlap. Note that in the literature, positroid tilings are sometimes
called \emph{(positroid) triangulations}. We avoid this terminology in order to avoid confusion with the notion of e.g. polytopal triangulations.}
\begin{definition}\label{def:tri}
	Let $\phi: \Grk \to X$ be 
 a continuous surjective map  where $\dim X=d$.
A \emph{positroid tiling} of $X$ (with respect to $\phi$) 
	is a collection $\{\overline{\phi(S_{\pi})}\}$  of images of 
$d$-dimensional positroid cells such that 
\begin{itemize}
	\item $\phi$ is injective on each $S_{\pi}$ from the collection
	\item pairs of distinct images $\phi(S_{\pi})$
		and $\phi(S_{\pi'})$ are disjoint
              \item $\cup \overline{\phi(S_{\pi})} = X$.
	\end{itemize}
\end{definition}

When $\phi$ is the moment map, the 
(closures of) the
images of the positroid cells $S_{\pi}$ 
are the \emph{positroid polytopes} $\Gamma_{\pi}$
\cite{tsukerman_williams}, so a positroid tiling of the hypersimplex
is a decomposition into positroid polytopes.
When $\phi$ is the amplituhedron map $\tilde{Z}$,
the (closures of) the images 
of the positroid cells $S_{\pi}$
are \emph{Grasstopes} $Z_{\pi}$, 
which were first studied in 
\cite{arkani-hamed_trnka}
as the building blocks of conjectural positroid tilings of the amplituhedron.
Note that neither the amplituhedron nor the Grasstopes are polytopes.

At first glance, the $(n-1)$-dimensional hypersimplex $\Delta_{k+1,n} \subset \R^n$ doesn't seem to have any relation to the $2k$-dimensional
amplituhedron $\mathcal{A}_{n,k,2}(Z) \subset \Gr_{k,k+2}$.  Nevertheless,
the recent paper \cite{LPW} showed that there are 
 surprising parallels between them.  
In particular, they showed that \emph{T-duality} gives a bijection between
 loopless cells $S_{\pi}$ of $\Gr_{k+1,n}^{\geq 0}$ and coloopless cells 
$S_{\hat{\pi}}$ of 
$\Grk$, and conjectured that T-duality gives a bijection between
positroid 
 tilings $\{\Gamma_{\pi}\}$ 
 of the hypersimplex $\Delta_{k+1,n}$ 
 and positroid tilings $\{Z_{\hat{\pi}}\}$ of the
amplituhedron $\mathcal{A}_{n,k,2}(Z)$.  \cite{LPW} proved this conjecture
for infinitely many tilings --  specifically, 
the positroid tilings of $\mathcal{A}_{n,k,2}(Z)$ obtained from 
a BCFW-like recurrence \cite{BaoHe}, and the positroid tilings 
of $\Delta_{k+1,n}$ obtained from an analogous recurrence.

In this paper we use \emph{twistor coordinates}
and the geometry of the hypersimplex and positroid
polytopes 
to obtain a deeper understanding of the amplituhedron.
We prove the conjecture of {\L}ukowski--Parisi--Spradlin--Volovich 
\cite{Lukowski:2019sxw} 
 classifying   \emph{positroid tiles},
full-dimensional images of positroid cells which map injectively into the amplituhedron $\mathcal{A}_{n,k,2}(Z)$. We then give a new characterization of them in terms of the signs of their twistor coordinates.
We use this result to prove a conjecture  of Arkani-Hamed--Thomas--Trnka that
$\mathcal{A}_{n,k,2}(Z)$ can be characterized using sign flips of twistor coordinates.
And we prove two results relating the amplituhedron to cluster algebras.
First, we prove the \emph{cluster adjacency} conjecture
\cite{Lukowski:2019sxw} for $\mathcal{A}_{n,k,2}(Z)$,
which says that the 
Pl\"ucker coordinates labeling facets of a given positroid tile
consist of pairwise compatible cluster variables. We also state and prove a generalization of this conjecture by showing that twistor coordinates of a positroid tile associated to Pl\"ucker coordinates compatible with the ones labelling its facets have constant sign. 
Second,
we associate a cluster variety to each
positroid tile in $\mathcal{A}_{n,k,2}(Z) \subset \Gr_{k, k+2}$,
and show that the positroid tile
is the totally positive part of that cluster variety.
We then have the strange phenomenon that 
the $2k$-dimensional amplituhedron
$\mathcal{A}_{n,k,2}(Z)$ can be decomposed into ${n-2 \choose k}$
$2k$-dimensional positroid tiles,
each of which is the totally positive part of a cluster variety.
(Moreover, there are many such decompositions.)

Additionally, we draw striking parallels between
$\Delta_{k+1,n}$ and $\mathcal{A}_{n,k,2}(Z)$, some of which are 
illustrated in 
\cref{tab:PPer}.
We find that the inequalities describing positroid polytopes
translate into sign conditions on {twistor coordinates} characterizing
the corresponding Grasstopes.
And we show that the sign patterns on twistor coordinates
naturally subdivide the amplituhedron into \emph{chambers}. We prove that the ones which are \emph{realizable} are exactly enumerated by the Eulerian numbers $E_{k,n-1}$, just as the hypersimplex can be subdivided
into simplices enumerated by $E_{k,n-1}$.
We use these properties to prove the main conjecture of \cite{LPW}:
a collection of positroid polytopes is a positroid tiling of 
$\Delta_{k+1,n}$ if and only if
the collection of T-dual Grasstopes is a positroid tiling of 
$\mathcal{A}_{n,k,2}(Z)$ for all $Z$.

Let us now explain how the various geometric objects in our story are related to scattering amplitudes in quantum field theories.
In 2009, Grassmannian formulations were introduced to describe scattering amplitudes in planar $\mathcal{N}=4$ super Yang-Mills \cite{ArkaniHamed:2009dn,Bullimore:2009cb}. Remarkably, this led to the discovery that in fact the \emph{positive} Grassmannian encodes most of the physical properties of amplitudes \cite{abcgpt}. Building on these developments and on Hodges' ideas \cite{Hodges:2009hk}, Arkani-Hamed and Trnka arrived at the definition of the \emph{amplituhedron} $\mathcal{A}_{n,k,m}(Z)$ \cite{arkani-hamed_trnka} in 2013.

The $m=4$ amplituhedron $\mathcal{A}_{n,k,4}(Z)$ is the object most relevant
to physics: it encodes the geometry of (tree-level) scattering amplitudes in planar $\mathcal{N}=4$ SYM.  
The $m=2$ amplituhedron $\mathcal{A}_{n,k,2}(Z)$, often considered as a toy-model for the $m=4$ case, is also relevant for physics. For example, it governs the geometry of planar $\mathcal{N}=4$ SYM amplitudes at the subleading order in perturbation theory (`one-loop') of some sectors the theory, specifically the `MHV' and `NMHV' sector \cite{Kojima:2020tjf}. Moreover, as discovered in \cite{LPW}, $\mathcal{A}_{n,k,2}(Z)$  is at the center of the `$m=2$' version of an important physical duality called
\emph{T-duality}.

Scattering amplitudes in planar $\mathcal{N}=4$ SYM enjoy a remarkable duality called `Amplitude/Wilson loop duality' \cite{Alday:2008yw}, which was shown to arise from a more fundamental duality in String Theory called `T-duality' \cite{Berkovits:2008ic}. The geometric counterpart of this fact is a conjectural duality between collections of $4k$-dimensional `BCFW' cells of $\Gr^{\geq 0}_{k,n}$ which give positroid tilings of the $m=4$ amplituhedron $\mathcal{A}_{n,k,4}(Z)$, and corresponding collections of $(2n-4)$-dimensional cells of $\Gr^{\geq 0}_{k+2,n}$ which give positroid tilings the \emph{momentum amplituhedron} $\mathcal{M}_{n,k,4}$ \cite{mamp, LPW}. This duality was evocatively called \emph{T-duality} in \cite{LPW} and conjectured to generalize for any (even) $m$. In particular, for $m=2$, the hypersimplex $\Delta_{k+1,n}$ (a polytope) and the $m=2$ amplituhedron $\mathcal{A}_{n,k,2}(Z)$ (not a polytope!) are T-dual. As for all dualities in mathematics and physics, the aim is to learn something new of one side from the other, trading complexity of the former with simplicity of the latter. 
Interestingly, T-duality has recently appeared also in the context of critical varieties \cite{GalCritVar}, relating electric networks and Ising models.

One recent trend in physics is the connection between 
analytic properties of scattering amplitudes and \emph{cluster algebras} \cite{ca1} and the \emph{positive tropical Grassmannian} \cite{troppos}. `Cluster phenomena', which have led to both computational and theoretical advances, are manifest in momentum-twistor space \cite{Golden:2013xva}. \emph{Twistor coordinates} have been crucial to explore the emergence of some cluster phenomena \cite{Drummond:2017ssj,Drummond:2018dfd} from the geometry of the amplituhedron \cite{Lukowski:2019sxw, Gurdogan:2020tip}. Further connections with plabic graphs were established in \cite{Mago:2020kmp,He:2020uhb}. In this paper, we use twistor coordinates to prove (and generalize) the conjecture of {\L}ukowski--Parisi--Spradlin--Volovich \cite{Lukowski:2019sxw} about $m=2$ \emph{cluster adjacency} and probe new cluster structures in the amplituhedron. Twistor coordinates and their sign-flips recently played a major role in providing a geometric origin of the expansion of the one-loop MHV scattering amplitudes in $\mathcal{N}=4$ SYM  in terms of `chiral pentagons' and in probing new geometries, as in the search of the \emph{dual amplituhedron}  \cite{Kojima:2020tjf,Herrmann:2020qlt}. In general, considering regions of the geometry where some twistor coordinates have constant sign leads to a deeper understanding of the geometry itself and new ways to subdivide it. This would provide new representations of scattering amplitudes, not obtainable from standard physics methods.

Our discovery that Eulerian numbers count sign chambers of  the $m=2$ amplituhedron
is intriguing because 
Eulerian numbers have also come up in the context of \emph{scattering equations} \cite{Cachazo:2013iaa}.
Scattering equations connect the singularity structure of scattering amplitudes of $n$-particles to 
that of the boundaries of the moduli space of \emph{Riemann spheres with $n$ punctures}. 
For $\mathcal{N}=4$ SYM, the number of solutions of the `N$^k$MHV' sector of 
the theory is exactly the Eulerian number $E_{k,n-3}$ \cite{Spradlin:2009qr,Cachazo:2013iaa}. Moreover, \cite{Cachazo:2016ror} provided an explicit bijection between such solutions and permutations on $[n-3]$ with $k$ descents. Finally, in the case of certain scalar quantum field theories, the authors of \cite{Cachazo:2019ngv} formulated a generalization of scattering equations. By studying  `arrays of Feynman diagrams', they made connections to the positive tropical Grassmannian, and, by results of \cite{LPW}, to the hypersimplex. It would be fascinating to explore possible relations between (generalized) scattering equations, simplices of the hypersimplex, and chambers of the amplituhedron. 

 We note that some of the ideas used in 
 this paper can be applied to amplituhedra for other $m$, and to the momentum amplituhedron; we will pursue this in a separate paper.

The structure of this paper is as follows. 
In \cref{sec:background} we give background on the positive Grassmannian and the amplituhedron.
In \cref{sec:twistor} we define twistor coordinates for the amplituhedron, and define
the \emph{sign stratification} of $\mathcal{A}_{n,k,m}(Z)$, which is analogous to the oriented
matroid stratification of the Grassmannian.
In \cref{secSigns} we study positroid tiles of $\mathcal{A}_{n,k,2}(Z)$:
we prove 
a conjecture of 
{\L}ukowski--Parisi--Spradlin--Volovich characterizing tiles in terms of
\emph{bicolored subdivisions} of a polygon, and we give an inequality description of tiles in terms of signs of twistor coordinates.
In \cref{sec:equivalence} we prove Arkani-Hamed--Thomas--Trnka's conjectural description of $\mathcal{A}_{n,k,2}(Z)$ in terms 
of sign flips of twistor coordinates.
In \cref{sec:cluster} we discuss cluster structures in the amplituhedron. In particular, we introduce a generalization of the {\L}ukowski--Parisi--Spradlin--Volovich \emph{$m=2$ cluster adjacency conjecture} and define a cluster varieties for each positroid tile.
In \cref{sec:background2} we give background on the hypersimplex, T-duality, and positroid tilings
of the hypersimplex.
In \cref{sec:Tduality1}, we describe T-duality as a map on decorated
permutations and plabic graphs. 
In \cref{sec:Tduality2} we discuss the close parallel between the inequality descriptions and facets of positroid tiles in $\Delta_{k+1, n}$ and the T-dual positroid tiles in $\AA_{n, k, 2}$.
We also prove 
(a generalization of) the \emph{$m=2$ cluster adjacency conjecture} for positroid tiles.
In \cref{sec:wsimplices} we show how the 
subdivision of 
$\Delta_{k+1,n}$ into $w$-simplices
corresponds to the decomposition of 
 $\mathcal{A}_{n,k,2}(Z)$ into $w$-chambers, 
 where in both cases $w$ ranges over a set of permutations enumerated
 by the Eulerian number.
In \cref{sec:wsimplices2} we use this correspondence to prove the main conjecture of 
\cite{LPW} about positroid tilings.  We also present algorithms to find tilings of $\Delta_{k+1,n}$ and $\mathcal{A}_{n,k,2}(Z)$ based on $w$-simplices and $w$-chambers, and we show some examples. We also explain other combinatorial manifestations of this correspondence:
decomposing the hypersimplex into positroid polytopes based on descent positions corresponds
to the sign-flip (or ``kermit'') tiling of the amplituhedron. Moreover, we provide an amplituhedron-analogue of slicing the hypercube into hypersimplices.
In \cref{sec:appendixB} we prove that positroid tiles are enumerated by a refinement of Schr\"oder numbers via a bijection with \emph{separable} permutations. \cref{sec:appendix} gives background on the combinatorics 
of the positroid cell decomposition of $\Grk$.

\textsc{Acknowledgements:}
M.P. would like to thank Fatemeh Mohammadi, Leonid Monin, and Lionel Mason for useful discussions. 
M.P. and L.W. would like to thank Tomek {\L}ukowski for the previous paper \cite{LPW} which inspired some of this work.  The authors would like to thank
Pasha Galashin, Tomek {\L}ukowski, and David Speyer for helpful comments on the 
first draft of the paper, and an anonymous referee for many useful comments.
M.P. would like to acknowledge the generous support of the Sachs Scholarship at Princeton and the ERC grant number 724638.
M.S.B. would like to acknowledge the support of a National Science Foundation Graduate Research Fellowship, grant number DGE-1752814.
L.W. would  like to acknowledge the support of the National Science Foundation
under agreements No.\ DMS-1854316 and No.\ DMS-1854512, as well as the support of the 
Radcliffe Institute for Advanced Study at Harvard University,
 where some of this work ``took place'' (virtually).  Any opinions,
findings and conclusions or recommendations expressed in this material
are those of the authors and do not necessarily reflect the
views of the National Science Foundation.

\begin{landscape}
\thispagestyle{empty}
\begin{table}
\begin{adjustwidth}{+0.4in}{-1in}
	\caption{Table of correspondences via T-duality: the hypersimplex vs the amplituhedron.} 
\centering
\begin{tabular}{l c r} 
\hline\hline 
\rule{0pt}{3ex}  \large The hypersimplex $\Delta_{k+1,n}$ & \large VS & \large The amplituhedron $\AA_{n,k,2}$ \\ [1ex]
\hline 
\rule{0pt}{3ex} $\Delta_{k+1,n}=\mu(\Gr_{k+1,n}^{\geq 0})$ (moment map) & {} & (amplituhedron map) $\AA_{n,k,2}=\tilde{Z}(\Gr_{k,n}^{\geq 0})$\\[1ex]
\rule{0pt}{3ex} $\mbox{dim}(\Delta_{k+1,n})=n-1$ in $\mathbb{R}^n$  & {} &  $\mbox{dim}(\AA_{n,k,2})=2k$ in $\Gr_{k,k+2}$\\[1ex]
\hline
\rule{0pt}{2ex} {}& POSITROID TILES (PT)& {}\\
\hline
\rule{0pt}{3ex} $\Gamma_{G(\T)}$ (positroid polytope) & bicolored subdivision $\T$ of type $(k,n)$ & (Grasstope) $Z_{\hat{G}(\T)}$\\[1ex]
   \rule{0pt}{3ex} $x_{[h,j-1]} \geq \area_{\T}(h \rightarrow j)$  & compatible arc $h \rightarrow j$ & $\sgn \langle Y hj\rangle=(-1)^{\area_{\T}(h \rightarrow j)}$\\[1ex]
   \rule{0pt}{3ex} $x_{[h,j-1]} = \area_{\T}(h \rightarrow j)$  & facet defining arc $h \rightarrow j$ &  $\langle Y hj\rangle=0$\\[1ex]
\hline
\rule{0pt}{2ex} {}&$w$-SIMPLICES and $w$-CHAMBERS& {}\\
\hline
	\rule{0pt}{3ex} $w$-simplex $\Delta_w \subset \Delta_{k+1,n}$ such that: &  $w \in S_n$: $w_n=n, \# \des(w)=k$ & $w$-chamber $\hat{\Delta}_w(Z)\subset \AA_{n,k,2}$ such that:  \\[1ex]
 \, \, $\Delta_w = \conv\{e_{I_1},\ldots,e_{I_n}\}$ & 
	$I_a=I_a(w):=\cdes(w^{(a-1)})$ & \, $ \flip(
        \langle Y a \hat{1} \rangle , \langle Y a \hat{2} \rangle, \dots, \langle Y a n \rangle)=I_a \setminus \{a\} $ \,  \\[1ex]
   \rule{0pt}{3ex}     Hypersimplex $\Delta_{k+1,n}= \bigcup_w {\Delta_w}$ & {} &  Amplituhedron $\AA_{n, k, 2}(Z)= \bigcup_w {\asimp{w}}$\\[1ex]
\rule{0pt}{3ex}	PT $\Gamma_{\pi}= \bigcup_{\Delta_{w} \cap \Gamma^{\circ}_{\pi} \neq \emptyset} {\Delta_{w}}$ & & PT $\gt{\hat{\pi}}= \bigcup_{\asimp{w} \cap \gto{\hat{\pi}} \neq \emptyset} {\asimp{w}}$\\[1ex]
\rule{0pt}{3ex} $\simp{w} \subset \Gamma_{\pi}$ & $\Leftrightarrow$ &$\asimp{w} \subset \gt{\td{\pi}}$\\[1ex]
\rule{0pt}{3ex} Hypercube $\mbox{\mancube}_{n-1}=\bigcup_k \Delta_{k+1,n}=\bigcup_w \Delta_w$ & $w \in S_n: w_n=n$ & Total amplituhedron $\mathcal{G}_n^{(2)}=\bigcup_k \mathcal{G}_{n,k,2}=\bigcup_w \hat{\Delta}_w(\mathcal{G})$ \\[1ex]
 \hline
\rule{0pt}{2ex} {}&POSITROID TILINGS& {}\\
\hline
\rule{0pt}{3ex}  $\{\Gamma_{\pi}\}$ tiles $\Delta_{k+1, n}$& $\Leftrightarrow$ & $\{\gt{\td{\pi}}\}$ tiles $\AA_{n, k, 2}(Z),$ for all $Z$\\[1ex]
   \rule{0pt}{3ex}  BCFW tiling $\{\Gamma_\pi\}$ & BCFW recurrences & BCFW tiling $\{\gt{\td{\pi}}\}$  \\[1ex]
     \rule{0pt}{3ex}   Regular tiling $\{\Gamma_\pi\}$ & max cones of $\mbox{Trop}^+\Gr_{k+1,n}$ & Regular tiling $\{\gt{\td{\pi}}\}$  \\[1ex]
   \rule{0pt}{3ex}  Catalan tiling $\{\Gamma_\pi\}$ & Coloring vertices of a fixed planar 
	tree & Catalan tiling $\{\gt{\td{\pi}}\}$  \\[1ex]
    \rule{0pt}{3ex}  Descent tiling $\{\Gamma_\pi\}$ & positions of descents/sign flips & Sign-flip tiling $\{\gt{\td{\pi}}\}$  \\[1ex]
\hline 

\end{tabular}
\label{tab:PPer}
\end{adjustwidth}
\end{table}
\end{landscape}

\section{The positive Grassmannian and the amplituhedron }\label{sec:background}

\subsection{The Grassmannian and positive Grassmannian}

\noindent The {\itshape (real) Grassmannian} $\Gr_{k,n}$ is the space of all
$k$-dimensional subspaces of $\R^n$, for $0\le k \le n$.  An element of
$\Gr_{k,n}$ can be viewed as a $k\times n$ matrix of rank $k$, modulo left
multiplication by invertible $k\times k$ matrices. That is, two
$k\times n$ matrices of rank $k$ represent the same point in $\Gr_{k,n}$ if and only if they
can be obtained from each other by invertible row operations. For $C$ a full-rank $k \times n$ matrix, we will often abuse notation and write $C \in \Gr_{k, n}$, identifying $C$ with its rowspan.

Let $[n]$ denote $\{1,\dots,n\}$, and $\binom{[n]}{k}$ the set of all $k$-element subsets of $[n]$. We embed $\Gr_{k,n}$ into projective space $\PP(\wedge^k \RR^n)$ in the usual way. That is, choose $V \in \Gr_{k,n}$ and any representative matrix $C$ with rows $C_1, \dots, C_k$. We map $V$ to the equivalence class of $C_1 \wedge \dots \wedge C_k$ in $\PP(\wedge^k \RR^n)$. This equivalence class depends only on $V$, not on the choice of $C$.

%
The embedding $V \mapsto C_1 \wedge \dots \wedge C_k$ gives a natural choice of coordinates for the Grassmannian.
Let $\{e_1,\dots,e_n\}$ be the standard basis of $\R^n$, and for 
$I= \{i_1 < i_2 < \dots < i_k\} \subset {[n] \choose k}$, let 
$E_I:= e_{i_1} \wedge \dots \wedge e_{i_k}.$ 
Writing $C_1\wedge \dots\wedge C_k$ in terms of the $E_I$, we obtain
\begin{equation}
	C_1 \wedge \dots \wedge C_k = \sum_{I\in {[n] \choose k}} p_I(V) E_I \in \wedge^k(\R^n),
\end{equation}
where $p_I(V)$ is the maximal minor of $C$ located in column set $I$. The $p_I(V)$ are the {\itshape Pl\"{u}cker coordinates} of $V$, and are independent of $C$ (up to simultaneous rescaling by a  constant).

We will also use the notation $\lrangle{C_1, \dots, C_k}$ for $C_1 \wedge \dots \wedge C_k$.

\begin{defn}[{\cite[Section~3]{postnikov}}]\label{def:positroid}
We say that $C\in\Gr_{k,n}$ is {\itshape totally nonnegative} if $p_I(C)\ge 0$ for all $I\in\binom{[n]}{k}$, and {\itshape totally positive} if $p_I(C) > 0$ for all $I\in\binom{[n]}{k}$. The set of all totally nonnegative $C\in\Gr_{k,n}$ is the {\it totally nonnegative Grassmannian} 
	$\Gr_{k,n}^{\ge 0}$, 
	and the set of all totally positive $C$ is the {\itshape totally positive Grassmannian} $\Gr_{k,n}^{>0}$. For $\M\subseteq \binom{[n]}{k}$,
the {\it positroid cell} $S_{\M}$ is
the set of $C\in\Gr_{k,n}^{\geq 0}$ such that 
	$p_I(C)>0$ for all $I\in \M$, and 
	$p_J(C)=0$ for all $J\in\binom{[n]}{k}\setminus \M$. We call $\M$ a \emph{positroid} if $S_\M$ is nonempty. We let $Q(k,n)$ denote the poset on the cells of $\Gr_{k,n}^{\ge 0}$ defined by
$S_{\M} \leq S_{\M'}$ if and only if\footnote{Here, and in what follows, 
	we use closure in the Hausdorff topology.}
$S_\M \subseteq \overline{S_{\M'}}$.
\end{defn}

\begin{remark}\label{rem:PL}
The positive and nonnegative part of a flag variety $G/P$ was 
first introduced by 
	Lusztig \cite{lusztig} (who gave a Lie-theoretic definition
	of $(G/P)_{>0}$ and defined $(G/P)_{\geq 0} :=\overline{(G/P)_{>0}}$),
and proved to have a cell decomposition by Rietsch \cite{rietsch}.  
Postnikov \cite{postnikov} subsequently defined the nonnegative part of the 
Grassmannian as in \cref{def:positroid}, and independently gave the above
	decomposition into cells.
From the beginning it was believed by experts that 
Postnikov's definition of $\Grk$ should agree
with Lusztig's (in the case $G/P$ is the Grassmannian); this was first 
	proved by Rietsch \cite{rietsch_private}, and reproved  
 in \cite[Corollary 1.2]{talaska_williams}, where the authors additionally proved
that the two cell decompositions
coincide.  Two subsequent proofs 
that the two definitions of $\Grk$ coincide were given in 
 \cite{lam, Lusztig3}.
\end{remark}


There are many ways to index the positroid cells of $\Grk$
 \cite{postnikov}, including
\emph{decorated permutations} $\pi$, 
 \emph{affine permutations} $f$,
and \emph{plabic graphs} $G$.  We will refer to the corresponding 
positroid cells using the notation $S_{\pi}$, $S_{f}$, $S_G$.
For background,
 see \cref{sec:appendix}.

\subsection{The amplituhedron}

Building on \cite{abcgpt},
Arkani-Hamed and Trnka \cite{arkani-hamed_trnka} introduced a  new
mathematical object called the \emph{(tree) amplituhedron}, which
is the image of the totally nonnegative Grassmannian under a particular map.
In what follows, we let $\Mat_{n,p}^{>0}$ denote the set of $n\times p$ matrices whose maximal minors
are positive.

\begin{defn}\label{defn_amplituhedron}
Choose positive integers $k<n$ and $m$ such that $k+m \leq n$, and 
	let $Z\in \Mat_{n,k+m}^{>0}$.
Then $Z$ induces a map
$\tilde{Z}:\Gr_{k,n}^{\ge 0} \to \Gr_{k,k+m}$ 
	defined by
$$\tilde{Z}(\langle c_1,\dots, c_k \rangle) := \langle Z(c_1),\dots, Z(c_k) \rangle.$$
Equivalently, if $C$ is a matrix representing an element of 
$\Gr_{k,n}^{\ge 0}$, then $\tilde{Z}(C)$ is defined to be 
	the element of $\Gr_{k,k+m}$
	represented by 
the matrix $CZ$.  
The \emph{(tree) amplituhedron} $\mathcal{A}_{n,k,m}(Z)$ is defined to be the image
$\tilde{Z}(\Gr_{k,n}^{\ge 0})$ inside $\Gr_{k,k+m}$.
\end{defn}

The fact that $Z$ has positive maximal minors ensures that $\tilde{Z}$
is well defined
\cite{arkani-hamed_trnka}. 
See \cite[Theorem 4.2]{karp} for a necessary and sufficient condition 
 (in terms of sign-variation)
for a matrix $Z$ 
to give rise to a well-defined map $\tilde{Z}$.
The amplituhedron $\mathcal{A}_{n,k,m}(Z)$  has full dimension $km$ inside $\Gr_{k,k+m}$.

In special cases the amplituhedron recovers familiar objects. If $Z$ is a square matrix, i.e.\
$k+m=n$, then $\mathcal{A}_{n,k,m}(Z)$ is isomorphic to
the totally nonnegative Grassmannian. If $k=1$,  $\mathcal{A}_{n,1,m}(Z)$ is a {\itshape cyclic polytope} in projective space $\PP^m$ \cite{Sturmfels}.
If $m=1$, then $\mathcal{A}_{n,k,1}(Z)$ can be identified with the 
complex of bounded faces of a cyclic hyperplane arrangement 
\cite{karpwilliams}.

We will consider the restriction of the $\tilde{Z}$-map to positroid 
cells in 
$\Gr_{k,n}^{\ge 0}$.

\begin{definition}
Fix $k, n, m$ with $k+m \leq n$ and choose
	$Z\in \Mat_{n,k+m}^{>0}$.  
Given a positroid cell $S_{\pi}$ of 
$\Gr_{k,n}^{\ge 0}$, we let 
$\gto{\pi} = \tilde{Z}(S_{\pi})$ and 
$\gt{\pi} = \overline{
	\tilde{Z}(S_{\pi})} = \tilde{Z}(\overline{S_{\pi}})$,
	and we refer to 
$\gto{\pi}$ and $\gt{\pi}$ as \emph{open Grasstopes} and 
\emph{Grasstopes}, respectively.
	We call
	$\gt{\pi}$  and $\gto{\pi}$
	a \emph{positroid tile}  and an \emph{open positroid tile}
	for $\mathcal{A}_{n,k,m}(Z)$ 
	if $\dim(S_\pi) =km$ and 
	 $\tilde{Z}$ is injective on $S_\pi$.
\end{definition}

\begin{defn}\label{def:facetG}
        Let $\gt{\pi}$ be a Grasstope of 
        $\mathcal{A}_{n,k,m}(Z)$. We say that $\gt{\pi'}$ is a \emph{facet} of $\gt{\pi}$ if it is maximal by inclusion among the Grasstopes satisfying the following three properties:
        \begin{itemize}
                \item the cell $S_{\pi'}$ is contained in $\overline{S_\pi}$
                \item $\gt{\pi'}$ is contained in the boundary $\partial \gt{\pi}$
                \item $\gt{\pi'}$ has codimension 1 in $\gt{\pi}$.
        \end{itemize}
\end{defn}

\begin{rmk}\label{rmk:closure} By \cite[Proposition 15.2]{lam},  $\tilde{Z}(\overline{S_\pi})=\overline{
\tilde{Z}(S_{\pi})}$.
\end{rmk}

If $k=1$ and $m=2$, the amplituhedron $\mathcal{A}_{n,1,2}(Z)$ is a convex $n$-gon in $\PP^2$. The positroid tiles are exactly the
  triangles on vertices of 
 the polygon.

Images of positroid cells under the map $\tilde{Z}$ 
have been studied since the introduction of the amplituhedron.  In particular,
Arkani-Hamed and Trnka \cite{arkani-hamed_trnka} conjectured that the images
of certain \emph{BCFW} collections of $4k$-dimensional cells in 
$\Gr_{k,n}^{\geq 0}$ give a {positroid tiling} of the amplituhedron 
$\mathcal{A}_{n,k,4}(Z)$.
Positroid tiles were called \emph{generalized triangles} in \cite{Lukowski:2019sxw}.
The terminology of Grassmann polytopes to describe images of positroid cells
in the amplituhedron was used in \cite{lam}.
For brevity, we prefer the 
term \emph{Grasstopes}.

\begin{remark}
	While the definition of the amplituhedron $\A(Z)$ depends on a choice of $Z \in \Mat_{n,k+m}^{>0}$,
	it is believed that many of its
	combinatorial properties do not depend on this choice.
	For example, whether or not 
 $\overline{\tilde{Z}(S_{\pi})}$ is a positroid tile should 
	be independent of the choice of $Z$; 
	 we will see that this 
	is true in \cref{thm:allGTs} in the case that $m=2$.
	It is also believed that whether or not 
	a collection of cells in $\Grk$ gives a positroid tiling 
	of 
	$\A(Z)$ should be independent of $Z$. 
\end{remark}

\begin{remark}\label{rem:twisted}
We note that matrices whose maximal minors are positive (or nonnegative) 
have a \emph{twisted cyclic symmetry}.
	If $Z\in \Mat^{>0}_{n,p}$ with $n \geq p$ has rows $Z_1,Z_2, \dots, Z_n$, and if we let $\hat{Z}_i$ denote
$(-1)^{p-1} Z_i$, then the matrix with rows 
	$Z_2,\dots, Z_n, \hat{Z}_1$ also lies in $\Mat^{>0}_{n,p}$.  
Similarly for the matrix with rows $Z_3,\dots, Z_n, \hat{Z}_1, \hat{Z}_2$, etc\footnote{We will use the `hat' notation $\hat{}$ also in the context of T-duality with a different meaning. It will be always clear from context which one we mean.}.
\end{remark}

\subsection{Previous work on the $m=2$ amplituhedron}
The original paper 
\cite{arkani-hamed_trnka}
gave a conjectural
positroid tiling $\{Z_{\pi}\}$ of $\mathcal{A}_{n,k,2}(Z)$.
\cite{Karp:2017ouj}  
proved that the above collection 
consists of positroid tiles, that is, $\tilde{Z}$ is injective
on the corresponding positroid cells.
A BCFW-style recursion for positroid tilings of 
$\mathcal{A}_{n,k,2}(Z)$ was also conjectured in \cite{Karp:2017ouj};
the fact that this recursion indeed produces positroid tilings
was proved in \cite{BaoHe}.  
A conjectural classification of $m=2$ positroid tiles was 
given in \cite{Lukowski:2019sxw}.

Meanwhile, \cite{ATT} gave a conjectural 
alternative description
of $\mathcal{A}_{n,k,2}(Z)$ in terms of sign flips of twistor coordinates;
they gave a proof sketch of one direction of the conjecture,
and an independent proof of the same direction was given in 
\cite{karpwilliams}.
In a different direction, \cite{Lukowski:2019kqi}
gave a conjectural description of the boundaries of the $m=2$
amplituhedron.  
Finally, \cite{LPW} discovered a link between the $m=2$ amplituhedron
and the hypersimplex via T-duality and 
the \emph{tropical positive Grassmannian}, which inspired
the present paper.



\section{The sign stratification of the amplituhedron}\label{sec:twistor}

In this section we introduce \emph{twistor coordinates} for the amplituhedron
$\mathcal{A}_{n,k,m}(Z)$, and we use them to 
define the \emph{sign stratification} of the amplituhedron.   We also 
introduce terminology for sign variation and sign flips.  We will subsequently
use twistor coordinates to prove a sign flip description of $\AA_{n, k, 2}$ in \cref{thm:main1}, to characterize positroid tiles, and to 
describe Grasstopes.  

The definitions and results in this section
hold for any positive $m$.  The subsequent sections of the paper are mostly concerned with $m=2$.  However, many of our techniques
 can be applied to other $m$, in particular $m=4$; 
we plan to investigate this in a separate paper.

Twistor coordinates were first 
considered in \cite{arkani-hamed_trnka}, and subsequently used in 
\cite{ATT} to give a conjectural
``sign flip'' description of the amplituhedron.
In the case $m=1$,  \cite[Corollary 3.19]{karpwilliams}
studied the sign stratification and 
proved a sign flip description 
of $\mathcal{A}_{n,k,1}(Z)$.


\subsection{Twistor coordinates for $\mathcal{A}_{n,k,m}$}

\begin{definition}\label{def:twistor}
	Fix positive $k<n$ and $m$ such that $k+m \leq n$.
	Choose
	$Z \in \Mat_{n,k+m}^{>0}$ and denote its rows by $Z_1,\dots, Z_n \in \R^{k+m}$.
	Given a matrix $Y$ with rows $y_1,\dots,y_k$ representing an element of 
	$Gr_{k,k+m}$, 
	and $i_1,\dots, i_m$  a sequence of elements of $[n]$,
	we let 
	$$\langle Y Z_{i_1} Z_{i_2} \dots Z_{i_m} \rangle=
	\langle y_1, \dots,y_k, Z_{i_1}, \dots, Z_{i_m} \rangle$$ 
	denote the determinant of the 
	$(k+m) \times (k+m)$ matrix whose rows are 
	$y_1, \dots,y_k, Z_{i_1}, \dots, Z_{i_m}$. 
	We  call
	$\langle Y Z_{i_1} Z_{i_2} \dots Z_{i_m} \rangle$ a \emph{twistor coordinate}.
	We abbreviate 
	$\langle Y Z_{i_1} Z_{i_2} \dots Z_{i_m} \rangle$ by writing 
	$\langle Y {i_1} {i_2} \dots {i_m} \rangle$, 
	when $Z$ is understood.
\end{definition}


Note that the twistor coordinates are a subset of the Pl\"ucker coordinates 
of the $(k+m)\times (k+n)$ matrix whose columns are $y_1,\dots, y_k, Z_1,\dots, Z_n$.
There is also an interpretation of the twistor coordinates as Pl\"ucker 
coordinates in $\Gr_{m,n}$, as we explain in \cref{prop:twistor}. 
In the context of scattering amplitudes of $n$ particles in SYM theory, $\Gr_{m,n}$ is the space of \emph{momentum twistors}\footnote{Momentum twistors, introduced by Hodges in \cite{Hodges:2009hk}, are points $z_1,\ldots,z_n$ in $\mathbb{P}^3$ encoding the kinematic data of scattering particles. Due to \emph{dual conformal symmetry} of scattering amplitudes in SYM theory, these are defined up to a $PGL_4$ transformation on $\mathbb{P}^3$. Therefore, momentum twistors can be embedded in $Gr_{4,n}/(\mathbb{C}^*)^{n-1}$ and scattering amplitudes are functions of Pl\"ucker 
coordinates in $\Gr_{4,n}$. See \cite{Golden:2013xva}.} for $m=4$, which is why we call the coordinates from \cref{def:twistor} \emph{twistor coordinates}. Remarkable connections between scattering amplitudes and the cluster algebra associated to the Grassmannian $Gr_{4,n}$ were discovered in these coordinates \cite{Golden:2013xva}. 

The fact that the twistor coordinates uniquely determine points of the amplituhedron can be deduced from some
results of \cite{karpwilliams}. 

\begin{definition} \label{def:B} \cite[Definition 3.8]{karpwilliams}.
	Given $W\in \Gr_{k+m,n}^{>0}$, we define the \emph{$\mathcal{B}$-amplituhedron}
	$$\mathcal{B}_{n,k,m}(W):= \{V^{\perp} \cap W \ \vert \ V\in \Gr_{k,n}^{\geq 0}\} \subseteq \Gr_m(W),$$
	where $V^{\perp} \in \Gr_{n-k,n}$ denotes 
	the orthogonal complement of $V$ in $\R^n$ and $\Gr_m(W) \subseteq \Gr_{m,n}$ denotes the subset of $\Gr_{m,n}$ of elements
	$X\in \Gr_{m,n}$ with $X\subseteq W$.
\end{definition}

\begin{prop}
	\label{prop:twistor} 
	\cite[Lemma 3.10, Proposition 3.12]{karpwilliams}
	Fix $k, n, m$ and $Z$ as in \cref{def:twistor}, and let $W \in \Gr_{k+m,n}^{>0}$
	be the column span of $Z$.  Then the map 
	\begin{align*}
		f_Z:\Gr_m(W) &\to \Gr_{k,k+m}\\
		X& \mapsto Z(X^{\perp})=\{Z(x) \ \vert \ x\in X^{\perp}\} = 
		\rowspan(X^{\perp} Z)=:Y
	\end{align*}
	is 
	an isomorphism. 
	Here $X^{\perp} \in \Gr_{n-m,n}$ denotes 
	the orthogonal complement of $X$ in $\R^n$.
	
	Moreover, for $X\in \Gr_m(W)$, $Y:=f_Z(X)$, 
and  $I =\{i_1<\cdots < i_m\} \subseteq [n]$, we have 
	\begin{equation} \label{Pluckertwistor}
		p_I(X)=\langle Y Z_{i_1} \dots Z_{i_m} \rangle
	\end{equation} (where we view  Pl\"ucker  and  twistor coordinates as coordinates on points in projective space). 

	Finally, $f_Z: \mathcal{B}_{n,k,m}(W) \to \mathcal{A}_{n,k,m}(Z)$
	is a homeomorphism sending $V^{\perp} \cap W \mapsto \tilde{Z}(V)$.
\end{prop}
From \eqref{Pluckertwistor} we see that $Y\in \Gr_{k,k+m}$ is uniquely determined by its twistor coordinates.

\begin{remark} 
As an alternative to \cref{prop:twistor} 
we can consider the injective map $\psi_Z$
\begin{align*}
	\psi_Z:\Gr_{k, k+m} &\to \Gr_{m,n}\\
	Y& \mapsto Y^{\perp} Z^T=:z
\end{align*}
where $Y^\perp$ is any matrix representing the orthogonal complement of $Y$. 
Then it's not hard to see that 
	for $I =\{i_1<\cdots < i_m\} \subseteq [n]$, $p_I(z)=\langle Y Z_{i_1} \dots Z_{i_m} \rangle$ (viewing both Pl\"ucker and twistor coordinates as coordinates on points in projective space). 
\end{remark}

The following  
 expansion formula \eqref{eq:expand} will be useful in our proofs on positroid tiles. 
\begin{lemma}\label{lem:expansion}
Use the notation of 
\cref{def:twistor}.  
If we write $Y \in \Gr_{k,k+m}$ as $Y=CZ$
with $C\in \Gr_{k,k+n}$, 
we can write the twistor coordinates
in the form 
	\begin{equation}\label{eq:expand}
		\langle CZ, Z_{i_1}, \dots, Z_{i_m} \rangle = 
		\sum_{\{j_1<\dots<j_k\} \in {[n] \choose k}} p_J(C) \langle Z_{j_1}, \dots, Z_{j_k}, Z_{i_1},\dots, Z_{i_m}\rangle.
	\end{equation}
\end{lemma}
	\begin{proof}
	Identifying the $k \times (k+m)$ matrix  $CZ$ with the corresponding 
	element $\langle CZ \rangle$
	of $\wedge^k(\C^{k+m})$, we have
	\begin{equation*} 
		\langle CZ \rangle = \sum_{\{j_1<\dots<j_k\} \in {[n] \choose k}} p_J(C) \langle Z_{j_1}, \dots, Z_{j_k}\rangle.
	\end{equation*} 
This implies the result.
\end{proof}

We will give a description of positroid tiles in $\AA_{n, k, 2}$ using signs of twistor coordinates. One ingredient in our proofs is the following easy sufficient condition for a twistor coordinate to have constant sign on a Grasstope, which follows directly from 
	 \eqref{eq:expand}.

\begin{lemma}\label{lem:useful}
	Fix positive $k<n$ and $m$ such that $k+m \leq n$.
	Let $S_{\mathcal{M}}$ be a cell of $\Grk$.
	Fix $Z\in \Mat^{>0}_{n,k+m}$ and as usual let $Z_1,\dots, Z_n$ denote the row
	vectors of $Z$.  Choose an $m$-element subset 
	$1\leq i_1 < i_2 < \dots < i_m \leq n$.
	\begin{itemize}
		\item If $\langle Z_{j_1},\dots,Z_{j_k}, Z_{i_1}, \dots, Z_{i_m} \rangle \geq 0$ for each $J=\{j_1<\dots<j_k\}\in \mathcal{M}$,
			 then\\	
		$\langle CZ, Z_{i_1},\dots, Z_{i_m} \rangle \geq 0$ for each $C\in 
		S_{\mathcal{M}}$.
		\item 
		If in addition 
			$\langle Z_{j_1},\dots,Z_{j_k}, Z_{i_1}, \dots, Z_{i_m} \rangle > 0$ for some $J=\{j_1<\dots<j_k\}\in\mathcal{M}$
			then 
		$\langle CZ, Z_{i_1},\dots, Z_{i_m} \rangle > 0$ for each $C\in 
		S_{\mathcal{M}}$.
	\end{itemize}
\end{lemma}



\subsection{The sign stratification of $\mathcal{A}_{n,k,m}$}

Since $Y\in \Gr_{k,k+m}$ is uniquely determined by its twistor coordinates, it makes sense
to stratify $\mathcal{A}_{n,k,m}(Z) \subset \Gr_{k,k+m}$ by the signs of the twistor coordinates.
This was done in \cite{karpwilliams} in the case that $m=1$.
Moreover, this sign stratification is closely related to
the \emph{oriented matroid stratification} on the Grassmannian, which partitions 
elements of the real Grassmannian into strata based on the signs of the Pl\"ucker coordinates.
By \cref{prop:twistor}, the twistor coordinates of $Y\in \mathcal{A}_{n,k,m}(Z)$ are 
Pl\"ucker coordinates 
on the corresponding element of the \emph{B-amplituhedron} \cite{karpwilliams} or \emph{amplituhedron in momentum twistor space} \cite{ATT},
so this sign stratification reduces to the oriented matroid stratification 
in momentum twistor space.

\begin{definition}[Amplituhedron chambers]\label{def:chamber}
Fix positive $k<n$ and $m$ such that $k+m \leq n$.
Let $\sigma= (\sigma_{i_1,\dots,i_m})\in \{0, +,-\}^{n \choose m}$ be a nonzero sign vector, considered\footnote{Pl\"ucker and twistor coordinates are 
	defined only up to multiplication by a common scalar.} 
	 \emph{modulo multiplication by 
	$\pm 1$.}  
	Set $$\mathcal{A}^{\sigma}_{n,k,m}(Z):=\{Y\in \mathcal{A}_{n,k,m}(Z) \ \vert \ 
	\sign \lrangle{Y Z_{i_1} \dots Z_{i_m}} = \sigma_{i_1,\dots,i_m}\}.$$
	We call 
	$\mathcal{A}^{\sigma}_{n,k,m}(Z)$ an \emph{(amplituhedron) sign stratum.}
	Clearly $$\mathcal{A}_{n,k,m}(Z) = \sqcup_{\sigma} \mathcal{A}^{\sigma}_{n,k,m}(Z).$$
	If $\sigma\in \{+,-\}^{n \choose m}$,
	we call 
	$\mathcal{A}^{\sigma}_{n,k,m}(Z)$ an open 
	\emph{(amplituhedron) chamber}.\footnote{We borrow the word ``chamber''
	 from the theory of hyperplane arrangements.}
\end{definition}

For $m=1$, all strata are nonempty \cite[Definition 5.2]{karpwilliams},
but this is not true for $m>1$.  Moreover, whether or not 
$\mathcal{A}^{\sigma}_{n,k,m}(Z)$ is empty depends on  $Z$, see \cref{sec:wsimplices2}.
\begin{definition} \label{def:realizable}
We say that a sign vector $\sigma$ (or sign stratum 
$\mathcal{A}^{\sigma}_{n,k,m}$) is \emph{realizable} for $\mathcal{A}_{n,k,m}$
if 
	$\mathcal{A}^{\sigma}_{n,k,m}(Z)$ is nonempty for some $Z$.
\end{definition}

\subsection{Sign variation and sign flips}

Signs and sign flips will be important to our description of the amplituhedron, so we 
introduce some useful terminology here.

\begin{definition}\label{def:var}
	Given $v\in \R^n$, let
	$\var(v)$ be the number of times $v$ changes sign 
	when we read the components from left to right and ignore any
	zeros.  We also define
	$$\overline{\var}(v):=\max\{\var(w) \ \vert \ 
	w\in \R^n \text{ such that }w_i=v_i\text{ for all }
	i\in [n] \text{ with }v_i \neq 0\}.$$
	If $v\in \{0,+,-\}^n$, we define $\var(v)$ and 
	$\overline{\var}(v)$ in the obvious way.
\end{definition}

For example, if $v:=(4,-1,0,-2)\in \R^4$ then $\var(v)=1$
and $\overline{\var}(v)=3$.


\begin{definition}\label{def:flip}
		If $v\in \R^n$, or $v\in \{+,-, 0\}^n$, we say that 
		$v$ has a \emph{sign flip in position $i$} if $v_i, v_{i+1} \neq 0$
		and they have different signs, where indices are considered modulo $n$.
		We define
		$$\flip(v) = \flip(v_1,\dots,v_n):= \{i \ \vert \ v \text{ has a sign flip in position }i\} \subseteq [n].$$
\end{definition}

\begin{remark} We caution the reader that $|\flip(v)|$ may not equal $\var(v)$. For example, the sequence $(+,0,-,0,+,+,-)\in \{+,-, 0\}^7$ has sign flips in positions
$\{6,7\}$ , but $\var(+,0,-,0,+,+,-)$ is 3.
\end{remark}

\section{Positroid tiles 
of $\mathcal{A}_{n,k,2}$}\label{secSigns}

Recall that a \emph{positroid tile} of $\mathcal{A}_{n,k,m}(Z)$
is the full-dimensional image of a positroid cell on which $\tilde{Z}$ is injective.
In this section, we will obtain a detailed description of the 
positroid tiles of $\mathcal{A}_{n,k,2}(Z)$.  The main results of this
section are the following:
\begin{itemize}
	\item In \cref{thm:allGTs} we classify the positroid tiles
of $\mathcal{A}_{n,k,2}(Z)$, describing them 
	as the Grasstopes  $Z_{\hat{G}(\overline{\T})}$ obtained from 
	 the $2k$-dimensional positroid
cells 
		$S_{\hat{G}(\overline{\T})}$ associated to bicolored subdivisions of polygons, 
		proving a conjecture of 
 \cite{Lukowski:2019sxw}.
This implies  that whether or not 
	$\tilde{Z}(S_{\pi})$ is a positroid tile is independent of the choice
of $Z$.
\item  
In \cref{thm:surjectivity} we  characterize
		each (open) positroid tile $Z^\circ_{\hat{G}(\T)}$ 
		as the subset of $\Gr_{k,k+2}$
		where certain twistor coordinates have a fixed sign;
		this shows that each positroid tile is a union
		of (closures of) amplituhedron chambers.
	\item In \cref{prop:niceRep} we  solve a kind of 
		``inverse problem'' for positroid tiles: 
		given an element $Y\in \Gr_{k,k+2}$ which lies
		in an open positroid tile $Z^{\circ}_{\hat{G}(\T)}$, 
		we  explicitly
		construct an element $C\in \Gr_{k,n}$ whose image
		in $\mathcal{A}_{n,k,2}(Z)$ is $Y$, i.e. 
		$CZ = Y$; the entries of $C$ are in fact twistor coordinates.
\end{itemize}
We note that the techniques that we use in this section can be extended to 
give a \emph{cell decomposition} of $\mathcal{A}_{n,k,2}(Z)$.  This will
be explored in a separate paper.

\begin{figure}[h]
\includegraphics[width=0.65\textwidth]{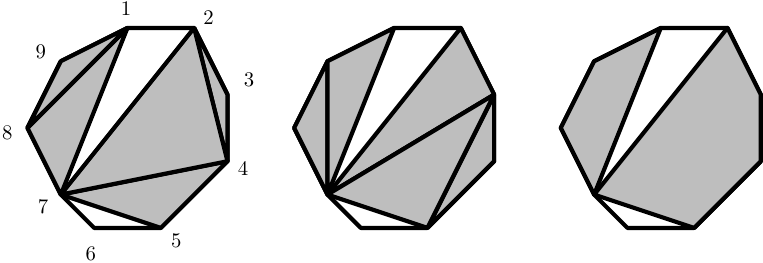}
	\caption{Two equivalent bicolored triangulations $\T_1$ and $\T_2$ of type $(5,9)$
	, and the corresponding 
	bicolored subdivision $\overline{\T}_1 = \overline{\T}_2$ of type $(5,9)$ .}
	\label{fig:unpunctured}
\end{figure}

\begin{definition}[Bicolored triangulations and subdivisions]\label{def:tiling}
Let 
 $\mathbf{P}_n$ be a convex $n$-gon with vertices labeled from 
$1$ to $n$ in clockwise order.
A \emph{bicolored triangulation of type $(k, n)$} is a triangulation of $\mathbf{P}_n$ where $k$ triangles are colored black and the rest are colored white.
Two bicolored triangulations are
\emph{equivalent}
if the union of the black triangles of one is equal to the union of the black triangles of the other.
We represent the equivalence class 
of a bicolored triangulation $\T$ by erasing the diagonals that separate
pairs of triangles of the same color.  The resulting object
	$\overline{\T}$ 
	is a subdivision of $\mathbf{P}_n$ into white and black
	polygons, and is called a \emph{bicolored subdivision of type $(k,n)$}.
	See \cref{fig:unpunctured}.
\end{definition}

Note that in a bicolored subdivision, as defined above, no two polygons of the same color share an edge. Bicolored triangulations of type $(k,n)$ were
called \emph{$k$ nonintersecting triangles in a convex $n$-gon} 
 in \cite{Lukowski:2019sxw}. We will see later (c.f. \cref{rmk:subdivisionAreTiling}) that bicolored triangulations and subdivisions are special cases of the \emph{plabic tilings} of \cite{OPS}.

Given $\T$ a bicolored triangulation of type $(k,n)$,
we build a corresponding bipartite graph 
$\hatG(\T)$
	as in \cref{fig:Kasteleyn},
then use the recipe from 
\cref{thm:Kasteleyn} and 
\cref{rm:Kasteleyn} 
to construct all 
points of the $2k$-dimensional 
cell 
$S_{\hatG(\T)}$ 
of $\Grk$.

\begin{definition}\label{def:param}
Given $\T$ a bicolored triangulation of type $(k,n)$, 
we build a labeled bipartite graph $\hatG(\T)$ by 
placing black boundary vertices
labeled $B_1,B_2,\dots,B_n$ in 
	clockwise order  
	at the $n$ vertices of the $n$-gon,
	and placing a trivalent white vertex in the middle of 
	each black triangle, connecting it to the three vertices of the 
	triangle.
We label the $k$ white vertices by $W_1,\dots, W_k$; we will usually
	label them in the order specified by \cref{rem:lex}.
\end{definition}

\begin{figure}[h]
\includegraphics[width=0.4\textwidth]{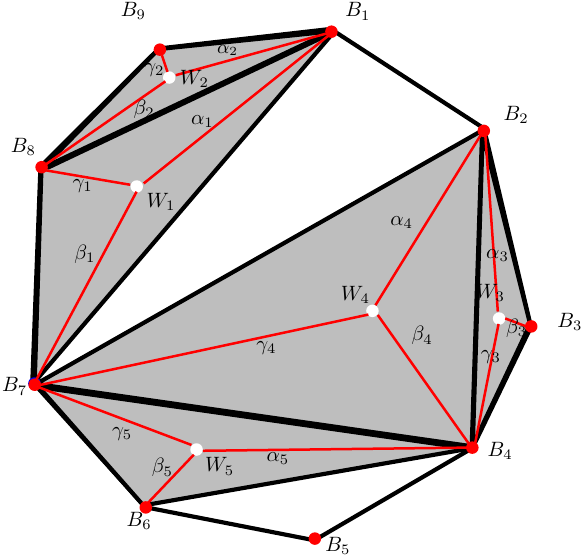}
	\caption{The planar bipartite graph $\hatG(\T_1)$ together
	with its edge-weighting. }
	\label{fig:Kasteleyn}
\end{figure}

\begin{remark}\label{rmk:hatGplabic}
	We can think of $\hatG(\T)$ as a \emph{plabic graph} (see \cref{def:plabic}) if 
we enclose it 
in a slightly larger disk and 
add $n$ edges connecting each $B_i$ to the boundary of the disk.
We will often abuse terminology and refer to $\hatG(\T)$
as a plabic graph. Note that $\hatG(\T)$ does not depend on the triangulation of the white polygons of $\overline{\T}$.
\end{remark}

\begin{lemma}\label{lem:equiv}
If two bicolored triangulations $\T_1$ and $\T_2$
are equivalent, then the plabic graphs 
	$\hatG(\T_1)$ and $\hatG(\T_2)$ are move-equivalent (see \cref{def:move}). 
In other words, these two plabic graphs represent the same
cell of $\Grk$.
\end{lemma}
\begin{proof}
The fact that $\T_1$ and $\T_2$ are equivalent means that 
we can get from $\T_1$ to $\T_2$ by flipping diagonals inside the black and white polygons of $\overline{\T_1}$.
A flip inside a white polygon does not change the plabic graph, while a flip inside a black polygon corresponds
performing a square move on the plabic graph.
So 
$\hatG(\T_1)$ and $\hatG(\T_2)$ are move-equivalent.
\end{proof}
In light of \cref{lem:equiv}, we 
let $S_{\hatG(\overline{\T})}$ denote the cell specified by any triangulation
 of $\overline{\T}$.

\begin{remark}\label{rem:lex}
We identify each black triangle $T$ in a bicolored triangulation $\mathcal{T}$
with its three vertices $a<b<c$ listed in increasing order.
We list the $k$ black triangles 
$$(T_1,\dots,T_k)=(\{a_1<b_1<c_1 \}, \ldots, \{a_k < b_k <c_k \})$$ 
 in lexicographically
increasing order, and label the white vertex inside of $T_i$ by $W_i$. 
\end{remark}

For example, we list the five black triangles of the bicolored triangulation 
$\T_1$ from \cref{fig:unpunctured} in the order
$$(\{1<7<8\},\{1<8<9\},\{2<3<4\},\{2<4<7\},\{4<6<7\}).$$
We label the white vertices of 
$\hatG(\T_1)$
in \cref{fig:Kasteleyn} so as to reflect
this ordering on black triangles.

\begin{definition}[Statistics of bicolored triangulations]	\label{def:arcs}
	Given a bicolored triangulation $\T$ of type $(k,n)$
	and a pair of vertices $ h, j$ of $\mathbf{P}_n$, we say that the arc $h \to j$ is:
	\begin{itemize}
		\item \emph{compatible} with $\T$ if the arc does not cross any arcs of the underlying bicolored subdivision $\overline{\T}$,
	i.e. it either bounds a polygon
	of $\overline{\T}$ or it lies entirely inside a black or white polygon;
	\item a \emph{black arc} of $\T$ if it bounds a black triangle of $\T$;
		\item \emph{facet-defining} if it bounds a black polygon of $\overline{\T}$ on its left.
	\end{itemize}
	In particular, each black arc of $\T$
	is compatible with $\T$.
	
	When $h\to j$ is compatible with $\T$, we let 
	$\area(h\to j) = \area_{\T}(h\to j)$ 
	denote the number of black triangles to the left of $h\to j$ 
	in any triangulation of $\overline{\T}$ which uses $h\to j$.
\end{definition}

	For example, the arcs $1\to 8$, $1\to 7$ and $2\to 6$ are compatible
	with the bicolored triangulation $\T$ from \cref{fig:Kasteleyn}, and we
	have $\area(1\to 8) = 4$, $\area(1\to 7) = 3$, $\area(2\to 6) = 2$.
	However, the arcs $2\to 8$ and $3\to 8$ are not compatible with $\T$.


We can easily write down representative matrices for points in $S_{\hat{G}(\T)}$ using the theory of Kasteleyn matrices. Note that matrices with the same pattern of zero/nonzero entries appeared in \cite{Lukowski:2019sxw} (though the authors did not prove \cref{prop:signsKmatrix} there).

\begin{prop}\label{prop:signsKmatrix}
Let $\mathcal{T}$ be a bicolored triangulation of type $(k,n)$.
	We 
let	\[(\{a_1<b_1<c_1 \}, \ldots, \{a_k < b_k <c_k \})\]
denote the list of $k$ black triangles of $\T$, written  
in lexicographically
increasing order, 
as in \cref{rem:lex}. 
	Choose 
	a set of edge-weights for the graph $\hatG(\mathcal{T})$,
	which we write as 
	$$
	({\Aalpha}, \Bbeta, \Ggamma) = 
	((\alpha_1,\beta_1,\gamma_1), (\alpha_2, \beta_2, \gamma_2),\dots,
	(\alpha_k, \beta_k, \gamma_k))
	\in (\mathbb{R}_{> 0})^{3k},
	$$
	with $\alpha_i, \beta_i, \gamma_i$ denoting the weights on the 
	edges from $W_i$ to $B_{a_i}, B_{b_i}$ and $B_{c_i}$, respectively.

Let $M_{\T}(\Aalpha, \Bbeta, \Ggamma)=(M_{i,j})$ be the $k \times n$ matrix
	with precisely $3$ nonzero entries in each row:
	\begin{equation}\label{eq:M}
		M_{i, a_i}=\alpha_i, \quad
		M_{i, b_i}= (-1)^{\area(a_i \to b_i)} \beta_i, \quad
		M_{i, c_i}=(-1)^{\area(a_i \to b_i)+ \area(b_i \to c_i)} \gamma_i.
	\end{equation}
	Then the 
	cell $S_{\hat{G}(\mathcal{T})}$ is the
	image of the 
	map $(\mathbb{R}_{>0})^{3k} \to \Gr_{k,n}^{\geq0}$ sending 
	$(\Aalpha, \Bbeta, \Ggamma) \mapsto M_{\T}(\Aalpha, \Bbeta, \Ggamma)$.
\end{prop}

Note that 
$M_{\T}(\Aalpha, \Bbeta, \Ggamma)$
has rows and columns indexed by the 
white  and black
vertices of $\hatG(\T)$.  The
	$ij$-entry is nonzero if and only if there 
 is an edge $e$ in $\hatG(\T)$ between $W_i$ and $B_j$,
 and in that case is 
 (up to a sign) 
equal to the weight of $e$.

	\begin{remark}
	Clearly the image of the map 
	$(\Aalpha, \Bbeta, \Ggamma) \mapsto M_{\T}(\Aalpha, \Bbeta, \Ggamma)$ 
	is unchanged if we rescale each row of the matrix
	so that the leftmost nonzero entry is $1$, i.e. set each $\alpha_i=1$.
		This map is then a homeomorphism from 
		$(\R_{>0})^{2k}$ to 
		the positroid cell $S_{\hat{G}(\mathcal{T})}$.
	\end{remark}

	\begin{proof}[Proof of \cref{prop:signsKmatrix}]
	This follows from \cref{thm:Kasteleyn} and 
		\cref{rm:Kasteleyn}. For completeness, we sketch why the choice of signs of entries is correct. For a black triangle $T_i=\{a<b<c\}$ of $\T$, define
				\begin{equation*}
			\epsilon_{i, a}:=(-1)^{\#\{j < i: a_i =a_j\}}, \quad \epsilon_{i, b}:=\epsilon_{i, a} \cdot (-1)^{\area(a_i \to b_i)}, \quad  \epsilon_{i, c}:=\epsilon_{i, a} \cdot (-1)^{\area(a_i \to c_i)+1}.
		\end{equation*}	
		Let $(d_1,\ldots,d_k)$ be a tuple of distinct vertices of black triangles of $\T$ such that $d_{i} \in T_i$. The sign of the permutation $\sigma$ such that $d_{\sigma(1)}<\ldots<d_{\sigma(k)}$ is the product $\epsilon_{1,d_1} \cdots \epsilon_{k,d_k}$.
		Then:
		\begin{equation*}
		p_{I}(M) E_I=\sum_{(d_1,\ldots,d_k)} M_{1,d_1} \cdots M_{k,d_k} \langle e_{d_1},\ldots, e_{d_k} \rangle=\sum_{(d_1,\ldots,d_k)} (\epsilon_{1,d_1} M_{1,d_1} )\cdots (\epsilon_{k,d_k} M_{k,d_k} ) \, E_I 
		\end{equation*}
		where the sum is over the collections defined above satisfying $\{d_1,\ldots,d_k\}=I$. A sufficient condition for $p_{I}(M)\geq0$ is that $\sgn{M_{i,d_i}}=\epsilon_{i,d_i}$. Up to rescaling the row $i$ of $M$ by $\epsilon_{i,a_i}$, this is true, as $\area(a_i \to c_i)=\area(a_i \to b_i)+\area(b_i \to c_i)+1$.
	\end{proof}

\begin{example}\label{ex:1}
	For example, the matrix $M_{\T_1}$ 
	corresponding to the bicolored triangulation $\T_1$ from \cref{fig:Kasteleyn}
	is 
	\begin{equation}\label{matrix1}
		\begin{pmatrix}
			\alpha_1&0&0&0&0&0&-\beta_1&-\gamma_1&0 \\
			\alpha_2&0&0&0&0&0&0&\beta_2&\gamma_2\\
			0&\alpha_3&\beta_3&\gamma_3&0&0&0&0&0\\
			0&\alpha_4&0&-\beta_4&0&0&\gamma_4&0&0\\
			0&0&0&\alpha_5&0&\beta_5&\gamma_5&0&0
		\end{pmatrix}.
	\end{equation}
	\cref{prop:signsKmatrix} says that if we let 
	the parameters $((\alpha_1,\beta_1,\gamma_1),\dots,
	(\alpha_5,\beta_5, \gamma_5))$ range over 
	all elements of $(\R_{>0})^{15}$, the matrices 
	\eqref{matrix1} will sweep out all points of the 
	cell $S_{\hatG(\T_1)}$.
\end{example}

\begin{remark}
	The matrices constructed in \cref{prop:signsKmatrix} may have non-positive maximal minors rather than non-negative maximal minors. To obtain a matrix which has non-negative maximal minors, multiply row $j$ by $(-1)^{\#\{i < j: a_i =a_j\}}$.
\end{remark}

\begin{lemma}\label{lem:closure}
	Let $\T = \{T_1,\dots,T_k\}$ be a bicolored triangulation of type $(k, n)$.
	Then 
	$P_I\neq 0$ on the positroid cell
	 $S_{\hatG(\T)}$ 
	if and only if there is a bijection 
	$\phi: I = \{i_1,\dots,i_k\} \to \{T_1,\dots, T_k\}$ with 
	$i$ a vertex of $\phi(i)$ for all $i$.
\end{lemma}
\begin{proof}
	It suffices to show that 
	the Pl\"ucker coordinate $P_I$ is nonzero on $S_{\hatG(\T)}$ 
	if and only if there is a bijection 
	$\phi: I = \{i_1,\dots,i_k\} \to \{T_1,\dots, T_k\}$ with 
	$i$ a vertex of $\phi(I)$ for all $i$.

	By 
	\cref{thm:Kasteleyn},
	$p_I \neq 0$ on $S_{\hatG(\T)}$ if and only if 
	there is a  matching $M$ of $\hatG(\T)$
	such that $\partial M = I$.  
	Note that any 
	matching $M$ of $\hatG(\T)$ consists of 
	$k$ edges, obtained by pairing
	each white vertex $W_j$ with one of its three incident
	black vertices $\{B_{a_j}, B_{b_j}, B_{c_j}\}$.    
	The $k$ black vertices $\{B_{i_1},\dots,B_{i_k}\}$
	obtained in this way 
	must  be distinct (since $M$ is a matching), 
	so we get
	a bijection between $I:=\{i_1,\dots,i_k\}$ 
	and the 
	black triangles $T_1,\dots,T_k$.
	Moreover
	$\partial M = I$.
\end{proof}

Now, we turn to the open Grasstopes $\gto{\hatG(\T)}$ and their properties.

\begin{theorem}[Definite signs of twistor coordinates]\label{thm:sign1}
Let $\T$ be a bicolored triangulation of type $(k,n)$
and let $Y:=CZ\in \Gr_{k,k+2}$, where $C$ is a matrix 
representing a point of  the cell 
$S_{\hatG(\T)}$.
Choose $h<j$ such that the chord $h\to j$ is compatible with $\T$.
Then 
	\begin{equation} \label{signcoord}
		\sgn\langle Y Z_h Z_j\rangle = (-1)^{\area(h\to j)}, 
		\text{ or equivalently, }
		(-1)^{\area(h\to j)} \lrangle{Y Z_h Z_j} > 0.
	\end{equation}
	In other words, we have that 
	\begin{equation} \label{eq:inclusion}
		Z^{\circ}_{\hat{G}(\T)} \subseteq \{Y\in \Gr_{k,k+2} \ \vert \ 
	\eqref{signcoord} \text{ holds for all arcs }h\to j
	\text{ compatible with }\T.\}
	\end{equation}
\end{theorem}

\begin{proof}
	We start by choosing a bicolored triangulation $\T_1$
	such that $\overline{\T_1} = \T$, and such that 
	the chord $h\to j$ is one of the diagonals of $\T_1$.  By 
	\cref{lem:equiv}, the choice of $\T_1$ does not affect the corresponding
	positroid cell.
	By \cref{lem:useful},
	it suffices to verify \eqref{signcoord} for each 
	$Y:=\lrangle{Z_{i_1},\dots,Z_{i_k}}$ indexed by 
	$\{i_1<\dots<i_k\}=I$ such that 
	$p_I\neq 0$ on the cell
	$S_{\hatG(\T)}$.  
	And by \cref{lem:closure}, $p_I \neq 0$ on 
	$S_{\hatG(\T)}$ if and only if 
	there is a bijection 
	$\phi: I = \{i_1,\dots,i_k\} \to 
	\{T_1,\dots, T_k\}$ with 
	$i$ a vertex of $\phi(i)$ for all $i$.
	
	Towards this end, 
	choose $I=\{i_1<\dots < i_k\}$ such that 
	$p_I \neq 0$ on the cell
	$S_{\hatG(\T)}$. 
	We need to calculate
	$\sgn 
	\lrangle{Z_{i_1},\dots,Z_{i_k}, Z_h, Z_j}$.
	
	If $h\in I$ or $j\in I$, 
	$\lrangle{Z_{i_1},\dots,Z_{i_k}, Z_h, Z_j} = 0$.  So without 
	loss of generality, we can assume that 
	$h$ and $j$ are not elements of $I$. 
	Recall that maximal minors of $Z$ are positive: this means that 
	for any ordered sequence $\ell_1 < \dots < \ell_{k+2}$, we have
	$\sgn \lrangle{Z_{\ell_1},\dots, Z_{\ell_{k+2}}} = 1$.
	To determine 
	$\sgn 
	\lrangle{Z_{i_1},\dots,Z_{i_k}, Z_h, Z_j}$, we need to know how many swaps
	are required to put the sequence
	$(i_1,\dots,i_k,h,j)$ in order. 
	Any $i_{\ell}$ which is greater than both $h$ and $j$ needs
	to get swapped past both of them, which has no effect on the 
	sign of the determinant.
	Any $i_{\ell}$ which is less than both $h$ and $j$ does not 
	need to get swapped past either.  Each $i_{\ell}$ such 
	that $h<i_{\ell} < j$ needs to get swapped past $h$ (but not $j$).
	Therefore the parity of the number of swaps required to put 
	the sequence $(i_1,\dots,i_k,h,j)$ in order is 
	the same as the parity of $\#\{i_{\ell}\in I: h < i_{\ell} < j\}$.
	It follows that 
	$\sgn 
	\lrangle{Z_{i_1},\dots,Z_{i_k}, Z_h, Z_j}
	= (-1)^{\#\{i_{\ell}\in I: h < i_{\ell} < j\}}$.
	Finally, the existence of the bijection $\phi$
	means that 
	$\#\{i_{\ell}\in I: h < i_{\ell} < j\}$
	is the number of black triangles of $\T_1$ which are to the left of 
	$h\to j$.  
	
	To complete the proof, we must show that there is some $I \in \binom{[n]}{k}$ containing neither $h$ nor $j$ such that $p_I$ is nonzero. Equivalently, we must find a matching of $\hatG(\T_1)$ which does not have $h$ or $j$ in its boundary. We do so by induction on the number of black triangles of $\T_1$. Clearly there is such a matching if $\hatG(\T_1)$ has a single black triangle. If $h \to j$ is contained in a white polygon of $\T_1$, we cut along $h \to j$ to obtain two smaller bicolored triangulations $\T_2$ and $\T_3$. By induction, we can find matchings of $\hatG(\T_2)$ and $\hatG(\T_3)$ avoiding $h$ and $j$; their union gives the desired matching of $\hatG(\T_1)$. Otherwise, $h \to j$ is the boundary of a black triangle $T_r$ of $\T_1$. Let $c$ be the third vertex of this triangle. Cut $\T_1$ along $h \to j$, $j \to c$ and $c \to h$ to obtain bicolored triangulations of smaller polygons. The $\hatG$ plabic graphs of these bicolored triangulations have matchings avoiding $h, j,c$ by induction, since each smaller polygon contains exactly two of these vertices. The union of these matchings, together with the edge from $B_c$ to $W_r$, gives the desired matching of $\T_1$.
\end{proof}

	The following result solves a kind of 
		``inverse problem:'' given
		$Y\in Z^{\circ}_{\hat{G}(\T)}$, we can construct 
	a particular matrix representative $\twMtx(Y)$ of $\Gr_{k,n}$ whose  image in $\mathcal{A}_{n,k,2}(Z)$ is $Y$. 
		
		\begin{defn}[Twistor coordinate matrix]\label{defn:twistorMtx}
			Let $Y \in \Gr_{k,k+2}$ and let $\T$ be a bicolored triangulation of type $(k,n)$ with black triangles $T_1, \dots, T_k$ labeled as in \cref{rem:lex}. The \emph{twistor coordinate matrix} of $Y$ is the $k \times n$ matrix $\twMtx(Y)=(C_{i,j})$ with precisely 3 nonzero entries in each row:
			\begin{equation}\label{eq:Cmatrix}
				C_{i, a_i}= \lrangle{Yb_ic_i}, \quad
				C_{i, b_i}=- \lrangle{Y  a_i c_i}, \quad
				C_{i, c_i}=\lrangle{Y  a_i b_i}. 
			\end{equation}
			(Recall that e.g. $\lrangle{Y b_i c_i}$ is short-hand for 
			$\lrangle{Y Z_{b_i} Z_{c_i}}$.)
		\end{defn}
		
\begin{theorem}[Inverse problem]\label{prop:niceRep}
	Let $\T$ be a bicolored triangulation of type $(k,n)$ with black triangles $T_1, \dots, T_k$ labeled as in \cref{rem:lex}. 
	Let $Y
	\in Z^{\circ}_{\hatG(\T)}$, i.e. 
	$Y:=\tilde{Z}(V)$ 
for some $V \in S_{\hatG(\T)}$.
	Then $V$ is the row span of the twistor coordinate matrix $C':=\twMtx(Y)$.

In other words, if we let $Y' =C'Z$, then
there is a global scalar $\lambda$ (a polynomial
in $\lrangle{Yab}$'s) such that 
	\begin{equation*}
		\lrangle{Y'ij} = \lambda \lrangle{Yij} \text{ for all }i,j.
	\end{equation*}
\end{theorem}

\begin{example}\label{ex:2}
Let $\T_1$ be the bicolored triangulation from \cref{fig:Kasteleyn}.
\cref{prop:niceRep} says that 
if $V \in S_{\hatG(\T_1)}$ and $Y:=\tilde{Z}(V)$ is
the image of $V$ in $\AA_{n, k, 2}(Z)$, then $V$ 
is the row span of the following matrix:
\begin{equation*}\label{matrix2}
	\begin{pmatrix}
		\lrangle{Y Z_7 Z_8} &0&0&0&0&0&-\lrangle{Y Z_1 Z_8}&\lrangle{Y Z_1 Z_7}&0 \\
		\lrangle{Y Z_8 Z_9}&0&0&0&0&0&0&-\lrangle{Y Z_1 Z_9}&\lrangle{Y Z_1 Z_8}\\
		0&\lrangle{Y Z_3 Z_4}&-\lrangle{Y Z_2 Z_4}&\lrangle{Y Z_2 Z_3}&0&0&0&0&0\\
		0&\lrangle{Y Z_4 Z_7} &0&-\lrangle{Y Z_2 Z_7}&0&0&\lrangle{Y Z_2 Z_4}&0&0\\
		0&0&0&\lrangle{Y Z_6 Z_7}&0&-\lrangle{Y Z_4 Z_7}&\lrangle{Y Z_4 Z_6}&0&0
	\end{pmatrix}
\end{equation*}
\end{example}
\begin{proof}
	Choose a weight vector $(\Aalpha, \Bbeta,\Ggamma)$ so that the matrix $C:=M_\T(\Aalpha, \Bbeta,\Ggamma)$ from \cref{prop:signsKmatrix} represents $V$.
	
	Consider a black triangle $\{a<b<c\}$ of $\T$. Let $W$ be the white vertex of $\hatG(\T)$ in the middle of this triangle and let the edges from $W$ to $B_a$, $B_b$, and $B_c$, respectively, be denoted $e_a$, $e_b$, and $e_c$.  Say the weights of these edges are $\alpha$, $\beta$, and $\gamma$, respectively. 
	
	Choose $J \in \binom{[n]}{k-1}$ which does not contain $a$, $b$, or $c$. Then 
	\begin{equation}\label{eqn:3pluckerEqual}
		\frac{1}{\alpha}p_{J \cup \{a\}}(C)= \frac{1}{\beta} p_{J \cup \{b\}}(C)= \frac{1}{\gamma}p_{J \cup \{c\}}(C).
	\end{equation}
	Indeed, each Pl\"ucker coordinate is a sum of weights of 
	matchings. Any matching $M_a$ contributing to $p_{J \cup \{a\}}(C)$ must include an edge covering the white vertex $W$. Since $b, c \notin J \cup \{a\}$, this edge must be $e_a$. Now, $M_b := M_a \setminus \{e_a\} \cup \{e_b\}$ is a valid matching because $M_a$ does not include any edges covering $B_b$. Moreover, the boundary of $M_b$ is $J \cup \{b\}$. This is easily seen to be a bijection between matchings with boundary $J \cup \{a\}$ and matchings with boundary $J \cup \{b\}$. It is also easy to see that $\wt(M_a)/\alpha= \wt(M_b)/\beta$, so the first equality above holds. The second equality is similar.
	
	Now, we consider the twistor coordinate
	\[\lrangle{Y b c} = \sum_{I \in \binom{[n]}{k}} p_I(C) \lrangle{Z_{i_1}, Z_{i_2}, \dots, Z_{i_k}, Z_{b}, Z_{c}}
	\]
	which is nonzero by \cref{thm:sign1}.

	Notice that the terms in this sum indexed by $I$ containing $b$ or $c$ are zero. Further, for $I \cap \{b, c\}= \emptyset$, $p_I(C)$ is zero if $I$ does not contain $a$. So we can rewrite $\lrangle{Ybc}$ as
	\begin{equation}\label{eqn:twistorA}
		\lrangle{Y bc} = \alpha \cdot \sum_{\substack{{J \in \binom{[n]}{k-1}:}\\{\{a, b, c\} \cap J = \emptyset}}} \frac{1}{\alpha}~p_{J \cup a}(C) \lrangle{Z_{j_1}, Z_{j_2}, \dots,Z_a, \dots Z_{j_{k-1}}, Z_{b}, Z_{c}}
	\end{equation} 
	where $Z_{j_1}, \dots, Z_a, \dots, Z_{j_{k-1}}$ are ordered so the indices are increasing.
	
	Similarly, we can write
	\begin{align}
		\label{eqn:twistorB} \lrangle{Ya c} &= \beta \cdot \sum_{\substack{{J \in \binom{[n]}{k-1}:}\\{\{a, b, c\} \cap J = \emptyset}}} \frac{1}{\beta}~ p_{J \cup b}(C) \lrangle{Z_{j_1}, Z_{j_2}, \dots,Z_b, \dots Z_{j_{k-1}}, Z_{a}, Z_{c}}\\
		\label{eqn:twistorC} \lrangle{Ya  b} &= \gamma	\cdot \sum_{\substack{{J \in \binom{[n]}{k-1}:}\\{\{a, b, c\} \cap J = \emptyset}}} \frac{1}{\gamma}~ p_{J \cup c}(C) \lrangle{Z_{j_1}, Z_{j_2}, \dots,Z_c, \dots Z_{j_{k-1}}, Z_{a}, Z_{b}}.
	\end{align}
	
	Consider a nonzero term in \eqref{eqn:twistorA}, which is indexed by $J$ such that $p_{J \cup a}(C)$ is nonzero. The corresponding term in \eqref{eqn:twistorB} is also nonzero. Because of the first equality in \eqref{eqn:3pluckerEqual}, these two terms differ only by the sign $(-1)^s$, where 
	\[\lrangle{Z_{j_1}, Z_{j_2}, \dots,Z_a, \dots Z_{j_{k-1}}, Z_{b}, Z_{c}}= (-1)^s\lrangle{Z_{j_1}, Z_{j_2}, \dots,Z_b, \dots Z_{j_{k-1}}, Z_{a}, Z_{c}}.\]
	In other words, $s=|J \cap [a+1, b-1]|+1= |(J \cup a) \cap [a+1, b-1]|+1$. Because $J \cup a$ is the boundary of some matching, the size of $(J \cup a) \cap [a+1, b-1]$ is exactly $\area(a \to b)$, and in particular does not depend on $J$.
	
	Similarly, consider the term of \eqref{eqn:twistorC} indexed by $J$. The sign difference between this term and the corresponding one in \eqref{eqn:twistorB} is $(-1)^s$, where \[\lrangle{Z_{j_1}, Z_{j_2}, \dots,Z_b, \dots Z_{j_{k-1}}, Z_{a}, Z_{c}}= (-1)^s\lrangle{Z_{j_1}, Z_{j_2}, \dots,Z_c, \dots Z_{j_{k-1}}, Z_{a}, Z_{b}}.\]
	It is not hard to see that $s=\area(b \to c)+1$.
	
	Altogether, we have 
	\begin{align*}
		\lrangle{Y b c}&= \alpha \cdot Q\\
		\lrangle{Y ac}&= (-1)^{\area(a\to b)+1}\beta \cdot Q\\
		\lrangle{Y ab}&= (-1)^{\area(a \to b)+\area(b \to c)} \gamma \cdot Q
	\end{align*}
	where $Q$ is a nonzero scalar. Notice that up to the factor of $Q$, these three twistor coordinates recover the entries of $C$ corresponding to the edges $e_a$, $e_b$, and $e_c$. This means that the matrix $C'$ with non-zero entries
	\begin{equation*}
		C'_{j, a_j}= \lrangle{Y  b_j c_j}, \quad
		C'_{j, b_j}=-\lrangle{Y  a_j  c_j}, \quad
		C'_{j, c_j}=\lrangle{Y  a_j b_j}
	\end{equation*}
	is related to $M_{\hatG(\T)}(\Aalpha, \Bbeta, \Ggamma)$ by rescaling rows, and so also represents the subspace $V$.
\end{proof}

Using \cref{prop:niceRep}, we can show that $\tilde{Z}$ is injective on $S_{\hatG(\T)}$, and moreover prove that $\tilde{Z}$ is not injective on any other 
$2k$-dimensional positroid cells. This will prove the conjectural
characterization of positroid tiles from 
 \cite{Lukowski:2019sxw}.  We note that the injectivity of $\tilde{Z}$ on $S_{\hatG(\T)}$ was also proved rather indirectly in \cite[Proposition 6.4]{LPW} using results of \cite{BaoHe}.

\begin{theorem}[Characterization of positroid tiles]\label{thm:allGTs}
Fix $k<n$ and $Z \in \Mat^{>0}_{n, k+2}$. Then $\tilde{Z}$ is injective on the $2k$-dimensional cell $S_\M$ if and only if  $S_\M = S_{\hatG(\overline{\T})}$ for some bicolored subdivision $\overline{\T}$ of type $(k,n)$.
That is, the positroid tiles for $\AA_{n, k, 2}$ are exactly  
	the Grasstopes $\gt{\hatG(\overline{\T})}$, where $\overline{\T}$ is a bicolored subdivision of type $(k,n)$.  
\end{theorem}
\begin{corollary}	
Whether or not $\overline{\tilde{Z}(S_{\pi})}$
is a positroid tile is independent of 
	$Z$. 
\end{corollary}
\begin{proof}[Proof of \cref{thm:allGTs}]
	This proof uses some facts from \cref{sec:Tduality1}.
	We first show that all cells $S_{\hatG(\overline{\T})}$ are positroid tiles. The cell $S_{\hatG(\overline{\T})} \subset \Grk$ is $2k$-dimensional because it is T-dual to an $(n-1)$-dimensional cell in $\Gr_{k+1,n}^{\geq 0 }$ (see \cref{rk:treesTilings}) and T-duality preserves codimension (see \cref{prop:posetiso}).
Say $V, V' \in S_{\hatG(\T)}$ are represented by matrices $C, C'$, and suppose $Y:=CZ, Y':=C'Z$ represent the same subspace. Then by \cref{prop:niceRep} $V$ and $V'$ are represented by the twistor coordinate matrices $N$ and $N'$ of $Y$ and $Y'$, respectively. But the twistor coordinates of $Y$ and $Y'$ are the same up to a global scalar, so $V=V'$.

Now, suppose a $2k$-dimensional cell $S_\M$ is not equal to $S_{\hatG(\overline{\T})}$ for any $\overline{\T}$.
	We will show $\tilde{Z}$ is not injective on $S_\M$.

First, suppose $M$ has a \emph{coloop} $c$; that is, $p_I$ is identically 0 on $S_\M$ for all $I\in \binom{[n]}{k}$ that do not contain $c$. Then the twistor coordinate $\lrangle{Y c j}$ is identically zero on $\gto{\M}$ for all $j$. Indeed, in the sum
\[\lrangle{Y cj}= \sum_{I \in \binom{[n]}{k}} p_I(C) \lrangle{Z_{i_1} \dots Z_{i_k} Z_c Z_j},
\]
$p_I(C)$ is zero for $c \notin I$ and $\lrangle{Z_{i_1} \dots Z_{i_k} Z_c Z_j}$ is zero for $c \in I$. In particular, $\gto{\M}$ is contained in the hypersurface $\{Y \in \Gr_{k,k+2}:\lrangle{Yc (c+1)}=0\}$, and so has dimension at most $2k-1$. So $\tilde{Z}$ is not injective on $S_\M$.

Now, if $\M$ does not have a coloop, then $S_\M$ is $T$-dual to an $(n-1)$-dimensional cell $S_\pi$ of $\Gr^{\geq 0}_{k+1, n}$ by \cref{prop:posetiso}. Because $S_\M$ is not of the form $S_{\hatG(\T)}$, a plabic graph $G$ with trip permutation $\pi$ is not a tree and so has at least one internal face. Since $G$ has $n$ faces total, $G$ is not connected.

Let $G$ be a plabic graph with trip permutation $\pi$, and say $[i, j-1]$, $[j, l]$ are the boundary vertex sets of two connected components of $G$. There is a single boundary face $f$ which is adjacent to $i-1$, $i$, $j-1$ and $j$. In the plabic graph $\hatG$ for $S_\M$ (constructed in \cref{prop:TdualityPlabicGraphs}), notice that $i$ and $j$ are adjacent to the same black vertex, $\hat{b}(f)$. After adding bivalent white vertices to $\hatG$ so that every boundary vertex is adjacent to a white vertex, it is clear that all matchings of $\hatG$ have either $i$ or $j$ in the boundary. This means that if $I$ contains neither $i$ nor $j$, then $p_I$ is identically zero on $S_\M$. Just as in the coloop case, $\lrangle{Y i j}$ is identically zero on $\gto{\M}$, because all terms of 
\[\sum_{I \in \binom{[n]}{k}} p_I(C) \lrangle{Z_{i_1} \dots Z_{i_k} Z_i Z_j}\]
	vanish for $C \in S_\M$. So 
	$\gto{\M}$ is contained in a hypersurface and hence  $\dim \gto{\M} \leq 2k-1$.
\end{proof}

\begin{rmk}
As conjectured in \cite{Lukowski:2019sxw}, the number of positroid tiles 
	for $\AA_{n, k, 2}$ is sequence A175124 in the OEIS \cite{OEIS}, a refinement of the \emph{large Schr\"oder numbers}
	(see \cref{sec:appendixB}).
\end{rmk}

Refining \eqref{eq:inclusion}, we will now give an explicit description of 
each open positroid tile as a subset of 
$\Gr_{k,k+2}$ where certain twistor
coordinates have a definite sign.  In fact, since 
there are generally multiple bicolored triangulations represented by of one bicolored subdivision $\overline{\T}$, \cref{thm:surjectivity}
gives multiple  descriptions 
of each open positroid tile -- one for each bicolored triangulation represented by $\overline{\T}$.

\begin{theorem}[Sign characterization of positroid tiles]\label{thm:surjectivity}
Fix $k<n$, $m=2$, and $Z\in \Mat^{>0}_{n,k+2}$.
	Let $\T$ be a bicolored triangulation of type $(k,n)$. 
	Then we have 
\begin{align*}
\gto{\hatG(\overline{\T})}&= \{Y \in \Gr_{k, k+2} \ \vert \  
	  \sgn \lrangle{Y i j}= (-1)^{\area(i \to j)} 
	\text{ for all black arcs } i \to j \text{ of } \T \text{ with } i <j \}
\end{align*}
Moreover, 
if	$Y\in Z^{\circ}_{\hat{G}(\T)}$, then 
	$C':=\twMtx(Y)$ (cf. \cref{defn:twistorMtx}) 
	lies in the positroid cell $S_{\hat{G}(\T)}$,
	and $Y$ and 
	$C'Z$ represent the same element of $\Gr_{k,k+2}$. 
\end{theorem} 

In the proof of \cref{thm:surjectivity}, we use the notation $N_i(A):=\#\{a \in A: a<i \}$ and $N_{i, j}(A):=\#\{a \in A: i<a<j \}$. We will need the following lemmas.

\begin{lemma} \label{lemma:kernelZ}
	Let $S \in {[n] \choose k+3}$, and define $\omega^S \in \mathbb{R}^n$ as
		\[\omega_i^S = \begin{cases}
		(-1)^{N_i(S)} \lrangle{Z_{S \setminus \{i\}}} & \text{ if } i \in S\\
		0 & \text{ else }.
	\end{cases}
	\]
Then $\omega^S$ is in the left kernel of $Z$.
\end{lemma}
\begin{proof} We have that 
	\begin{equation*}
		 (\omega^S )^T \cdot Z=\sum_{i=1}^n Z_i \omega^S_i=\sum_{i \in S} (-1)^{N_i(S)} Z_i \, \lrangle{Z_{S \setminus \{i\}}}=\sum_{i \in S} \epsilon_{\{i\},S \setminus \{i\}} Z_i \, \lrangle{Z_{S \setminus \{i\}}}.
	\end{equation*}

	From the rightmost expression, one can see that the $j$th coordinate of $\omega^S \cdot Z$ is the determinant of the submatrix of $Z$ using rows $S$ and columns $1, \dots, j, j, \dots, k+2$, written using Laplace expansion along column $j$. Therefore it is zero.
\end{proof}

\begin{proposition}\label{prop:kindaInverse}
	Let $Y \in \Gr_{k, k+2}$ and let $\T$ be a bicolored triangulation of type $(k,n)$. Let $\twMtx=\twMtx(Y)$ be the twistor coordinate matrix of $Y$ and let $Y':=\twMtx Z$. Then 
	\[\rowspan(Y') \subseteq \rowspan(Y).\]
\end{proposition}

\begin{proof}
	We start by writing $Y=CZ$, where $C$ is a full-rank $k \times n$ matrix (we can always do this because the linear map $Z: \R^n \to \R^{k+2}$ is surjective).
	
		Let $C_1,\ldots, C_k$ be the rows of the matrix $C$.
	We will replace each row $C_i$ with a linear combination of the 
	rows of $C$ \emph{and} a linear combination of elements of 
	$\ker(Z)$ to obtain a new matrix $C'$; by construction, the rowspan of $C'Z$ is contained in the rowspan of $CZ$. We will show that this new matrix $C'$ is equal to $\twMtx$.
	
	Specifically, let $T_{i}=\{a<b<c\}$ be a black triangle in $\mathcal{T}$. The $i$th row $C'_i$ of $C'$ is
	\begin{equation} \label{eq:C'asComega}
		C'_i:=\sum_{j=1}^k {\lambda}_{j} C_{j}+\sum_{S \in {[n]\choose k+3}} \rho_{S} \omega^S, \text{ where } 
	\end{equation}
	\begin{equation*}
		\lambda_{j}=\sum_{J \in {[n] \choose k}} (-1)^{j+1+N_a(J)+N_{b,c}(J)} p_{J \setminus \{a\}}(C_{\hat{j}}) \langle Z_{J \cup \{b,c \}} \rangle \text{ and } 
		\rho_{S}=(-1)^{N_a(S)+N_{b,c}(S)} p_{S \setminus \{a,b,c\}}(C),
	\end{equation*}
	where $C_{\hat{j}}$ denotes the
	matrix obtained from $C$ by removing row $j$, and we 
	make the convention that $p_{A \setminus B}(C)=0$ if 
	$B$ is not contained in $A$, and $\langle Z_{A \cup B} \rangle=0$ if $A$ intersects $B$.

	{\bf Step 1.}
	We first show that $\langle Y b c\rangle = 
	C'_{i a}$. 
	By \eqref{eq:C'asComega}, we have 
	\begin{equation} \label{eq:C'asComega2}
		C'_{ia}:=\sum_{j=1}^k {\lambda}_{j} C_{ja}+\sum_{S \in {[n]\choose k+3}} \rho_{S} \omega^S_a. 
	\end{equation}

	Let us expand $\langle Y b c\rangle$ as:
	\begin{equation} \label{eq:setp1expansion}
		\langle Y b c\rangle=\sum_{J \in {[n] \choose k}}(-1)^{N_{b,c}(J)} p_J(C) \langle Z_{J \cup \{b,c \}} \rangle.
	\end{equation}
	Call the terms in this sum with $a \in J$ ``type A" and the other terms ``type B." 
	
	When $a\in J$, we can compute $p_J(C)$ by Laplace expansion around column $a$:
	\begin{equation*}
		p_J(C)=\sum_{j=1}^k (-1)^{j+1+ N_a(J)} p_{J \setminus \{a\}}(C_{\hat{j}}) \, C_{j a}.
	\end{equation*}
	Inserting this into the type A terms and summing over $J$, we obtain the first term in the right hand side of \eqref{eq:C'asComega2}.
	
	For the type B terms, we can change the summation index 
	in \eqref{eq:setp1expansion}
	from $J$ to $S=J \cup \{a,b,c\}$,  obtaining:
	\begin{equation*}
		\sum_{S \in {[n] \choose k+3}}(-1)^{N_{b,c}(S \setminus \{a,b,c\})} p_{S \setminus \{a,b,c \}}(C) \langle Z_{S \setminus \{a\}} \rangle
	\end{equation*}
	Since $a<b<c$, we have $N_{b,c}(S \setminus \{a,b,c\})=N_{b,c}(S)$.
This gives the second term in the right hand side of \eqref{eq:C'asComega2}.
	Hence, summing the terms of type A and type B we get exactly $C'_{i a}$.
	
	{\bf Step 2.} 
	We will show that 
	$\langle Y a c \rangle = 
	-C'_{i b}$.
	
	Let us consider the first term (`type A') in the right hand side of \eqref{eq:C'asComega}. We observe that:
	\begin{equation*}
		\sum_{j=1}^k (-1)^{j+1+N_a(J)} p_{J \setminus \{a\}}(C_{\hat{j}}) C_{j b}=p_J(C^{a \rightarrow b})=(-1)^{N_{a,b}(J)}p_{J \setminus \{a\} \cup \{b\}}(C),
	\end{equation*}
	where $C^{a \rightarrow b}$ is the matrix $C$ with column $a$ substituted with column $b$. Noting that $N_{a,b}(J)+N_{b,c}(J)=N_{a,c}(J)$ as terms with $b \in J$ do not contribute, type A reads:
	\begin{equation*}
		\sum_{J \in {[n] \choose k}}(-1)^{N_{a,c} (J)} p_{J \setminus \{a\} \cup \{b\}}(C) \langle Z_{J \cup \{b,c\}} \rangle
	\end{equation*}
	Finally, we change summation index into $J'=J \setminus \{a\} \cup \{b\}$ and use $N_{a,c}(J')=N_{a,c}(J' \setminus \{b\} \cup \{a\})+1$, as $b \in J'$ and $a<b<c$, to obtain:
	\begin{equation} \label{eq:2ndTypeA}
		-\sum_{J' \in {[n] \choose k}: b \in J'}(-1)^{N_{a,c}(J')} p_{J'}(C) Z_{J' \cup \{a,c\}}.
	\end{equation}
	
	Let us consider the second term (`type B') in the right hand side of \eqref{eq:C'asComega}. 
	Using $N_b(S)-N_a(S)=N_{a,b}(S)-1$ and $N_{a,b}(S)+N_{b,c}(S)=N_{a,c}(S)-1$, as $a,b \in S$, type B reads:
	\begin{equation*}
		\sum_{S \in {[n] \choose k+3}}(-1)^{N_{a,c}(S)} p_{S \setminus \{a,b,c \}}(C) \langle Z_{S \setminus \{b\}} \rangle
	\end{equation*}
	Finally, we perform the change of summation index into $J=S \setminus \{a,b,c\}$ and note that $N_{a,c}(S)=N_{a,c}(J \cup \{a,b,c\})=N_{a,c}(J)+1$, as $b \not \in J$ and $a<b<c$. We obtain:
	\begin{equation} \label{eq:2ndTypeB}
		-\sum_{J \in {[n] \choose k}: b \not \in J}(-1)^{N_{a,c}(J)} p_{J}(C) \langle Z_{J \cup \{a,c\}} \rangle.
	\end{equation}
	Hence adding together \eqref{eq:2ndTypeA} and \eqref{eq:2ndTypeB} we immediately get $-\langle Y a c\rangle= C'_{i b}$.
	
	{\bf Step 3.}  Showing that 
	$\langle Y a b \rangle = 
	C'_{i c}$ is similar to the previous case.

	{\bf Step 4.}
	We will show that $C'_{i \ell}=0$ for $\ell \not \in \{a,b,c\}$. 
	
	Let us consider the first term (`type A') in the right hand side of \eqref{eq:C'asComega}. We observe that:
	\begin{equation*}
		\sum_{j=1}^k (-1)^{j+1+N_a(J)} p_{J \setminus \{a\}}(C_{\hat{j}}) C_{j \ell}=p_J(C^{a \rightarrow \ell})=(-1)^{\tilde{N}_{a,\ell}(J)}p_{J \setminus \{a\} \cup \{\ell\}}(C),
	\end{equation*}
	where $\tilde{N}_{a,\ell}(J)$ is defined as $N_{a,\ell}(J)$ if $a<\ell$ and $N_{\ell,a}(J)$ if $\ell<a$. Then type A reads:
	\begin{equation*}
		\sum_{J \in {[n] \choose k}}(-1)^{N_{b,c} (J)+\tilde{N}_{a,\ell}(J)} p_{J \setminus \{a\} \cup \{\ell\}}(C) Z_{J \cup \{b,c\}}.
	\end{equation*}
	By changing the summation index into $J'=J \setminus \{a\} \cup \{\ell\}$ and noting that $N_{b,c}(J' \setminus \{\ell\} \cup \{a\})=N_{b,c}(J' \setminus \{\ell\})$ and $\tilde{N}_{a,\ell}(J' \setminus \{\ell\} \cup \{a\})=\tilde{N}_{a,\ell}(J')$, we obtain:
	\begin{equation} \label{eq:4thTypeA}
		\sum_{J' \in {[n] \choose k}}(-1)^{N_{b,c}(J' \setminus \{\ell\})+\tilde{N}_{a,\ell}(J')} p_{J'}(C) \langle Z_{J' \setminus \{\ell\}\cup \{a,b,c\}} \rangle.
	\end{equation}
	
	The Type B term can be rewritten as:
	\begin{equation*}
		-\sum_{S \in {[n] \choose k+3}}(-1)^{N_{b,c}(S)+\tilde{N}_{a,\ell}(S)} p_{S \setminus \{a,b,c \}}(C) \langle Z_{S \setminus \{\ell\}} \rangle
	\end{equation*}
	using $(-1)^{N_a(S)+N_\ell(S)}=(-1)^{\tilde{N}_{a,\ell}(S)+1}$. Indeed, $N_a(S)-N_\ell(S)=N_{\ell,a}(S)+1$ if $\ell<a$ and $N_\ell(S)-N_a(S)=N_{a,\ell}(S)+1$ if $\ell>a$, since $\ell,a \in S$.  
	Finally, we perform the change of variables $J'=S \setminus \{a,b,c\}$ and note that  $N_{b,c}(J' \cup \{a,b,c\})=N_{b,c}(J')$ and $\tilde{N}_{a,\ell}(J' \cup \{a,b,c \})=\tilde{N}_{a,\ell}(J' \cup \{b,c \})$, as $a<b<c$, obtaining:
	\begin{equation}\label{eq:4thTypeB}
		-\sum_{J' \in {[n] \choose k}}(-1)^{N_{b,c}(J')+\tilde{N}_{a,\ell}(J' \cup \{b,c \})} p_{J'}(C) \langle Z_{J' \setminus \{\ell\}\cup \{a,b,c\}} \rangle.
	\end{equation}
	In order to complete the proof, we need to show that the sum of type A in \eqref{eq:4thTypeA} with type B in \eqref{eq:4thTypeB} is zero. Therefore it is enough to show that: 
	\begin{equation}\label{eq:lastequation}
		(-1)^{N_{b,c}(J')+\tilde{N}_{a,\ell}(J' \cup \{b,c \})}=(-1)^{N_{b,c}(J' \setminus \{\ell\})+\tilde{N}_{a,\ell}(J')},
	\end{equation}
	recalling that the only terms contributing have $\ell \in J'$ and $a,b,c \not \in J'$. If $\ell<b$, then  $\tilde{N}_{a,\ell}(J' \cup \{b,c \})=\tilde{N}_{a,\ell}(J')$ and $N_{b,c}(J' \setminus \{\ell\})=N_{b,c}(J')$. If $b<\ell<c$, then $\tilde{N}_{a,\ell}(J' \cup \{b,c \})=\tilde{N}_{a,\ell}(J')+1$ and $N_{b,c}(J' \setminus \{\ell\})=N_{b,c}(J')-1$. Finally, if $\ell>c$ then $\tilde{N}_{a,\ell}(J' \cup \{b,c \})=\tilde{N}_{a,\ell}(J')+2$ and $N_{b,c}(J' \setminus \{\ell\})=N_{b,c}(J')$. Therefore \eqref{eq:lastequation} holds for all three cases and the proof that $C'_{i \ell}=0$ when $\ell \not \in \{a,b,c\}$ is complete. 
\end{proof}

\begin{proof}[Proof of \cref{thm:surjectivity}]
	By \eqref{eq:inclusion}, we just need to show the inclusion
		$$Z^{\circ}_{\hat{G}(\T)} \supseteq \{Y\in \Gr_{k,k+2} \ \vert \ 
	\eqref{signcoord} \text{ holds for all black chords }h\to j
	\text{ of }\T.\}$$
	We will do this using the twistor coordinate matrix $C':=\twMtx(Y)$ for $Y$ in the right-hand set. 
	First, we show that $C' \in S_{\hat{G}(\T)}$. 
	The nonzero entries of $C'$ correspond  to the edges of $\hatG(\T)$. By \cref{thm:Kasteleyn} and \cref{rm:Kasteleyn}, whether or not $C'$ represents an element of $S_{\hat{G}(\mathcal{T})}$ is just a question of whether or not the nonzero entries have the correct signs. By assumption, the nonzero entries of $C'$ have the same signs as the nonzero entries of the matrix $C''$ from \cref{prop:niceRep}. Since the matrix $C''$ represents an element of $S_{\hat{G}(\mathcal{T})}$, so does $C'$.
	
	Now, let $Y':=C'Z$. By \cref{prop:kindaInverse}, $\rowspan Y' \subseteq \rowspan Y$. Because $C'$ is an element of $\Gr_{k,n}^{\geq 0}$, in fact $Y'$ has rank $k$, so the two rowspans are equal. Thus, $Y=C'Z$, which shows $Y \in \gto{\hat{G}(\T)}$.

\end{proof}

\begin{remark}
In the previous proof, the only place that we used the fact that the twistor coordinates $\lrangle{Yhj}$ associated to 
black arcs
	had particular signs  was in showing that the matrix $C'$ that we constructed has maximal minors all nonnegative (or 
	all nonpositive).  We will use this observation in \cref{GTcluster}, when we show that each positroid tile is the totally positive part of a \emph{cluster variety}.
\end{remark}

\begin{cor}\label{cor:bijection}
Let $\T$ be a bicolored triangulation of type $(k,n)$.
The map sending the $k\times n$ matrix 
$M:=M_{\T}(\Aalpha, \Bbeta, \Ggamma)$ from \eqref{eq:M}
representing a point of 
 $S_{\hat{G}(\mathcal{T})} \cong (\R_{>0})^{2k}$
to $Y:=MZ \in Z^{\circ}_{\hatG(\T)}$ is a bijection
 from $S_{\hat{G}(\mathcal{T})} \cong (\R_{>0})^{2k}$
to $Z^{\circ}_{\hatG(\T)}$, and we have  
	\begin{equation}\label{eq:ratio}
		\frac{M_{i,b_i}}{M_{i, a_i}}= -\frac{\lrangle{Y a_i c_i}}{\lrangle{Y b_i c_i}} \quad\text{ and } \quad
		\frac{M_{i,c_i}}{M_{i, a_i}}= \frac{\lrangle{Y a_i b_i}}{\lrangle{Y b_i c_i}}
	\end{equation}
	for all black triangles $\{a_i, b_i,c_i\}$ of $\T$.
	In particular, the $2k$ ratios of twistor coordinates
		$\{\frac{\lrangle{Y a_i c_i}}{\lrangle{Y b_i c_i}}, 
		\frac{\lrangle{Y a_i b_i}}{\lrangle{Y b_i c_i}}\}$
are algebraically independent.
\end{cor}
\begin{proof}
Injectivity follows from 
\cref{thm:allGTs}.
Surjectivity follows from 
\cref{thm:surjectivity}.
	Finally, \eqref{eq:ratio} follows from 
	\cref{prop:signsKmatrix} and	\cref{prop:niceRep}. 
\end{proof}

\section{The equivalence of the two definitions of the amplituhedron}\label{sec:equivalence}

In this section we will give an alternative description of the 
amplituhedron $\mathcal{A}_{n,k,2}(Z)$ in terms of sign flips
of twistor coordinates; this description was conjectured by 
Arkani-Hamed--Thomas--Trnka 
 \cite[(5.6)]{ATT}.
 In \cite[Section 5.4]{ATT}, they sketched
an argument that all elements of $\AA_{n, k, m}(Z)$ satisfy the sign flip
description; 
a proof using a different argument was independently given
in \cite[Corollary 3.21]{karpwilliams}.
However, the opposite inclusion remained open.  We will complete
the proof for $m=2$ using the results of the previous section. 
Finally, we will translate the sign-flip characterization
of  $\AA_{n,k,2}(Z)$ into a sign-flip characterization of the $\mathcal{B}$-amplituhedron $\mathcal{B}_{n,k,2}(W)$.

Recall the definition of $\hat{Z}_i$ from \cref{rem:twisted}.
\begin{theorem}[Sign-flip characterization of $\mathcal{A}_{n,k,2}$]\label{thm:main1}

Fix $k<n$  and $Z\in \Mat^{>0}_{n,k+2}$.
	Let 
	\begin{align*}
		\mathcal{F}^\circ_{n,k,2}(Z):=	\{Y\in Gr_{k,k+2} \ &\vert \  
	\langle Y Z_i Z_{i+1}\rangle>0 \text{ for }1 \leq i \leq n-1,
	\text{ and }   \langle Y Z_n \hat{Z}_1 \rangle >0,\\
		&\text{ and } 
		\var({\langle Y Z_1 Z_2\rangle, 
	\langle Y Z_1 Z_3\rangle,  \dots 
		\langle Y Z_1 Z_n\rangle})=k.\}
	\end{align*}
	Then $\mathcal{A}_{n,k,2}(Z) = \overline{\mathcal{F}^\circ_{n,k,2}(Z)}.$
\end{theorem}

\begin{proof}
	Let $\mathcal{A}^{\circ}_{n,k,2} (Z):=\tilde{Z}(\Gr_{k,n}^{>0})$. 
	By \cref{rem:PL} and \cref{rmk:closure},
	$\AA_{n, k, 2}(Z)=\overline{\mathcal{A}^{\circ}_{n,k,2}(Z)}.$
	
We first show that $\mathcal{A}^{\circ}_{n,k,2} \subseteq
	\mathcal{F}^\circ_{n,k,2}(Z).$
Suppose that $C\in \Gr_{k,n}^{>0}$ and let $Y:=CZ$.
Choose $1 \leq i \leq n-1$, and consider
	any  $J=\{ j_1<\dots <j_k \}\in {[n] \choose k}$. 
Since $Z$ has maximal minors positive,
	the sign of $\langle Z_{j_1},\dots,Z_{j_k}, Z_{i}, Z_{i+1} \rangle$
is determined by the parity of the number of swaps needed to put
the sequence $\{j_1,\dots, j_k, i, i+1\}$ into increasing order.
Clearly this number is even, so 
	$\langle Z_{j_1},\dots,Z_{j_k}, 
	Z_{i}, Z_{i+1} \rangle \geq 0$
	(with equality if $J \cap \{i,i+1\} \neq \emptyset$).
	Therefore by \cref{lem:useful}, 
	$\langle Y Z_i Z_{i+1}\rangle > 0$.
	The argument that $\langle Y Z_n \hat{Z}_1 \rangle > 0$
	is similar, using the fact that the 
	matrix with rows $Z_2,\dots,Z_n, \hat{Z}_1$ has maximal 
	minors positive.
	To see that $Y$ satisfies the sign variation condition,
	see the proof sketch in 
	\cite[Section 5.4]{ATT} or  
	\cite[Corollary 3.21]{karpwilliams}.
This implies that $\mathcal{A}^{\circ}_{n,k,2} \subseteq 
	\mathcal{F}^\circ_{n,k,2}(Z)$
and hence $\mathcal{A}_{n,k,2} \subseteq 
	\overline{\mathcal{F}^\circ_{n,k,2}(Z)}.$

For the other direction, we will show that 
$\mathcal{F}^\circ_{n,k,2}(Z) \subseteq 
\mathcal{A}_{n,k,2}(Z)$.  
Suppose $Y
\in \mathcal{F}^\circ_{n,k,2}(Z).$
We want to show that we can write 
	$\rowspan{Y}=\rowspan{CZ}$ for some $C\in \Grk$.

	Since $\lrangle{Y Z_1 Z_2}>0$, and 
$\var({\langle Y Z_1 Z_2\rangle, 
\langle Y Z_1 Z_3\rangle,  \dots 
\langle Y Z_1 Z_n\rangle})=k$, we can 
find a sequence 
 $1=i_0<i_1< \dots <i_k\leq n-1$ 
	such that $\sgn\lrangle{Y Z_1 Z_{i_\ell+1}}=(-1)^{\ell}$
	for all $\ell$; 
choose the lexicographically minimal such sequence.
Let $\T$ be the bicolored triangulation of type $(k,n)$ whose $k$ black triangles have vertices
$\{1, i_\ell, i_{\ell}+1\}$ for $1\leq \ell \leq k$.
 By \cref{prop:kindaInverse}, if we 
let $\twMtx=\twMtx(Y)$ be the twistor coordinate matrix of $Y$,
and  $Y':=\twMtx Z$, then
      $\rowspan(Y') \subseteq \rowspan(Y).$ To complete the proof,
      we need to show that $\twMtx \in \Gr_{k,n}^{\geq 0}$,
      and $Y'$ has full rank.

Using \cref{thm:Kasteleyn}
(as in the proof of \cref{thm:surjectivity}),
$\twMtx(Y)$ is a Kasteleyn matrix associated to the bipartite graph 
obtained from $\hatG(\T)$, as in \cref{fig:Kasteleyn}.
(Some of the twistor coordinates
$\lrangle{YZ_1 Z_i}$ of $Y$ may vanish, in which case we just 
erase some of the edges of the bipartite graph.)
If none of the twistor coordinates vanish, \cref{thm:surjectivity}
implies that all nonzero minors of $\twMtx(Y)$ have the same sign.
Erasing some of the edges of the bipartite graph preserves this property.
We now claim that $\twMtx$ has full-rank.
To see this, note that if we let $I:=\{i_1,\dots,i_k\}$,
then $p_I(\twMtx)\neq 0$.  This is because when we restrict to columns 
$i_1,\dots,i_k$,  the only nonzero entry in column 
$i_{\ell}$ (for $1\leq \ell \leq k$) is the entry
	$\lrangle{Y Z_1 Z_{i_\ell+1}}$ in row $\ell$, which has sign $(-1)^{\ell}$.
Therefore $\twMtx \in \Grk$, so $Y'=\twMtx Z$ has full rank. 
\end{proof}

\begin{corollary}\label{cor:cyclic}
Fix $k<n$, $m=2$, and $Z\in \Mat^{>0}_{n,k+2}$. For any $a$ with
	$1\leq a \leq n$,  we define
	\begin{align*}
		\mathcal{F}^{\circ,a}_{n,k,2}(Z)=	\{Y\in Gr_{k,k+2} \ &\vert \  
	\langle Y Z_i Z_{i+1}\rangle>0 \text{ for }1 \leq i \leq n-1,
		\text{ and }   \langle Y Z_n \hat{Z}_1 \rangle >0, \text{ and}\\
		& 
		\var({\langle Y Z_a Z_{a+1}\rangle, 
		\dots \lrangle{Y Z_a Z_n}, \lrangle{Y Z_a \hat{Z}_1},\dots,
		\lrangle{Y Z_a \hat{Z}_{a-1}}})=k.\}
	\end{align*}
	We have $\mathcal{A}_{n,k,2}(Z) = 
	\overline{\mathcal{F}^{\circ,a}_{n,k,2}(Z)}= 
	\overline{\mathcal{F}^{\circ}_{n,k,2}(Z)}.$
\end{corollary}
\begin{proof}
The proof is nearly the same as the one for \cref{thm:main1}. To adapt it,
	 in the second paragraph of that proof, 
	we choose the sequence $i_0<i_1<\dots < i_k \leq n-1$ 
	based on examining the signs of 
	the sequence 
		$({\langle Y Z_a Z_{a+1}\rangle, 
		\dots \lrangle{Y Z_a Z_n}, \lrangle{Y Z_a \hat{Z}_1},\dots,
		\lrangle{Y Z_a \hat{Z}_{a-1}}})$.
		We then use the bicolored triangulation whose $k$ black triangles have vertices
		$\{a, i_{\ell}, i_{\ell}+1\}$ for $1 \leq \ell \leq k$.
\end{proof}

By combining \cref{prop:twistor} with \cref{thm:main1} (or \cref{cor:cyclic}), 
we can obtain a
sign-flip characterization of the $\mathcal{B}$-amplituhedron $\mathcal{B}_{n,k,2}(W)$ (see \cref{def:B}).

\begin{corollary} \label{cor:G}
Fix $k<n$ and $W\in \Gr_{k+2,n}^{>0}$.
Let 
	\begin{align*}
		\mathcal{G}^{\circ}_{n,k,2}(W):=	
		\{ X\in \Gr_{2}(W) \ & \vert \  
		 p_{i, i+1}(X)>0 \text{ for }1 \leq i \leq n-1,
		\text{ and }   p_{n,\hat{1}}(X)  >0,\\
		&\text{ and } 
		\var((p_{12}(X), 
		p_{13}(X),  \dots 
		p_{1n}(X))=k\}
	\end{align*}
	where for $i<j$, $p_{j \hat{i}}(X):=(-1)^k p_{ij}(X)$.  
	Then $\mathcal{B}_{n,k,2}(W) = \overline{\mathcal{G}^\circ_{n,k,2}(W)}.$
\end{corollary}

The set $\mathcal{G}^{\circ}_{n,k,2}(W)$ should agree with the set $\mathcal{G}$ from 
\cite[Prop 3.20]{karpwilliams} when $m=2$.

\section{Cluster algebras and the amplituhedron} \label{sec:cluster}
In this section, we discuss two aspects of how 
amplituhedra and their positroid tiles are related to 
\emph{cluster algebras} \cite{ca1}.  We assume the reader
has some familiarity with the basics of cluster algebras and cluster
varieties, as in 
\cite{ca123,  GHK}.

In \cref{subsec:ClustAdj}, we will discuss
the \emph{cluster adjacency} conjecture, which says that facets of 
a positroid tile for $\AA_{n, k, m}$ should be naturally associated to 
a collection of compatible cluster variables in $\Gr_{m,n}$.  
We will prove this 
conjecture for $m=2$ in \cref{thm:cluster}.

In \cref{GTcluster}, we will prove a  related but more geometric
statement, which illustrates a new phenomenon in the setting of 
amplituhedra: we will associate a cluster variety to each 
positroid tile of $\mathcal{A}_{n,k,2}(Z) \subset \Gr_{k, k+2}$, 
and we will show that the positroid tile
is the totally positive part of that cluster variety.  
We then have the strange phenomenon that the 
amplituhedron 
$\mathcal{A}_{n,k,2}(Z)$ can be subdivided into ${n-2 \choose k}$
$2k$-dimensional 
positroid tiles,
each of which is the totally positive part of a cluster variety.
(In contrast, most other geometric objects with a cluster structure 
have a unique top-dimensional stratum which is the totally 
positive part of a cluster variety.)

\subsection{Cluster adjacency} \label{subsec:ClustAdj}
In 2013, Golden--Goncharov--Spradlin--Vergu--Volovich \cite{Golden:2013xva}  established that singularities of scattering amplitudes of planar $\mathcal{N}=4$ SYM at loop level can be described using cluster algebras. In particular, a large class of loop amplitudes can be expressed in terms of \emph{multiple polylogarithms} whose branch points are encoded in the so-called \emph{symbol alphabet}. Remarkably, elements of this alphabet were observed to be 
$\mathcal{X}$-cluster variables for $\Gr_{4,n}$. This enabled the powerful program of \emph{cluster bootstrap} which pushed both the computation and the understanding of the mathematical structure of scattering amplitudes beyond the frontiers, see \cite{Caron-Huot:2020bkp} for a recent review. 
In 2017 Drummond--Foster--G\"urdo\u gan \cite{Drummond:2017ssj} enhanced the connection with cluster algebras by observing phenomena they called \emph{cluster adjacencies}, related to compatibility of cluster variables. Shortly thereafter, they conjectured that the terms in tree-level $\mathcal{N}=4$ SYM amplitudes coming from the BCFW recursions are rational functions whose poles correspond to compatible cluster variables of the cluster algebra associated to $\Gr_{4,n}$ \cite{Drummond:2018dfd}. In \cite{Mago:2019waa}, this conjecture was 
extended to all (rational) \emph{Yangian invariants}, i.e. the `building blocks' of tree-level amplitudes and leading singularities of planar $\mathcal{N}=4$ SYM. 

These conjectures can be reformulated  in terms of the geometry of the 
amplituhedron $\AA_{n,k,m}(Z)$ and the {facets} of its positroid tiles.  This version of cluster adjacency for the $m=2$ amplituhedron
was studied in \cite{Lukowski:2019sxw}, and  
for the $m=4$ amplituhedron in \cite{Gurdogan:2020tip},
where the authors made connections with \emph{leading} and \emph{Landau singularities}.

For each positroid tile $Z_{\hat{G}(\T)}$ of $\AA_{n,k,2}(Z)$, the corresponding Yangian invariant is a rational function\footnote{Within the framework of \emph{positive geometries}, this is the \emph{canonical function} of  
$Z_{\hat{G}(\T)}$
\cite{Arkani-Hamed:2017tmz}.} in the twistor coordinates.
A defining property of this function is that it has a simple pole at $\langle Yij\rangle=0$ if and only if there is a facet of $Z_{\hat{G}(\T)}$ lying on the hypersurface $\{\langle Yij\rangle=0\}$. Let us consider the collection 
$\{\langle Yij\rangle\}_{\hat{G}(\T)}$
of twistor coordinates corresponding to such poles, and identify it via 
\cref{prop:twistor} with a collection of 
 Pl\"ucker coordinates $\{p_{ij}(X)\}_{\hat{G}(\T)}$ in the Grassmannian $\Gr_{2,n}(\C)$ 
 (with $Y$ the row span of $X^{\perp}Z$).  These Pl\"ucker coordinates are cluster variables of the 
 type $A_{n-3}$ cluster algebra associated to $\Gr_{2,n}(\C)$ \cite{Fomin_2003}. 
 In this cluster algebra, $p_{a b}$ and $p_{c d}$
are \emph{compatible} cluster variables if the arcs $a \to b$ and $c \to d$ in the polygon $\mathbf{P}_n$ do not cross.
The \emph{$m=2$ cluster adjacency conjecture} of
{\L}ukowski--Parisi--Spradlin--Volovich  
\cite{Lukowski:2019sxw} 
says that the cluster variables of $\Gr_{2,n}(\C)$ associated to the facets of 
a positroid tile of $\AA_{n,k,2}$ are compatible.
We generalize this conjecture as follows.

\begin{conj}\label{conj:adjacency}
	Let $Z_{\hat{G}(\T)}$ be a positroid tile of $\mathcal{A}_{n,k,2}(Z)$.
Each facet lies on a hypersurface
$\langle Yij\rangle=0$,
and the collection of Pl\"ucker coordinates 
	$\{p_{ij}\}_{\hat{G}(\T)}$ 
	corresponding to facets is a collection of compatible cluster variables
for $\Gr_{2,n}(\C)$.

	Moreover, if $p_{hl}$ is
	 compatible with 
	$\{p_{ij}\}_{\hat{G}(\T)}$, then 
	$\lrangle{Yhl}$  has a fixed sign on $Z^{\circ}_{\hat{G}(\T)}$. 
\end{conj}

We will prove \cref{conj:adjacency} in \cref{thm:cluster}. 

We now generalize this conjecture for other $m$.
The relevant cluster algebra is the homogeneous coordinate ring of $\Gr_{m,n}(\C)$
\cite{Scott}.  
Each cluster variables is a polynomial
 $Q(p_I)$ in the ${n \choose m}$ Pl\"ucker coordinates. 
Each facet of a positroid tile $Z_{\pi}$ of $\AA_{n,k,m}$ 
lies on a hypersurface defined by the vanishing of some (often non-linear)
polynomial $Q(\langle Y Z_I \rangle)$ in the ${n \choose m}$ twistor coordinates, 
where  we write $\lrangle{Y Z_I}$ for $\lrangle{Y Z_{i_1} \dots Z_{i_m}}$.

\begin{conj}[Cluster adjacency for  $\mathcal{A}_{n,k,m}$]\label{conj:cluster}
Let $Z_{\pi}$ be a positroid tile of the amplituhedron $\AA_{n,k,m}(Z)$ and let
\begin{equation*}
\Facet(Z_{\pi}):=\{Q(p_I) \ \vert \ 
	       \text{a facet of $Z_{\pi}$ lies on 
	       the hypersurface $Q(\lrangle{Y Z_I})=0$}\},
\end{equation*}
	where $Q$ is a polynomial in the ${n \choose m}$ Pl\"ucker coordinates.
 Then 
 \begin{enumerate}
	 \item  Each $Q\in \Facet(Z_{\pi})$ is a cluster variable for $\Gr_{m,n}(\C)$.
	 \item 
		 $\Facet(Z_{\pi})$ consists of compatible cluster variables.
 	\item If $\tilde Q$ is a cluster variable compatible with
 	$\Facet(Z_{\pi})$, 
 	the polynomial $\tilde Q(\lrangle{YZ_I})$ in twistor coordinates has a fixed sign
 	on $Z^{\circ}_{\pi}$.
 \end{enumerate}

\end{conj}

Positroid tiles for $m=4$ 
are not yet 
 characterized.\footnote{conjecturally they are images of positroid cells with \emph{intersection number} one \cite{Gurdogan:2020tip}, which correspond to `rational' Yangian invariants \cite{Mago:2019waa}.} 
	In general, the polynomials appearing in the sets $\Facet(Z_{\pi})$ are unknown.  Moreover, for $n \geq 8$, there is no classification of the cluster variables of $\Gr_{4,n}$. Also note that 
	Part
	$(1)$ of \cref{conj:cluster} is in a similar spirit to \cite[Conjecture 19.8]{lam} .

\subsection{Positroid tiles are totally positive parts of cluster varieties}\label{GTcluster}

In this subsection, we build a cluster variety $\mcv_{\overline{\T}}$ in $\Gr_{k, k+2}(\C)$ for each positroid tile $\gt{\hatG(\overline{\T})}$ of $\AA_{n, k, 2}$. Each bicolored triangulation represented by $\overline{\T}$ gives a \emph{seed torus} of $\mcv_{\overline{\T}}$. We will show that the positroid tile 
$\gto{\hatG(\overline{\T})}$ 
is exactly the \emph{totally positive part} of $\mcv_{\overline{\T}}$.

Fix a bicolored subdivision $\overline{\T}$ of type $(k,n)$, with black polygons $P_1, \dots, P_r$. For each black polygon $P_i$, fix an arc $h_i \to j_i$ with $h_i < j_i$ in the boundary of $P_i$. We call this the \emph{distinguished boundary arc} of $P_i$. 
   We will build $\mcv_{\overline{\T}}$ by defining seeds in the field of rational functions on $\Gr_{k, k+2}(\C)$.

\begin{defn}[Cluster variables]
	Let $a \to b$ with $a<b$ be an arc which is contained in a black polygon $P_i$ and is not the distinguished boundary arc $h_i \to j_i$. We define
	\[x_{ab}:=\frac{(-1)^{\area(a \to b)}\lrangle{Y ab}}{(-1)^{\area(h_i\to j_i)}\lrangle{Y h_i j_i}}.\]
	This is a rational function on $\Gr_{k, k+2}(\C)$ and is regular away from the hypersurface $\{\lrangle{Y h_i j_i}~=~0\}$.
\end{defn}

\begin{defn}[Seeds]
	Let $\T$ be a bicolored triangulation represented by $\overline{\T}$. The \emph{quiver} $Q_{\T}$ is obtained as follows:
	\begin{itemize}
		\item Place a frozen vertex on each boundary arc of $P_1, \dots, P_r$ and a mutable vertex on every other black arc of $\T$.
		\item If arcs $a \to b$, $b \to c$, $c \to a$ form a triangle, put arrows between the corresponding vertices, going clockwise around the triangle. Then delete the frozen vertex on the distinguished boundary arc (and all arrows involving this vertex) and arrows connecting two frozen vertices.
	\end{itemize}
	
	We label the vertex of $Q_\T$ on arc $a \to b$ of $\T$ with the function $x_{ab}$. The collection of vertex labels is the \emph{(extended) cluster} $\textbf{x}_{\T}$. The pair 
	$(Q_{\T},\textbf{x}_\T)$ 
	is the \emph{seed} $\Sigma_\T$.
\end{defn}

Note that there are no frozen variables corresponding to the distinguished boundary arcs, and the cluster $\mathbf{x}_\T$ has size $2k$. Note also that $\Sigma_\T$ does not depend on the triangulation of the white polygons of $\overline{\T}.$ See \cref{fig:seed} for an example. 

\begin{figure}
	\includegraphics[width=0.3\textwidth]{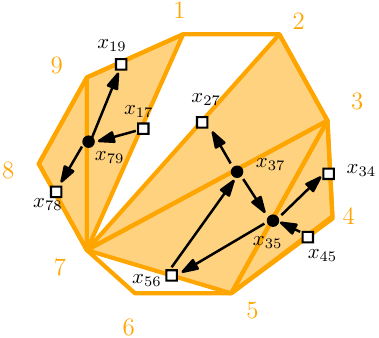}
	\caption{\label{fig:seed} In orange, a bicolored triangulation $\T$. In black, the seed $\Sigma_\T$. The distinguished boundary arcs are $2\to3$ and $8 \to 9$. }
\end{figure}

Now we show that each seed gives a \emph{seed torus} in $\Gr_{k, k+2}(\C)$.

\begin{prop}\label{prop:tori} Let $\T$ be a bicolored triangulation represented by $\overline{\T}$.
	Consider the Zariski-open subset
	\[\mcv_\T:= \left \{Y \in \Gr_{k, k+2}(\C):  \prod_{a \to b \text{ black arc of }\T}\lrangle{Y a b} \neq 0 \right\}.\]
	This is birational to an algebraic torus of dimension $2k$, with 
	field of rational functions $\C(\mathbf{x}_\T)$, 
	the field of rational functions in the cluster $\mathbf{x}_\T$.
\end{prop}

\begin{proof}
The main idea is that \cref{cor:bijection}---which gave a 
	 bijection between
	$$\gto{\hatG(\overline{\T})}= \{Y \in \Gr_{k, k+2}(\R): \text{for all arcs } i \to j \text{ of } \T \text{ with } i <j,  (-1)^{\area(i \to j)} \lrangle{Y i j}>0 \}$$ and $(\R_{>0})^{2k}$---extends directly to give a birational morphism from
	$\mcv_\T$ to $(\C^*)^{2k}$.
        When we let the
	edge weights $\alpha_i, \beta_i, \gamma_i$ (used to define 
	matrix $M$ in \eqref{eq:M}) range over all nonzero complex
	numbers, the set of $k \times n$ matrices we get 
	sweeps out the \emph{open Deodhar stratum}\footnote{Parameterizations 
 of Deodhar strata in flag varieties are given 
	in \cite{MR}; in the Grassmannian, these can be equivalently 
	parameterized using weighted networks, as shown in \cite{talaska_williams}.}  
	$D_{\T}$  
	as opposed to the positroid cell $S_{\hat{G}(\T)}$.
	That is, the stratum $D_\T \subset \Gr_{k, n}(\C)$ consists of subspaces represented by the matrices $M_\T(\Aalpha, \Bbeta, \Ggamma)$ of \eqref{eq:M},
 where $(\Aalpha, \Bbeta, \Ggamma)$ vary over $(\C^*)^{3k}$ rather than $(R_{>0})^{3k}$.

	Let us define the map
	\begin{align*}
		\mcv_\T &\to (\C^*)^{2k}\\
		Y & \mapsto \mathbf{x}_\T(Y).
	\end{align*}

To see that the map is surjective onto a Zariski-open subset of $(\C^*)^{2k}$, consider some $2k$-tuple of nonzero complex numbers $\mathbf{q}_\T$. Define a weight vector $(\Aalpha, \Bbeta, \Ggamma)$ for $\hatG(\T)$, where for a triangle $\{a_i, b_i, c_i\}$, the weights are \[\alpha_i=q_{b_ic_i}, \quad \beta_i=q_{a_ic_i}, \quad \gamma_i=q_{a_ib_i}.\] 
(As usual, if $a \to b$ is a distinguished boundary arc, we take $q_{ab}=1$.) Let $C:=M_\T(\Aalpha, \Bbeta, \Ggamma).$ 

The matrix $C$ lies in the Deodhar stratum $D_\T$ and so has full rank. Let $Y:= CZ$. Consider an arc $a \to b$ of $\T$ which is in a black polygon $P$. From the proof of \cref{prop:niceRep}, we have $$\lrangle{Yab}=(-1)^{\area(a\to b)} q_{ab} \cdot \mathcal{Q}_{P}$$ where $\mathcal{Q}_P$ is a polynomial with positive coefficients in the $q_{ij}$'s and the minors of $Z$, and depends only on the polygon $P$. $\mathcal{Q}_P$ is generically
nonzero, in which case it is easy to check that $x_{ab}(Y)=q_{ab}$. Moreover, in this case, $Y$ is a full-rank matrix, since it has at least one non-zero twistor coordinate.

Now, suppose $\mathbf{q}_\T$ lies in the open subset $O$ of $(\C^*)^{2k}$ where the polynomials $\mathcal{Q}_P$ are nonzero for all polygons $P$. Then $Y$, as defined above, lies in $\mcv_\T$ and maps to $\mathbf{q}_\T$.

The map is injective on the preimage of $O$. Indeed,  
pick $Y, Y' \in \mcv_\T$ which map to $\mathbf{q}_\T \in O$. Consider the twistor coordinate matrices $C:=\twMtx(Y)$ and $C':=\twMtx(Y')$. \cref{prop:kindaInverse} works equally well for matrices with complex entries, so the rowspans of $CZ$ and $C'Z$ are contained in $Y$ and $Y'$, respectively. On the other hand, the rows of $C$ and $C'$ can both be rescaled to obtain the matrix $M_\T(\Aalpha, \Bbeta, \Ggamma)$ defined above, so the rowspans of $CZ$ and $C'Z$ are the same. Finally, because of the assumption $\mathbf{q}_\T \in O$, the matrix $CZ$ has some nonzero twistor coordinate and so in particular is full rank. This shows the rowspan of $CZ$ is equal to $Y$ and to $Y'$.
\end{proof}

Next, we verify that the seeds given by different bicolored triangulations are related by mutation.

\begin{prop}\label{prop:mutation}
	Let $\T$ be a bicolored triangulation represented by $\overline{\T}$ and let $a \to b$ correspond to a mutable vertex of $\Sigma_\T$. Let $\T'$ be related to $\T$ by flipping the arc $a \to b$. Then $\Sigma_\T$ and $\Sigma_{\T'}$ are related by mutation at $x_{ab}$.
	
	The seeds which can be obtained from $\Sigma_\T$ by an arbitrary sequence of mutations are exactly the seeds $\Sigma_{\T'}$ where $\T'$ is represented by $\overline{\T}$.
\end{prop}
In light of \cref{prop:mutation}, we can make the following definition.
\begin{definition}
Let $\T$ be a bicolored triangulation and $\overline{\T}$
the corresponding bicolored subdivision.
We let $\mathcal{A}(\overline{\T})$ denote the cluster algebra
 $\mathcal{A}(Q_{\T},\textbf{x}_\T)$.
\end{definition}
\begin{proof}[Proof of \cref{prop:mutation}]
On the level of quivers, the first statement follows immediately from the well-known combinatorics of type A cluster algebras. 
	
	Say the arc $a \to b$ is in triangles $\{a<u< b\}$ and $\{a<b< v\}$ in $\T$, so $a \to b$ is flipped to $u \to v$ (the argument is analogous if instead $v<a$). We need to check that, in the field of rational functions on $\Gr_{k, k+2}(\C)$, we have
	\[x_{ab}x_{uv}=x_{au}x_{bv}+x_{av}x_{ub}\]
	(where $x_{h_i j_i}$ is defined to be $1$). This follows easily from the 3-term Pl\"ucker relations for the corresponding twistor coordinates.
	
	The second statement follows  from the fact that triangulations are
	flip-connected.
\end{proof}

Together, \cref{prop:tori} and \cref{prop:mutation} tell us that the union of the seed tori is a cluster variety in $\Gr_{k, k+2}(\C)$.

\begin{thm}\label{thm:clustvar}
	Let $\overline{\T}$ be a bicolored subdivision of type $(k, n)$. Then 
	\[\mcv_{\overline{\T}}:= \bigcup_\T \mcv_\T\]
	is a cluster variety in $\Gr_{k, k+2}(\C)$, where the union is over bicolored triangulations represented by $\overline{\T}$.
	We call $\mcv_{\overline{\T}}$ the \emph{amplituhedron (cluster) 
	variety}\footnote{This is closely related to the 
	amplituhedron variety defined in \cite{LamStanley}.}  of 
	 $Z_{\hat{G}(\T)}$.

	Moreover, the positive part 
	\[\mcv_{\overline{\T}}^{>0}:=\{Y \in \mcv_{\overline{\T}}: x_{ab}(Y)>0 \text{ for all cluster variables } x_{ab}\}\]
	 is equal to the positroid tile $\gto{\hatG(\overline{\T})}$.
\end{thm}

\begin{proof}
	The first assertion follows from the definition of cluster variety and Propositions~\ref{prop:tori} and \ref{prop:mutation}. 
	
	For the second statement, note that by \cref{thm:surjectivity}, points of $\gto{\hatG(\overline{\T})}$ are in the positive part $\mcv_{\overline{\T}}^{>0}$. To see the opposite inclusion, take a point $Y$ in the positive part and choose a bicolored triangulation $\T$ represented by $\overline{\T}$. Let $C:=\twMtx(Y)$ be the twistor coordinate matrix of $Y$.
	
	If row $i$ of $C$ corresponds to a triangle in $\T$ lying in polygon $P_i$, rescale row $i$ by $(-1)^{\area(h_i \to j_i)}/\lrangle{Yh_ij_i}$. Call the resulting matrix $C'$. Because $x_{ab}>0$ for all arcs $a \to b$ of $\T$, the entry $\lrangle{Yab}$ of $C$ has been rescaled to a real number with sign $(-1)^{\area(a \to b)}$. By the same argument as the last paragraph of \cref{thm:surjectivity}, $C'$ (and thus $C$) represents an element of $S_{\hatG(\T)}$. This, together with \cref{prop:kindaInverse}, implies that $Y\in \gto{\hatG(\overline{\T})}$.
\end{proof}

\begin{theorem}
	The cluster algebra 
	$\mathcal{A}(\overline{\T})$ equals the upper cluster algebra 
	$\overline{\mathcal{A}}(\overline{\T})$.
	 If the bicolored subdivision $\overline{\T}$ has 
	black polygons $P_1,\dots,P_r$, where $P_i$ has $n_i$ vertices, then 
	$\mathcal{A}(\overline{\T})$ is a finite type cluster algebra of Cartan-Killing type
	$A_{n_1-2} \times \dots \times A_{n_r-2}$.
\end{theorem}
\begin{proof}
The quiver we are associating to each bicolored triangulation is a disjoint union
quivers associated to a triangulated $n_i$-gon, or equivalently to 
$\C[\Gr_{2,n_i-2}]$. Notice that for each one of these quivers, the corresponding exchange matrix has full $\Z$-rank (the argument is very similar to the one in 
\cite[Proof of Theorem 5.3.2]{ca45}).

It is well known that the quiver associated to a triangulated
$r$-gon has Cartan-Killing type $A_{r-2}$ \cite{ca1}.  This implies that 
	$\mathcal{A}(\overline{\T}) 
	= \mathcal{A}(Q_{\T},\textbf{x}_\T)$ 
has type $A_{n_1-2} \times \dots \times A_{n_r-2}$.

Because our quiver is just a disjoint union of type $A$ quivers (one from each $P_i$), our
cluster algebra has an acyclic seed.  
Moreover, since the exchange matrix corresponding to each of these type $A$ quivers is full rank,
	$\mathcal{A}(\overline{\T})$ also has a full $\Z$-rank exchange matrix.

Using \cite[Proposition 1.8 and Remark 1.22]{ca3}, the fact that 
	$\mathcal{A}(\overline{\T})$ has an acyclic seed and also has a full rank exchange
	matrix implies that the upper cluster algebra 
$\overline{\mathcal{A}}({\T})$ equals the cluster algebra
	$\mathcal{A}(\overline{\T})$. 
\end{proof}

\begin{remark}
Given what we've proved, one can make an argument as in the proof of \cite[Theorem 2.10]{ca3} that 
	$\mathcal{A}(\overline{\T})$ is the coordinate ring of 
	the amplituhedron variety 
	$\mcv_{\overline{\T}}$ and 
	also the closely related variety
	$$V_{\overline{\T}}:=\{Y \in \Gr_{k,k+2}(\C)  \ \vert \ \lrangle{Yij} \neq 0 \text{ for }
	h\to j \text{ a boundary arc of a black polygon of $\overline{\T}$.}\}.$$
\end{remark}

\section{Background on the hypersimplex, T-duality, and positroid tilings}\label{sec:background2}

In \cite{LPW}, a surprising parallel was found between the 
amplituhedron map $\tilde{Z}$ on $\Gr^{\geq 0}_{k, n}$ and the moment map $\mu$ on $\Gr^{\geq 0}_{k+1, n}$. A correspondence called \emph{T-duality} was used to relate Grasstopes in the amplituhedron $\AA_{n, k, 2}$ to positroid polytopes in the hypersimplex $\Delta_{k+1, n}$.
In the second part of this paper, we further explore this relationship and prove some of the conjectures of \cite{LPW}. We present relevant background here\footnote{We
 use `$k+1$' instead of `$k$' here in order to match conventions of later sections.}.

\subsection{The hypersimplex $\Delta_{k+1,n}$ and positroid polytopes.}

Throughout, for $x \in \R^n$ and $I \subset [n]$, we use the notation $x_I:=\sum_{i \in I} x_i$. 

\begin{definition}[The hypersimplex]
Let $e_I := \sum_{i \in I} e_i \in \R^n$, where
$\{e_1, \dotsc, e_n\}$ is the standard basis of $\RR^n$.
The \emph{$(k+1,n)$-hypersimplex} 
$\Delta_{k+1,n}$
 is the convex hull of the points $e_I$ where $I$ runs over $\binom{[n]}{k+1}$.
\end{definition}

\begin{remark}\label{rem:hypercube} The hypersimplex $\Delta_{k+1,n}$ is obtained by intersecting the unit hypercube $\mbox{\mancube}_{n}$ with the hyperplane $x_{[n]}=k+1$.
Alternatively, under the projection $P:(x_1,\dots,x_n) \mapsto (x_1,\dots,x_{n-1})$,
$\Delta_{k+1,n}$ is linearly equivalent to
$$\tilde{\Delta}_{k+1,n} := \{(x_1,\dots,x_{n-1}) \ \vert \
0 \leq x_i \leq 1; k \leq x_{[n-1]} \leq k+1\} \subset \RR^{n-1}.$$
That is,
$\tilde{\Delta}_{k+1,n}$ is the slice of $\mbox{\mancube}_{n-1}$ between the hyperplanes
$x_{[n-1]} = k$ and $x_{[n-1]}= k+1$.
\end{remark}

The torus $T = (\C^*)^n$ acts on $\Gr_{k+1,n}$ by scaling the columns of 
a matrix representative $A$.  (This is really an $(n-1)$-dimensional torus
since the Grassmannian is a projective variety.)  
We let $TA$ denote the orbit of $A$ under the action of $T$, and 
$\overline{TA}$ its closure.  

The \emph{moment map} from the Grassmannian $\Gr_{k+1,n}$ to $\R^n$
is defined as follows.
\begin{definition}[The moment map]
	Let $A$ be a $(k+1) \times n$ matrix representing a point of 
	$\Gr_{k+1,n}$.
	The \emph{moment map} 
	$\mu: \Gr_{k+1,n} \to \R^n$ is defined by 
	$$\mu(A) = \frac{ \sum_{I \in \binom{[n]}{k+1}} |p_I(A)|^2 e_I}
	{\sum_{I \in \binom{[n]}{k+1}} |p_I(A)|^2}.$$
\end{definition}
It is well-known that the image of the Grassmannian $\Gr_{k+1,n}$ under the moment
map is the hypersimplex
$\Delta_{k+1,n}$. 
If one restricts the moment map to $\Gr_{k+1,n}^{\geq 0}$ then the image is again the hypersimplex
$\Delta_{k+1,n}$ 
\cite[Proposition 7.10]{tsukerman_williams}.

In general, it follows from classical work of 
Atiyah \cite{A:82} and Guillemin-Sternberg \cite{GS} that the image
$\mu(\overline{TA})$ is a convex polytope, whose vertices
are the images of the torus-fixed points, i.e. the vertices are the
points $e_I$ such that $p_I(A) \neq 0$ and $p_J(A)=0$ for $J \not = I$. This motivates the notion of \emph{matroid polytope}.  
Recall that any full
rank $(k+1)\times n$ matrix $A$ gives rise to a matroid 
$\M(A)=([n],\B)$, where $\B = \{I \in \binom{[n]}{k+1} \ \vert \ p_I(A) \neq 0\}$.

\begin{definition}\label{def:mpolytope}
Given a matroid $\M=([n],\B)$, the (basis) \emph{matroid polytope} $\Gamma_\M$ of $\M$ is the convex hull of the indicator vectors of the bases of~$\M$:
\[
\Gamma_\M := \convex\{e_B : B \in \B\} \subset \RR^n.
\]
\end{definition}



%

Matroid polytopes also have a straightforward description in terms of inequalities.

\begin{proposition}[\cite{welsh}]\label{prop:inequalitiesmatroids}
Let $\M = ([n], \B)$ be any matroid of rank $k+1$, and let $r_\M:2^{[n]} \to \Z_{\geq 0}$ be its rank function. Then the matroid polytope $\Gamma_\M$ can be described as
	\begin{align*}
		\Gamma_\M &= \{ {\bf x} \in \RR^n : x_{[n] }= k+1, \, x_A \leq r_\M(A) \, \text{ for all $A \subset [n]$} \}\\
		\Gamma_\M &= \{ {\bf x} \in \RR^n :x_{[n] }= k+1, \, x_A \geq k+1- r_\M([n]\setminus A) \, \text{ for all $A \subset [n]$} \}.
	\end{align*}
\end{proposition}


Here, we are interested in \emph{positroid polytopes}, that is, matroid polytopes $\Gamma_\M$ where $\M$ is a positroid. They arise as $\mu(\overline{TA})$ where $A$ is a totally nonnegative matrix. Of more interest to us, they can also be obtained as moment map images of positroid cells. 
%

\begin{proposition}\cite[Proposition 7.10]{tsukerman_williams}\label{prop:moment}
Let $\M$ be the positroid associated to the positroid cell $S_{\pi}$.  Then 
$\Gamma_{\M} = \mu(\overline{S_{\pi}}) = \overline{\mu(S_{\pi})}$.
\end{proposition}

We will be particularly interested in the  cells 
on which the moment map is injective. 

\begin{definition}[Positroid Polytopes]\label{def:Gamma}
Given a positroid cell $S_{\pi}$ of 
$\Gr_{k+1,n}^{\ge 0}$, we let 
$\Gamma^{\circ}_{\pi} = \mu(S_{\pi})$ and 
$\Gamma_{\pi} = \overline{\mu(S_{\pi})}$, and we refer to 
$\Gamma^{\circ}$ and $\Gamma_{\pi}$ as \emph{open positroid polytopes} and 
\emph{positroid polytopes}, respectively.
	We call 
	$\Gamma_{\pi}$ a 
  \emph{positroid tile} 
	for $\Delta_{k+1,n}$ 
	if $\dim(S_\pi) =n-1$, and 
	 $\mu$ is injective on $S_\pi$.
\end{definition}



\begin{theorem}[Characterization of positroid tiles of $\Delta_{k+1,n}$]\label{prop:homeo2}\cite[Propositions 3.15, 3.16]{LPW}
Consider a positroid cell $S_{G} \subset \Gr_{k+1,n}^{\geq 0}$, with $G$ a reduced plabic graph.
	Then the moment map is injective on $S_G$ if and only if $G$ is a forest.  When $G$ is a forest, $\mu$ is moreover a stratification-preserving
	 homeomorphism from $\overline{S_{G}}$ to  the polytope
	$\Gamma_G \subset \mathbb{R}^n$.
We have $\dim S_{G}
=\dim \Gamma_{G} = n-c,$ where $c$ is the number of connected components of $G$. 

In particular, given an $(n-1)$-dimensional cell $S_{G} \subset \Gr_{k+1,n}^{\geq 0}$, $\Gamma_G$ is a positroid tile for $\Delta_{k+1,n}$ if and only if $G$ is a tree.
\end{theorem}

%
%
%

\subsection{T-duality and positroid tilings}

Recall the definition of positroid tiling from \cref{def:tri}.  Specializing to 
 $\mathcal{A}_{n,k,2}(Z)$, we get the following.
\begin{definition}[Positroid tilings of $\mathcal{A}_{n,k,2}$]\label{def:dissectionAmp}
	Let $\mathcal{C} = \{\gt{\pi}\}$ be a collection of Grasstopes, with $\{S_{\pi} \}$ positroid cells of $\Gr^{\geq 0}_{k,n}$.  
We say that 
	$\mathcal{C}$ is 
	a \emph{positroid tiling} of $\mathcal{A}_{n,k,2}(Z)$
	if:
	\begin{itemize}
		\item each Grasstope $\gt{\pi} $ is a positroid tile (i.e. $\widetilde{Z}$ is injective on $S_{\pi}$ and $\dim \gt{\pi} = 2k$);
		\item pairs of distinct open Grasstopes $\gto{\pi}$ and $\gto{\pi'}$ in the collection are disjoint; 
		\item $\cup_{\pi} \gt{\pi} = \mathcal{A}_{n,k,2}(Z)$. 
	\end{itemize}
\end{definition}

\begin{remark}
Alternatively, one could define a positroid tiling as coming
	from a collection $\{S_{\pi}\}$ of cells
	such that $\{Z_{\pi}\}$ is a positroid tiling (as above)
	for \emph{all} choices of $Z$.
We use \cref{def:dissectionAmp} here
since some objects we define will be sensitive to the choice of $Z$. 
\end{remark}

In the case of the hypersimplex, a positroid tiling is as follows.

\begin{definition}[Positroid tilings of $\Delta_{k+1,n}$]\label{def:dissectionSimp}
	Let $\mathcal{C} = \{\Gamma_{\pi}\}$ be a collection 
	of positroid polytopes, with $\{S_{\pi}\}$ positroid cells of $\Gr^{\geq 0}_{k+1,n}$. We say that 
	$\mathcal{C}$ is a \emph{positroid tiling} of $\Delta_{k+1,n}$ if:
	\begin{itemize}
		\item each $\Gamma_\pi$ is a positroid tile ($\mu$ is injective on $S_{\pi}$, and $\dim S_{\pi} = n-1$);
		\item pairs of distinct open positroid polytopes $\Gamma^{\circ}_{\pi}$ and $\Gamma^{\circ}_{\pi'}$ in the collection are disjoint;
		\item $\cup_{\pi} \Gamma_{\pi} = \Delta_{k+1,n}$.
	\end{itemize}
\end{definition}

\begin{remark}
``Positroid tiling" differs slightly from ``positroid triangulation" in \cite{LPW}.
\end{remark}

By \cref{prop:homeo2},  the positroid tiles of $\Delta_{k+1, n}$ are the
positroid polytopes $\Gamma_G$ where $G$ is a plabic tree. 
And by \cref{thm:allGTs},  the positroid tiles of $\AA_{n, k, 2}(Z)$ are 
the Grasstopes $\gt{\hatG(\T)}$ for $\T$ a bicolored subdivision of type $(k,n)$.
In \cite{LPW}, it was conjectured that positroid tiles 
and the two notions of positroid tiling are 
related by a very simple correspondence, called \emph{T-duality}.

\begin{definition}[T-duality on decorated permutations]\label{hatmap}
	Let  $\pi=a_1 a_2\dots a_n$ be a loopless decorated permutation (written
	in one-line notation). The \emph{T-dual} decorated permutation is
	$\hat{\pi} : i \mapsto \pi(i-1)$, so that
	$\hat{\pi} = a_n a_1 a_2\dots a_{n-1}$.
	Any fixed points in $\hat{\pi}$ are declared to be loops.
\end{definition}

\begin{rmk}
This map was previously defined in \cite[Definition 4.5]{Karp:2017ouj}) and 
was studied in \cite{LPW}, where it was used to 
	draw parallels between the hypersimplex $\Delta_{k+1,n}$ 
	and the 
	$m=2$ amplituhedron $\mathcal{A}_{n,k,2}(Z)$.  
	The T-duality map was 
	also studied in \cite{posquotients,
	GalCritVar}.
The map $\pi \to \hat{\pi}$ is an $m=2$ version 
of a map that appeared in \cite{abcgpt} for the case $m=4$.
\end{rmk}

\begin{lem}[\protect{\cite[Lemma 5.2]{LPW}}]\label{lem:tDualityBiject}
	The T-duality map $\pi \mapsto \td{\pi}$ is a bijection from loopless decorated permutations of type $(k+1, n)$ to coloopless decorated permutations of type $(k, n)$. That is, the map $S_\pi \mapsto S_{\td{\pi}}$ is a bijection from the set of loopless cells in $\Gr_{k+1,n}^{\geq 0}$ to the set of coloopless cells in $\Grk$.
\end{lem}

The philosophy of \cite{LPW} is that if the moment map behaves well on $S_\pi$, then the $\tilde{Z}$-map behaves well on $S_{\td{\pi}}$. For example, if the image of $S_\pi$ is a positroid tile for $\Delta_{k+1,n}$, then the image of $S_{\td{\pi}}$ is a positroid tile for $\mathcal{A}_{n,k,2}(Z)$ 
\cite[Proposition 6.6.]{LPW}. Moreover, there is a main conjecture involving positroid tilings:

\begin{conj}[\protect{\cite[Conjecture 6.9]{LPW}}]\label{conj:triangCorresp}
	A collection $\{\Gamma_\pi\}$ of positroid polytopes in $\Delta_{k+1, n}$ gives a positroid tiling of $\Delta_{k+1, n}$ if and only if for all $Z \in \Mat_{n,k+2}^{>0}$, the collection $\{\gt{\td{\pi}}\}$ of Grasstopes gives a positroid tiling of $\AA_{n, k, 2}(Z)$. 
\end{conj}

In \cref{sec:Tplabic}, we will prove a number of additional results on T-duality, 
upgrading it to a map on plabic graphs.
We will also prove  \cref{conj:triangCorresp} in \cref{thm:simpTriangGiveAmp}.


\section{T-duality on decorated permutations and plabic graphs}\label{sec:Tduality1}

In this section we prove that T-duality is a poset isomorphism and can be extended to a map on plabic graphs and plabic tilings. 

We refer the reader to \cref{sec:appendix} for the definition of decorated permutations, their affinizations, loops, coloops, etc., as well as details on plabic graphs and trips.

\subsection{T-Duality as a poset isomorphism}
Here we show that the bijection from \cref{lem:tDualityBiject}
 is a poset isomorphism. Abusing notation, in this subsection we use $\pi$, $\nu$ to denote bounded affine permutations rather than decorated permutations.

\begin{prop}[T-duality as a poset isomorphism]\label{prop:posetiso}
T-duality is a codimension-preserving poset isomorphism between 
loopless cells of $Gr^{\geq 0}_{k+1,n}$
and coloopless cells of $\Grk$.  That is, for $\pi, \nu$ loopless decorated permutations of type $(k+1, n)$, $S_\nu \subset \overline{S_\pi}$ if and only if $S_{\td{\nu}} \subset \overline{S_{\td{\pi}}}$. Furthermore,
	$\codim S_{\nu} = \codim S_{\td{\nu}}$. 
\end{prop}
\begin{proof}
We will work with the poset $\Bdkn$ of bounded affine permutations with respect to the Bruhat order \cite{KLS}, which is dual to the poset $Q(k,n)$ (see \cref{def:positroid}). In $\Bdkn$, $\pi \gtrdot \nu$ if $\pi= \tau \circ \nu$ for some transposition $\tau$ and $\inv(\pi)=\inv(\nu)+1$.

Let $\delta: \Z \to \Z$ be the map $i \mapsto i-1$. For loopless $\pi \in \Bound(k+1, n)$, the T-dual of $\pi$ is $\td{\pi}= \pi \circ \delta$. Fix loopless $\pi, \nu \in\Bound(k+1, n)$. 
Note that $\pi$ and $\pi \circ \delta$ have the same length. Further, $\widehat{\tau \circ \nu}=\tau \circ \nu\circ \delta= \tau \circ \td{\nu}$. So $\pi \gtrdot \nu$ if and only if $\td{\pi} \gtrdot \td{\nu}$. 

To extend this beyond cover relations, notice that if $\nu \in \Bound(k+1, n)$ has $\nu(i)=i$ or $\nu(i)= i+n$, then for all $\pi \in \Bound(k+1, n)$ with $\pi > \nu$, we have $\pi(i)=\nu(i)$. In matroidal terms, 
if the positroid $\M_\nu$ has a loop (resp. coloop) at $i$, then so does $\M_\pi$ for all $\M_\pi \subset \M_\nu$. 

Now, $\pi \geq \nu$ if and only if there exists a maximal chain $\pi \gtrdot \pi_1 \gtrdot \cdots \gtrdot \pi_r \gtrdot \nu$. Since $\pi$ is loopless, the observation in the previous paragraph shows that $\pi_i$ is loopless for $i=1, \dots, r$. Since T-duality and its inverse preserve cover relations, we have such a chain if and only if we have the chain $\td{\pi} \gtrdot \td{\pi}_1 \gtrdot \cdots \gtrdot \td{\pi}_r \gtrdot \td{\nu}$ in $\Bdkn$, which is equivalent to $\td{\pi} \geq \td{\nu}$. 

The codimension statement follows from the fact that the codimension of $S_\pi$ in $\overline{S_\nu}$ is the length of any maximal chain from $\pi$ to $\nu$ in $\Bdkn$.
\end{proof}

T-duality can also be defined for arbitrary even $m$, as in \cite[Equation 5.13]{LPW}, and is also of interest for understanding the $m=4$ amplituhedron. As is clear from \cite[Equation 5.13]{LPW}, the T-duality map for even $m$ is a composition of  the ``$m=2$" T-duality map $m/2$ times. \cref{prop:posetiso} also gives us information about this composition.

\begin{definition}
Let $L^r\Gr^{\geq 0}_{k,n}$ be the set of cells $S_{\pi} \subset \Gr^{\geq 0}_{k,n}$ such that $\pi(i) \geq i+r$ for all $i$. Analogously, let us define $CL^{-r}\Gr^{\geq 0}_{k,n}$ to be the set of cells $S_{\nu} \subset \Gr^{\geq 0}_{k,n}$ such that $\nu(i) \leq i+n-r$ for all $i$. Each is ordered by inclusion on the closures of cells.
\end{definition}
\begin{rmk}
The composition of T-duality $r$ times is a well-defined map from $L^r\Gr^{\geq 0}_{k+r,n}$ to $CL^{-r}\Gr^{\geq 0}_{k,n}$. Indeed, if $\pi(i) \geq i+r$, then applying T-duality $s$ times gives a loopless bounded affine permutation for $s=1, \dots, r-1$. 
Moreover, it is easy to see that applying T-duality $r$ times to such a $\pi$ gives a bounded affine permutation $\nu$ with $\nu(i) \leq i+n-r$.
\end{rmk}

\begin{rmk}
The bounded affine permutations labelling cells in $L^r\Gr^{\geq 0}_{k,n}$ ($CL^{-r}\Gr^{\geq 0}_{k,n}$) can be equivalently described in terms of the sets $\tilde{S}(-a,b)$ defined in \cite[Section 2]{parityduality}.
\end{rmk}

From Proposition \ref{prop:posetiso} we immediately have the following:
\begin{prop}
	The composition of T-duality $r$ times gives a poset isomorphism between 
 $L^r\Gr^{\geq 0}_{k+r,n}$ and $CL^{-r}\Gr^{\geq 0}_{k,n}$.  
	\end{prop}
	
\subsection{T-duality as a map on plabic graphs}\label{sec:Tplabic}

T-duality extends to an operation on particular plabic graphs. 

\begin{definition} A reduced plabic 
	graph is called \emph{black-trivalent} (resp. \emph{white-trivalent}) if all of its interior black (resp. white) vertices are trivalent.
\end{definition}

Note that in particular, black-trivalent (white-trivalent) graphs have no black (white) lollipops, so their trip permutations are loopless (coloopless).

Starting from a black-trivalent graph $G$ with trip permutation $\pi$, we now give an explicit construction of a white-trivalent graph $\hatG$ with trip permutation $\td{\pi}$.
This construction appeared first in \cite{Gal18} using plabic tilings. That the construction is bijective (up to certain moves) can be deduced from \cite{Gal18} (see \cite[Proposition 7.15]{GalashinPostWilliams}, \cite{BW},
\cite[Proposition 8.3]{GalCritVar}). Our phrasing of the bijection does not require passing to a plabic tiling, and so streamlines somewhat the presentation of the aforementioned references.

\begin{definition}[T-duality on plabic graphs]\label{defn:plabicTDual} Let $G$ be a reduced black-trivalent plabic graph. The \emph{T-dual} of $G$, denoted $\hat{G}$, is the graph obtained as follows:
\begin{enumerate}
\item In each face $f$ of $G$, place a black vertex $\hat{b}(f)$.
\item``On top of" each black vertex $b$ of $G$, place a white vertex $\hat{w}(b)$;
\item For each black vertex $b$ of $G$ in face $f$, put an edge $\hat{e}$ connecting $\hat{w}(b)$ and $\hat{b}(f)$;
\item Put $\hat{i}$ on the boundary of $G$ between vertices $i-1$ and $i$ and draw an edge from $\hat{i}$ to $\hat{b}(f)$, where $f$ is the adjacent boundary face.
\end{enumerate}
\end{definition}

\vspace{-.4cm}
\begin{figure}[h]
\includegraphics[height=2in]{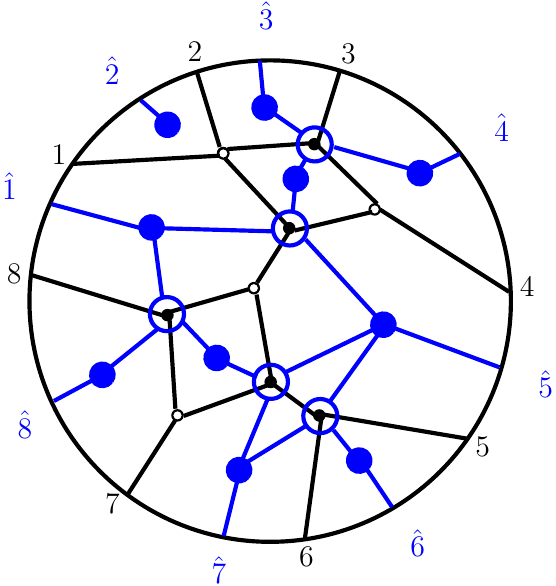}
	\caption{In black: a plabic graph $G$ of type $(4,8)$ with trip permutation $(2,4,7,1,8,5,3,6)$.
	In blue: the T-dual plabic graph $\hat{G}$ of type $(3,8)$ with trip permutation
	$(6,2,4,7,1,8,5,3)$, which is built using \cref{defn:plabicTDual}.}
	\label{fig:TdualityonPlabic}
\end{figure}

\begin{prop} \label{prop:TdualityPlabicGraphs}
Let $G$ be a reduced black-trivalent plabic graph with trip permutation $\pi$. Then $\hat{G}$ is a reduced white-trivalent plabic graph with trip permutation $\hat{\pi}$. 
\end{prop}
\begin{proof}
First observe that since $G$ is black-trivalent,  
 $\hat{G}$ is white-trivalent
	(see \cref{fig:tdualmisc}).
\begin{figure}[h]
	\centering
\includegraphics[height=.8in]{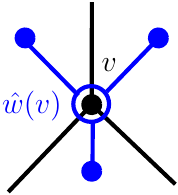}
	\caption{Black trivalent vertices of $G$ correspond to white trivalent vertices of $\hatG$.}
	\label{fig:tdualmisc}
\end{figure}

	We now show that if $G$ has the trip $\gamma: i \rightarrow \pi(i)$, then $\hat{G}$ has the trip $\hat{\gamma}: \widehat{i+1 }\rightarrow \widehat{\pi(i)}$.

Say $\gamma$ starts at $i$. Let $v$ be the first black vertex $\gamma$ meets. By the rules of the road, there is one edge $e$ attached to $v$ at the left of $\gamma$ (as $G$ is black-trivalent). Note that vertex $v$ is in the boundary face $f$ containing boundary vertices $i$ and $i+1$. This is because before meeting $v$, $\gamma$ meets only white vertices, and by the rules of the road there are no edges involving these vertices lying to the left of $\gamma$. So $\hat{w}(v)$ is also connected to $\hat{b}(f)$. And by definition, $\widehat{i+1}$ is connected to $\hat{b}(f)$. Note that at the vertex $\hat{b}(f)$, if we start at the edge to $\widehat{i+1}$ and go counterclockwise, we see the edge to $\hat{w}(v)$. This means $\hat{\gamma}$ starts at $\widehat{i+1}$, goes to $\hat{b}(f)$, then to $\hat{w}(v)$ (see \cref{fig:tdualsequences}). Now, let $g$ be the face of $G$ which contains $e$ and the edge of $\gamma$ following $v$. Clearly $\hat{b}(g)$ is connected to $\hat{w}(v)$. At the vertex $\hat{w}(v)$, if we start at the edge to $\hat{b}(f)$ and go clockwise, we see the edge to $\hat{b}(g)$. This means that $\gamma$ goes from $\hat{w}(v)$ to $\hat{b}(g)$.

\begin{figure}[h]
\includegraphics[height=.8in]{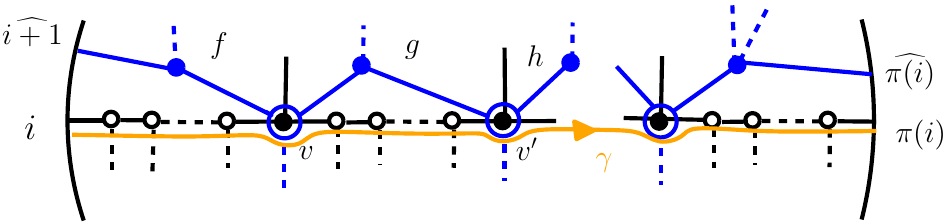}
	\caption{Black edges and vertices are in $G$; blue are $\hatG$. In orange, the trip $\gamma: i \to \pi(i)$ in $G$. The trip $\hat{\gamma}$ follows the solid blue edges.}
	\label{fig:tdualsequences}
\end{figure}

Now, let $v'$ be the next black vertex $\gamma$ meets. Again, the edges involving any white vertices on $\gamma$ between $v,v'$ must lie to the right of $\gamma$, and there is exactly one edge $e'$ at $v'$ to the left of $\gamma$. So the face $g$ also contains $v'$. Let $h$ be the face of $G$ which contains $e'$ and the edge of $\gamma$ following $v'$. Then $\hat{\gamma}$ goes from $\hat{b}(g)$ to $\hat{w}(v')$ to $\hat{b}(h)$ (see \cref{fig:tdualsequences}). Continuing in this way, we see that if $\gamma$ passes through a black vertex $v$, then $\hat{\gamma}$ passes through $\hat{w}(v)$ and then goes to $\hat{b}(f)$, where $f$ is the face to the left of $\gamma$ containing $v$ and the edge of $\gamma$ following $v$. If $v$ is the last black vertex on $\gamma$, then $f$ is the boundary face touching $\pi(i)-1$ and $\pi(i)$. Note that at the vertex $\hat{b}(f)$, if we start at the edge to $\hat{w}(v)$ and go counterclockwise, we see the edge to $\widehat{\pi(i)}$. So $\gamma$ will turn maximally right at $\hat{b}(f)$ to go to $\widehat{\pi(i)}$.

If $\gamma$ meets no black vertices, there are no edges of $G$ at the left of $\gamma$. This means $\pi(i)=i+1$. The boundary face $f$ between $i$ and $\pi(i)$ contains only white vertices, so there will be a loop in $\hat{G}$ at boundary vertex $\widehat{i+1}$. Clearly $\hat{\pi}(i+1)=i+1=\pi(i)$ as desired.

To show that $\hatG$ is reduced, it suffices to show $\hatG$ has $\dim(S_{\hat{\pi}})+1$  faces
	\cite[Corollary 7.4.26 and Corollary 7.10.5]{ca7}. Note that \cref{defn:plabicTDual} does not depend on the white vertices of $G$, so we may assume that $G$ is bipartite and has a white vertex adjacent to every boundary vertex. With this assumption, it is not hard to see that the faces of $\hatG$ are in bijection with white vertices of $G$.

Let $B, W, F, E$ denote the number of white vertices, black vertices, faces (excluding the infinite face), and edges (excluding edges between two boundary vertices) of $G$. Say that $G$ is of type $(k+1, n)$. Since
T-duality preserves codimension, we have
\[\dim(S_{\hat{\pi}})= \dim(S_{\pi})-n+2k+1.\]
As $G$ is reduced, $F=\dim(S_\pi)+1$. So to show $W=\dim(S_{\hat{\pi}})+1$, it suffices to show that $W=F-n+2k+1$. This follows immediately from
\[E=3B +n, \quad F=1-(W + B) + E, \quad W-B=k+1.\]
The first equation holds because every edge between two internal vertices contains a unique black vertex, and all black vertices are trivalent. The second equation follows from Euler's formula for planar graphs. The third holds because $G$ is type $(k+1, n)$.

\end{proof}

\begin{rmk}
	It is straightforward to check that exchanging the roles of black and white vertices in \cref{defn:plabicTDual} gives a map from white-trivalent plabic graphs to black-trivalent graphs. This shows that T-duality is a bijection between black-trivalent graphs of type $(k+1, n)$ and white trivalent graphs of type $(k, n)$ (where we consider both sets of graphs up to edge contraction and bivalent vertex addition/removal).
\end{rmk}

The map $G \to \hatG$ can also be phrased in terms of plabic tilings\footnote{We caution the reader that plabic tilings and positroid tilings are very different objects, despite having a word in common.} \cite{OPS}, which are dual to plabic graphs. Our notion of plabic tiling is slightly looser than that in \cite{OPS}.

\begin{definition}[Plabic tilings] \label{def:plabictilings}
Let $G$ be any connected reduced plabic graph with $n$ boundary vertices, and let $\mathbf{P}_n$ be a convex $n$-gon, whose vertices are labelled from $1$ to $n$ in clockwise order. The \emph{plabic tiling} $\T(G)$ dual to $G$ is a tiling of $\mathbf{P}_n$ by coloured polygons (bigons allowed) such that: i) it is the planar dual of $G$; ii) each black (white) vertex of $G$ is dual to a black (white) polygon in $\mathcal{T}(G)$; iii) vertex $i$ of $\mathbf{P}_n$ is dual to the face of $G$ touching boundary vertices $i-1$ and $i$. We consider two plabic tilings $\T(G)$ and $\T'(G')$ \emph{equivalent} if $G$ and $G'$ are move-equivalent. 

Conversely, if $\T$ is a plabic tiling, the dual plabic graph 
$G(\T)$ is obtained from $\T$ by placing a black vertex in each black polygon,
a white vertex in each white polygon, and connecting two vertices whenever they correspond
to two polygons which share an edge.
\end{definition}
\begin{figure}[h]
\includegraphics[height=1.4in]{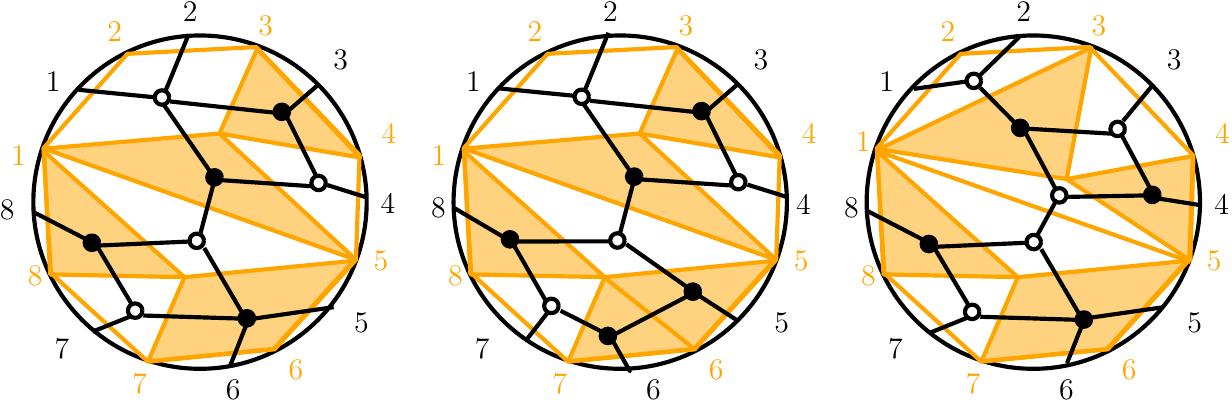}
	\caption{Three equivalent plabic tilings $\T$ (in orange), 
	and the corresponding dual plabic graphs $G(\T)$ (in black).
	The center plabic tiling is dual to a black-trivalent plabic graph.}
	\label{fig:plabicTiling}
\end{figure}

\cref{fig:plabicTiling} shows three move-equivalent plabic graphs and the
corresponding plabic tilings.

\begin{rmk} \label{rmk:subdivisionAreTiling}
	A bicolored subdivision or triangulation $\T$ of type $(k,n)$ is a plabic tiling whose dual plabic graph $G(\T)$ is a tree plabic graph of type $(k+1, n)$. All tree plabic graphs of type $(k+1, n)$ arise in this way.
\end{rmk}

The construction of $\hatG$ from $G$ of \cref{prop:TdualityPlabicGraphs} can also be phrased in terms of plabic tilings as follows. 
(This is equivalent to the construction in the proof of 
\cite[Proposition 8.3]{GalCritVar}, though the description there uses horizontal 
sections of fine zonotopal tilings.)

\begin{proposition}[T-duality and plabic graphs] \label{prop:tdualityplabictilings}
	Let $G$ be a connected reduced black-trivalent plabic graph and let $\mathcal{T}=\mathcal{T}(G)$ be the dual plabic tiling. Then the T-dual plabic graph $\hat{G}=\hat{G}(\T)$ 
	is obtained as follows: 
\begin{enumerate}
	\item Place a black vertex at each vertex of each black triangle in $\mathcal{T}$.
	\item Place a white vertex in the middle of each black triangle of $\mathcal{T}$ and connect it to the vertices of the triangle.
	\item Add an  edge of $\hat{G}$ from boundary vertex $i$ on the disc to the black vertex on boundary vertex $i$ of $\mathcal{T}$.
\end{enumerate}
\end{proposition}

\begin{figure}[h]
\includegraphics[height=1.4in]{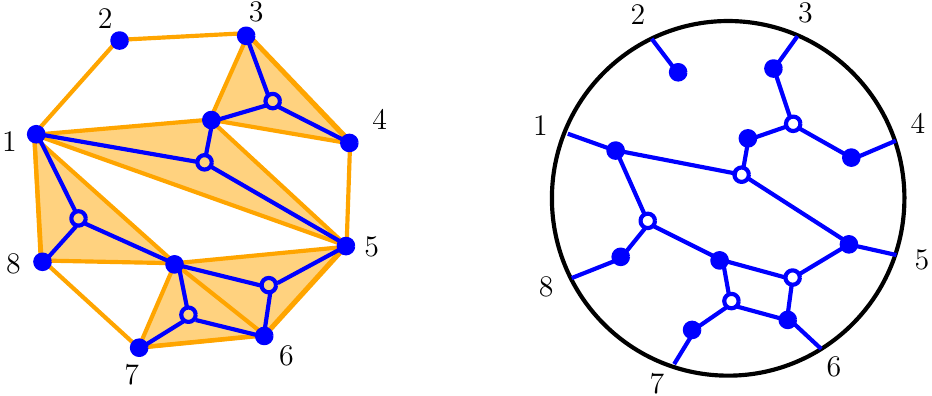}
	\caption{Left: In orange, the plabic tiling $\T$ dual to the black-trivalent graph in the center of \cref{fig:plabicTiling}. In blue,
	the result of operations (1), (2) of \cref{prop:tdualityplabictilings}.
	 At the right $\hat{G}(\T)$.}
	\label{fig:ghatplabicTiling}
\end{figure}

\begin{rmk}\label{rk:treesTilings}
The construction $\hat{G}(\T)$ from \cref{prop:tdualityplabictilings}
	generalizes the construction from
 \cref{def:param} (viewing a bicolored triangulation as a special case of a plabic tiling).
So \cref{prop:tdualityplabictilings} shows that the plabic graph $\hatG(\T)$ from \cref{def:param} is T-dual to the plabic tree $G(\T)$.
\end{rmk}

\section{T-duality, positroid tiles and cluster adjacency} \label{sec:Tduality2}
In this section, we show T-duality gives a bijection between positroid tiles of $\Delta_{k+1, n}$ and positroid tiles for $\AA_{n, k,2}(Z)$ (\cref{cor:GTsInBijection}). We then investigate parallels between the inequalities cutting out positroid polytopes $\Gamma_\pi$ and the T-dual Grasstopes $\gto{\hat{\pi}}$; for positroid tiles, both pieces of data are encoded by the same bicolored subdivision (\cref{thm:inequalitiesMatch}). We establish a similar parallel for facets of positroid tiles (\cref{th:facetsGT}), and use this to prove the \emph{$m=2$ cluster adjacency conjecture} of 
\cite{Lukowski:2019sxw}
in \cref{thm:cluster}.

\subsection{T-duality, Inequalities and Signs}
In this subsection, we will see how bicolored triangulations encode positroid tiles of both $\Delta_{k+1, n}$ and $\AA_{n, k, 2}(Z)$.

\cref{thm:allGTs} and \cref{prop:homeo2} characterize positroid tiles of $\AA_{n,k,2}(Z)$ 
and $\Delta_{k+1,n}$ in terms of bicolored subdivisions and tree plabic graphs, respectively.
These results with \cref{rk:treesTilings} imply that positroid tiles of 
$\AA_{n,k,2}(Z)$ and $\Delta_{k+1,n}$ are in bijection, and that both can be read off easily from 
bicolored subdivisions of type $(k,n)$ (see \cref{fig:unpuncTduality}).

\begin{corollary}\label{cor:GTsInBijection}
A positroid polytope $\Gamma_G$ is a positroid tile of $\Delta_{k+1,n}$ if and only if the T-dual Grasstope $Z_{\hat{G}}$ is a positroid tile of $\AA_{n,k,2}(Z)$.
We read $\Gamma_G$ and $Z_{\hat{G}}$ off of the same bicoloredsubdivision $\overline{\T}$ as follows:
\begin{itemize}
   \item Choose any triangulation $\T$ of $\overline{\T}$.  
	\item We let $G: = G(\T)$ be the dual plabic tree, as in 
\cref{def:plabictilings}.
	\item We let $\hat{G}:=\hat{G}(\T)$ be the graph from \cref{def:param} 
		(equivalently, in \cref{prop:tdualityplabictilings}).
\end{itemize}
\end{corollary}

\begin{figure}[h]
	\includegraphics[width=0.8\textwidth]{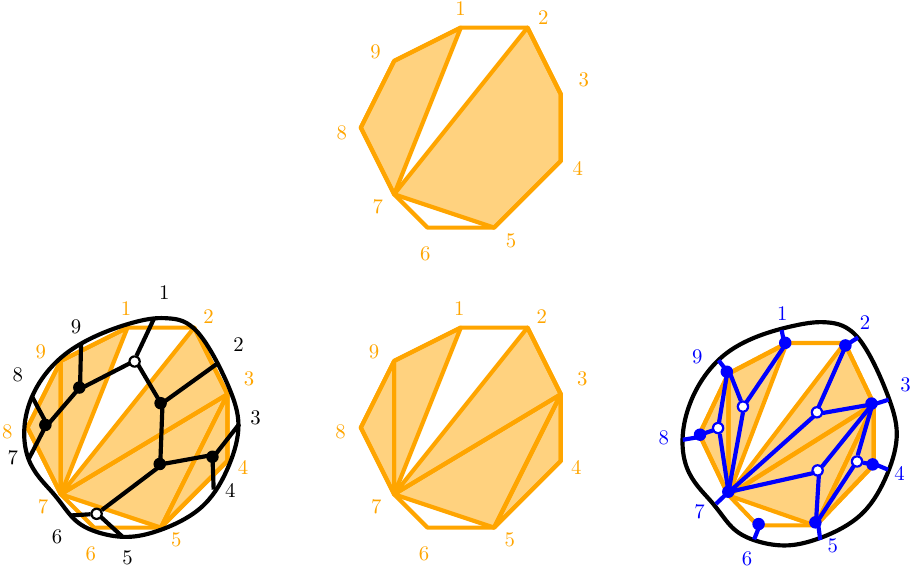}
	\caption{In the top row: a bicolored subdivision of type $(5, 9)$ $\overline{\T}$. 
		In the bottom row: a bicolored triangulation $\T$ obtained by triangulating 
		$\overline{\T}$, with the dual graph $G(\T)$ to its left, and the T-dual graph $\hatG(\T)$ to 
		its right.}
	\label{fig:unpuncTduality}
\end{figure}

From a bicolored subdivision $\overline{\T}$, we can 
obtain inequality descriptions of the positroid tile $\Gamma_{G(\T)} \subset \Delta_{k+1, n}$ and the T-dual positroid tile $\gto{\hatG(\T)} \subset \AA_{n, k, 2}(Z)$.

Given two positive numbers $a,b\in [n]$, the 
\emph{cyclic interval} $[a,b]$ is defined to be
$$[a,b]:=\begin{cases}
	\{a,a+1,\dots,b-1,b\}&\text{ if }a\leq b \\
	\{a,a+1,\dots,n,1,\dots,b\} &\text{ otherwise.}
\end{cases}$$

\begin{thm}[Inequalities and signs via T-duality]\label{thm:inequalitiesMatch}
	Let $\overline{\T}$ be a bicolored subdivision and let $h \to j$ be a compatible arc, with $h<j$. Let $G(\T)$ denote the tree plabic graph dual to $\T$, and $\hatG(\T)$ the T-dual. Then: 
\begin{align*}
	&\mbox{(1)} &\area(h \to j)+1 > x_{[h, j-1]} > \area(h \to j) \quad \text{for }&x \in \Gamma^{\circ}_{G(\T)}&\\
	&\mbox{(2)}& \sgn \langle Y h j\rangle=(-1)^{\area(h \to j)} \quad \text{for }&Y \in \gto{\hatG(\T)}.&
\end{align*}	
	The inequalities given by the arcs of any triangulation $\T'$ of $\T$ cut out $\Gamma^{\circ}_{G(\T)}$ and $\gto{\hatG(\T)}$.

\end{thm} 
\begin{example}
Consider the bicolored subdivision $\overline{\T}$ in \cref{fig:unpuncTduality}. 
We have 
\begin{align*}
	& 5 > x_{[1,7]} > 4, && 4 > x_{[1,6]} > 3, && 3 > x_{[2,5]} > 2, && \quad \text{for }  x \in \Gamma^{\circ}_{G(\T)}\\
&\langle Y 18 \rangle > 0,&& \langle Y 17 \rangle < 0, && \langle Y 26 \rangle > 0, && \quad \text{for } Y \in \gto{\hat{G}(\T)}.
\end{align*}
\end{example}

To prove 
\cref{thm:inequalitiesMatch}, 
we need a few results on positroid polytopes $\Gamma_G$. 

\begin{lem}\label{lem:kInTermsOfTriangles}
	Let $G$ be a bipartite plabic graph and let $\T$ be the dual plabic tiling. Let $W(G)$ and $B(G)$ denote the set of white and black vertices of $G$, respectively. Then \[|W(G)|-|B(G)|+|\{\text{bdry vt of } G\text{ adjacent to a black vt}\}|= \area(\T)-\punc(\T)+1\]
	where $\area(\T)$ is the number of black triangles in any triangulation of $\T$ and $\punc(\T)$ is the number of internal vertices of $\T$. 
\end{lem}
\begin{proof}
	Let $E$ denote the edges of $G$ involving at least one internal vertex. Each black vertex of $G$ is dual to a black polygon of $\T$ with $\deg(v)$ many sides, so we have
	\begin{align*}
		\area(\T)&= \sum_{v \in B(G)} \left( \deg(v)-2 \right) = \sum_{v \in B(G)} \deg(v) - 2|B(G)|\\
		&= |E(G)|-|\{\text{bdry vt adjacent to a white vt}\}|-2|B(G)|
	\end{align*}
	where the last equality follows from the fact that every edge of $G$ contains a unique black vertex, except edges between a boundary vertex and a white vertex. The claim follows from this formula together with Euler's formula for planar graphs.
\end{proof}

\begin{prop}\label{lem:treeInequalities}
	Let $\T$ be a bicolored subdivision and $G(\T)$ the dual bipartite tree plabic graph. For all arcs $h \to j$ compatible with $\T$, points of $\Gamma_{G(\T)}$ satisfy 
	\[\area(h \to j) +1 \geq x_{[h, j-1]}\geq \area(h \to j).
	\]
\end{prop}
\begin{proof}
	Let $G$ be the graph obtained from $G(\T)$ by adding bivalent white vertices so that every boundary vertex is adjacent to a white vertex. Note that $G$ is bipartite and represents the same positroid $\M$ as $G(\T)$. In particular, the boundaries of matchings of $G$ give the bases of $\M$. Let $W(G)$ and $B(G)$  denote the sets of white and black vertices of $G$, respectively.
	
	Note that if $j=h+1$, the inequality is clear. 

	We first deal with the case where $h \to j$ is an internal arc of $\T$. 
	Let $e$ be the edge of $G$ which is dual to $h \to j$, and say the vertices of $e$ are a white vertex $w$ and black vertex $b$. If we remove the edge $e$, $G \setminus e$ has two connected components, $G^w$ containing $w$ and $G^b$ containing $b$. Notice that both connected components are again bipartite plabic trees. Let $I^w$ and $I^b$ denote the boundary vertices of $G^w$ and $G^b$, respectively. Because vertex $i$ of $\T$ lies between boundary vertices $i-1, i$ of $G$,  $\{I^w, I^b\}= \{[h, j-1], [j, h-1]\}$.  
	
	Now, we would like to compute the ranks of $I_w, I_b$. That is, for a matching $M$ of $G$, we need to compute the maximum size of $\partial M \cap I^w$ and $\partial M \cap I^b$. 
	
	Let $M$ be a matching of $G$. If $M$ does not contain $e$, then $M$ restricts to a matching of $G^w$ and $G^b$. It is easy to see that \begin{align*}
		|\partial M \cap I^w|=&| W(G^w)|- | B(G^w)|\\
			|\partial M \cap I^b|=& | W(G^b)|- | B(G^b)|.
	\end{align*}

If $M$ does contain $e$, then choose a path $P$ from boundary to boundary which uses $e$ and alternates between edges in $M$ and edges not in $M$. Such a path can be constructed greedily because $G$ is a tree. Orient $P$ so it sees first $w$ and then $b$. The edges of $P$ in $M$ are exactly the ones oriented from a white vertex to a black vertex. The first edge of $P$ touches a boundary vertex in $I^w$ and is oriented to a white vertex, so is not in $M$. The last edge of $P$ touches a boundary vertex in $I^b$ and is in $M$. Define a new matching $N$ of $G$ by $N:= (M \setminus P) \cup (P \setminus M)$. The boundary $\partial N$ contains one more element of $I^w$ than $\partial M$, and one fewer element of $I^b$. The matching $N$ does not contain $e$, so using the previous computation, we see that 
\begin{align*}
	|\partial M \cap I^w|=& | W(G^w)|- | B(G^w)|-1\\
	|\partial M \cap I^b|=& |W(G^b)|- | B(G^b)|+1.
\end{align*}
	We conclude that $\rank(I^w) = |W(G^w)|-| B(G^w)|$ and $\rank(I^b)=| W(G^b)|- | B(G^b)|+1$. 

From  \cref{prop:inequalitiesmatroids} and the fact that the rank of $\M$ is $|W(G)|- | B(G)|= |W(G^w)|- |B(G^w)| + |W(G^b)|-| B(G^b)| $, we see that the points of $\Gamma_{G(\T)}$ satisfy
\begin{align*}
	| W(G^w) |-| B(G^w)|-1 \leq & x_{I^w} \leq | W(G^w)|-| B(G^w)|\\
	| W(G^b) |-| B(G^b)| \leq& x_{I^b} \leq | W(G^b)| - |B(G^b)|+1.
\end{align*}

All that remains is to rewrite the right hand sides of these inequalities in terms of area. Cut $\T$ along the arc $h \to j$, to get two smaller bicolored subdivisions $\T^w$ and $\T^b$ containing the polygons dual to $w$ and $b$, respectively. Notice that the graph $G(\T^w)$ dual to $\T^w$ can be obtained from $G^w$ by adding a boundary vertex adjacent to $w$. Similarly, $G(\T^b)$ is obtained from $G^b$ by adding a boundary vertex adjacent to $b$. So, using \cref{lem:kInTermsOfTriangles}, 
\begin{align*}
	|W(G^w)|- |B(G^w)|&= \area(\T^w)+1\\
	|W(G^b)|- |B(G^b)|+1&= \area(\T^b)+1.
\end{align*}

Now, choose $v \in \{b, w\}$ so that $I^v=[h, j-1]$. Since $\T^v$ is exactly the part of $\T$ to the left of $h \to j$, the proposition now follows.

We now consider the case where $h \to j$ is not an arc of $\T$. In this case, let $\T'$ be the plabic tiling obtained from $\T$ by adding the arc $h \to j$. Let $G$ be the tree plabic graph dual to $\T'$, which we make bipartite by adding an appropriately colored bivalent vertex $v$ to the edge dual to $h \to j$. We also add bivalent white vertices to $G$ to make all boundary vertices adjacent to a white vertex. Let $e$ and $f$ denote the edges containing $v$, and say $e$ is to the left of $h \to j$. Similar to the first case, removing edges $e, f$ and vertex $v$ from $G$ gives a graph with two connected components $G^e$, $G^f$ which contain vertices adjacent to $e$ and $f$, respectively. Notice that the boundary vertices $I^e$ of $G^e$ are exactly $[h, j-1]$. The rest of the argument is very similar to the first case.

\end{proof}

Recall that an arc $h \to j$ of a bicolored subdivision is facet-defining if it bounds a black polygon on its left. 
\begin{prop} \label{prop:treefacet}
	Let $\T$ be a bicolored subdivision, and let $G$ be the dual plabic tree of type $(k+1, n)$. Then $\Gamma_{G}$ is cut out of $\R^n$ by the equality $x_{[n]}=k+1$ and the following inequalities, each of which defines a facet:
	\begin{enumerate}
		\item $x_i \geq 0$ for $i$ a boundary vertex adjacent to a white vertex
		\item $x_{[h, j-1]} \geq \area(h \to j)$ for $h \to j$ a facet-defining arc of $\T$. 
	\end{enumerate}
\end{prop}

\begin{proof}
	
	Recall from \cref{prop:homeo2} that the moment map $\mu$ is a stratification-preserving homemorphism on the closure of $S_{G}$. So the facets of $\Gamma_G$ are exactly the positroid polytopes $\Gamma_{G'}$ where $S_{G'}$ is a positroid cell contained in $\overline{S_G}$ with codimension 1. From \cite[Corollary 18.10]{postnikov}, each such cell is indexed by a reduced plabic graph $G'$ obtained from $G$ by removing a single edge (if the edge removed is between a boundary vertex and an internal vertex $v$, we also add a lollipop which is the opposite color of $v$). 
	
	Because $G$ is a tree, $G'=G\setminus e$ is reduced for all edges $e$. If $e$ is between a boundary vertex $i$ and a white vertex, then $G'$ has a black lollipop at $i$. Thus $S_{G'}$ has a loop at $i$, and $\Gamma_{G'}$ is contained in the hyperplane $x_i=0$. Clearly $\Gamma_{G}$ lies on the positive side of this hyperplane, which explains the facet inequalities of type $1$. 
	
	If $e$ is an edge between a boundary vertex $i$ and a black vertex, then $e$ is dual to the arc $(i+1) \to i$ of $\T$, which is a facet-defining arc. Then $G'$ has a white lollipop at $i$, so $\Gamma_{G'}$ is contained in the hyperplane $x_i=1$. Since we also have $x_{[n]}=k+1$, $\Gamma_{G'}$ is also contained in the hyperplane $x_{[i+1, i-1]}= k=\area(i+1 \rightarrow i)$. 
	
	Now, consider the case when $e$ is an edge between two internal vertices of $G$. The edge $e$ is dual to the arc $h \to j$ of $\T$, which bounds a black polygon on the left. The proof of \cref{lem:treeInequalities} shows that $\Gamma_{G'}$ is contained in the hyperplane $x_{[h, j-1]}= \area(h \to j)$. This covers all edges of $G$, so we have described all facets. The directions of the facet inequalities follow immediately from \cref{lem:treeInequalities}.
\end{proof}

We can now prove \cref{thm:inequalitiesMatch}.

\begin{proof}[Proof of \cref{thm:inequalitiesMatch}]
	(1) follows from \cref{lem:treeInequalities} and (2) follows from \cref{thm:sign1}. 
	The statement about inequalities cutting out $\gto{\hatG(\T)}$ and $\Gamma^\circ_{G(\T)}$ follow from 
	\cref{thm:surjectivity}, \cref{prop:treefacet} and the fact that $x_{[n]}=\area(h \to j)+ \area(j \to h)+1$.
\end{proof}

We next generalize \cref{thm:inequalitiesMatch} by providing inequalities for full-dimensional positroid polytopes and Grasstopes from statistics of plabic tilings. We first generalize the definition of compatible arcs from \cref{def:arcs}.
\begin{definition}[Statistics of plabic tilings]	\label{def:arcsPT}
	Let $\T$ be a plabic tiling in a convex $n$-gon
	$\mathbf{P}_n$
	and $ h, j$ a pair of vertices of $\mathbf{P}_n$. We say that the arc $h\to j$
	is \emph{compatible} with $\T$ if 
	the arc either bounds or lies entirely inside a single polygon
	of $\T$. 
	When $h\to j$ is compatible with $\T$, we let 
	$\area(h\to j) = \area_{\T}(h\to j)$ 
	denote the number of black triangles to the left of $h\to j$ 
	in any triangulation of $\T$.   We call the internal vertices of $\mathcal{T}$ \emph{punctures}. We also let
	 $\punc(h\to j)=\punc_{\T}(h\to j)$ denote the number of punctures of $\mathcal{T}$ to the left of the arc $h\to j$. Note that black bigons do not contribute to the area.
\end{definition}

For example, the tiling $\T$ in \cref{fig:plabicTiling} has two punctures, and $1 \rightarrow 3, 1 \rightarrow 5, 5 \rightarrow 7$ are compatible arcs. We have $\area(1 \rightarrow 3)=0$, $\area(1 \rightarrow 5)=2$, $\area(5 \rightarrow 7)=1$, $\punc(1 \rightarrow 3)=\punc(5 \rightarrow 7)=0$, $\punc(1 \rightarrow 5)=1$.

\begin{thm}\label{prop:inequalitiesMatchGen}
	Let $\T$ be a plabic tiling and let $h \to j$ be a compatible arc, with $h<j$. Let $G(\T)$ denote the plabic graph dual to $\T$, and $\hatG(\T)$ the T-dual. Then: 
\begin{align*}
	&\mbox{(1)}& \area(h \to j)-\punc(h \to j)+1 > x_{[h, j-1]} > \area(h \to j)-\punc(h \to j) \quad \text{for } &x \in \Gamma^{\circ}_{G(\T)}&\\
	&\mbox{(2)}& \sgn \langle Y h j\rangle=(-1)^{\area(h \to j)-\punc(h \to j)} \quad \quad \quad \text{for }&Y \in \gto{\hatG(\T)}.&
\end{align*}	
\end{thm}

	Note that compatible arcs depend on the tiling $\T$, while $\Gamma^{\circ}_{G(\T)}$ and $\gto{\hatG(\T)}$ depend only on $\overline{\T}$. Any arc compatible with any tiling equivalent to $\T$ gives inequalities for $\Gamma^{\circ}_{G(\T)}$ and $\gto{\hatG(\T)}$ via \cref{prop:inequalitiesMatchGen}  .

\begin{proof}
	The proof of (1) proceeds similarly as in \cref{lem:treeInequalities}, where we compute the rank of $[h, j-1]$. The arc $h \to j$ is dual to an edge $e$ of $G(\T)$ (or a graph which differs from $G(\T)$ only by uncontracting an edge and adding a bivalent vertex). Removing $e$ gives two connected components, the boundary vertices of which are $[h, j-1]$ and $[j, h-1]$. Again, any matching of $G(\T)$ will either use $e$ or differs from a matching using $e$ by a ``swivel" (see \cite[Appendix B]{MullerSpeyer}) along one of the two boundary faces containing $e$ (which changes the boundary's intersection with $[h, j-1]$ by precisely 1) followed by swivels at faces contained in one of the connected components (which do not change the boundary's intersection with $[h, j-1]$). In this way we compute the rank of $[h, j-1]$ and $[j, h-1]$. One must apply \cref{lem:kInTermsOfTriangles} to obtain the ranks in terms of $\area$ and $\punc$.
	
	The proof of (2) proceeds similarly as in \cref{thm:sign1}. Any almost-perfect matching $M$ of $\hatG(\T)$ which does not have $h$ or $j$ in $\partial M$ will have $|\partial M \cap [h+1, j-1]|= \area(h \to j)-\punc(h \to j)$. Indeed, there are exactly $\area(h \to j)$ internal white vertices to the left of $h \to j$, which must be covered by an edge of $M$, and exactly $\punc(h \to j)$ many internal black vertices, which also must be covered. This leaves $\area(h \to j)- \punc(h \to j)$ edges of $M$ which cover a boundary vertex.
\end{proof}

\begin{example}
For the plabic tiling $\T$ in \cref{fig:plabicTiling}, \cref{prop:inequalitiesMatchGen} tells us that
\begin{align*}
&1 > x_{[1,2]} > 0, && 2 > x_{[1,4]} > 1, && 2 > x_{[5,6]} >1, && \quad \text{for }  x \in \Gamma_{G(\T)}; \\
&\langle Y 13 \rangle > 0,&& \langle Y 15 \rangle < 0, && \langle Y 57 \rangle < 0, && \quad \text{for } Y \in \gto{\hat{G}(\T)}.
\end{align*}
\end{example}

\subsection{T-Duality, Facets and Cluster Adjacency}
From a bicolored subdivision $\T$, we can also read off the facets of both $\Gamma_{G(\T)}$ and $\gt{\hatG(\T)}$ (see \cref{def:facetG}).

\begin{thm}[Facets via T-duality] \label{th:facetsGT}
	Let $\T$ be a bicolored triangulation and let $h \to j$ be a facet-defining arc of $\T$. Let $G:=G(\T)$ be the plabic tree dual to $\T$ and let $G'$ be the plabic forest obtained from $G$ by deleting the edge dual to $h \to j$. Let $\hatG$ and $\hatG'$ denote their T-duals.
	\begin{enumerate}
		\item  The positroid polytope $\Gamma_{G'}$ is a facet of $\Gamma_G$, and lies on the hyperplane
		$$x_{[h, j-1]}= \area(h \to j). $$
		\vspace{-0.5cm} \item The Grasstope $\gt{\hatG'}$ is a facet of $\gt{\hatG}$, and lies on the hypersurface
		$$ \lrangle{Y h j}=0. $$
	\end{enumerate}
 Moreover, if we let $h \to j$ range over the facet-defining arcs of $\T$ which are not on the boundary of $\textbf{P}_n$, we obtain all facets of $\Gamma_G$ and $\gt{\hatG}$ in the interior of $\Delta_{k+1, n}$ and $\AA_{n, k, 2}(Z)$.
\end{thm}

\begin{proof}
	$(1)$ follows immediately from \cref{prop:treefacet} and its proof.
	
	For $(2)$, we first show that $\gt{\hatG'}$ is contained in the hypersurface $\{\lrangle{Y h j}=0\}$. The arc $h \to j$ is in a unique triangle $T_r$ of $\T$; say its third vertex is $i$. Using \cref{prop:TdualityPlabicGraphs}, it is not hard to see that $\hatG'$ is obtained from $\hatG$ by deleting the edge $e$ from $B_i$ to $W_r$. This means that every almost perfect matching of $\hatG'$ must use either the edge from $B_h$ to $W_r$ or the edge from $B_j$ to $W_r$, so $h$ or $j$ is in the boundary. From \cref{lem:expansion}, we immediately conclude that $\lrangle{Yh j}$ is identically zero on $\gto{\hatG'}$ and thus on $\gt{\hatG'}$.
	
	Now, we show that $\tilde{Z}$ is injective on $\gt{\hatG'}$, by showing \cref{prop:niceRep} holds for $\gt{\hatG'}$ and then applying the first paragraph in the proof of \cref{thm:allGTs}. Consider $Y \in \gt{\hatG'}$, let $C:=\twMtx(Y)$ be the twistor coordinate matrix of $Y$ and let $Y':= C Z$. We would like to show that $C \in S_{\hatG'}$ and that $\rowspan Y'=\rowspan Y$; by \cref{prop:kindaInverse}, it suffices to show the former. Note that the Kasteleyn matrix $K'$ for $\hatG'$ is obtained from the Kasteleyn matrix $K$ for $\hatG$ by setting the parameter in row $r$ and column $i$ to 0. So we only need to show that for all arcs $a \to b$ of $\T$ with $\{a, b\} \neq \{h, j\}$, $\lrangle{Y' ab}$ is nonzero. 
	
	Pick such an arc $a \to b$ of $\T$. It suffices to show that there is a matching of $\hatG$ which does not use $e$ and does not have $a$ or $b$ in its boundary. We will argue by induction on the number of black triangles of $\T$. The base case, with 1 triangle, is clear by inspection.
	The arc $a \to b$ bounds some black triangle $T_s$ of $\T$, with third vertex $c$. Cut $\T$ along the arcs $a \to b$, $b \to c$ and $c \to a$ to obtain bicolored triangulations $\T_1, \T_2, \T_3$ of smaller polygons (one of which may be empty), which each contain a single edge of $T_s$. One will have $h \to j$ as a facet-defining arc. By induction, $\hatG(\T_i)$ has an almost-perfect matching $M_i$ whose boundary avoids the appropriate vertices of $T_s$ and does not use the edge $e$. Take the union of these matchings, together with the edge $f$ from $W_s$ to $B_c$. This gives a matching of $\hatG$ whose boundary avoids $a, b$. Note that since $\{a, b\} \neq \{h, j\}$, the edge $f$ is different from $e$, so this matching does not use $e$.

	Now we check that $\gt{\hatG'}$ is a facet of $\gt{\hatG}$. We first show that the hypersurface $H:=\{\lrangle{Yhj}=0\}$ intersects $\gt{\hatG}$ only on the boundary of $\gt{\hatG}$, which shows $\gt{\hatG'}$ is contained in the boundary as well. Recall that the open positroid tile $\gto{\hatG}$ is dense in $\gt{\hatG}$ and moreover, $(-1)^{\area(h \to j)}\lrangle{Y hj}$ is positive on $\gto{\hatG}$ (\cref{thm:sign1}). This implies that $(-1)^{\area(h \to j)}\lrangle{Y hj}$ is positive on the interior of $\gt{\hatG}$. Indeed, if the hypersurface $H$ intersected the interior of $\gt{\hatG}$, one could find an open set in the interior where $(-1)^{\area(h \to j)}\lrangle{Y hj}$ is negative. (This is because $\lrangle{Yhj}$ is linear in the Pl\"ucker coordinates, so $\lrangle{Yhj}$ takes both positive and negative values on any open set in $\Gr_{k, k+2}$ containing a point of $H$). But such a set cannot be in $\overline{\gto{\hatG}}$.
	
	Now we verify that $\gt{\hatG'}$ has the correct codimension. From the proof of \cref{prop:treefacet}, $S_{G'}$ is codimension 1 in $\overline{S_G}$. Since T-duality is a rank-preserving poset isomorphism, we also have that $S_{\hatG'}$ is contained in $\overline{S_{\hatG'}}$ and has codimension 1; that is, $S_{\hatG'}$ has dimension $2k-1$. Because $\tilde{Z}$ is injective on $S_{\hatG'}$, $\gto{\hatG'}$ and $\gt{\hatG'}$ also have dimension $2k-1$.
	
	To see the last statement of the proposition, note that any codimension 1 cell $S_H \subset \overline{S_{\hatG}}$ with a coloop $q$ will have $\gt{H}$ contained in the hypersurface $\{\lrangle{Y q(q+1)}=0\}$. So the facets $\gt{H}$ avoiding the amplituhedron boundaries must come from coloopless cells $S_H$. These coloopless cells are T-dual to loopless codimension 1 cells contained in $\overline{S_{G}}$. As $h \to j$ varies over all facet-defining arcs of $\T$, $S_{G'}$ varies over all such loopless cells, by \cref{prop:treefacet}. So the facets $\gt{H}$ avoiding the amplituhedron boundary are of the form $\gt{\hatG'}$ for some arc $h \to j$. From the proof above, we see that $\gt{\hatG'}$ is not contained in an amplituhedron boundary precisely when $h \to j$ is not a boundary arc of $\T$.
\end{proof}

\begin{example}
Consider the bicolored subdivision $\T$ in \cref{fig:unpunctured}. The facet-defining arcs not on the boundary of $\textbf{P}_9$ are $1\to 7$, $2\to 7$ and $4\to 6$, with $\area(1\to 7) =\area(2\to 7) = 3$, $\area(4\to 6) = 0$. The corresponding internal facets lie on the following hyperplanes:
\begin{align*}
&\langle Y 17 \rangle= 0,&& \langle Y 27 \rangle= 0, && \langle Y 46 \rangle= 0, && \quad \mbox{for } Z_{\hat{G}(\T)};\\
&x_{[1,6]}=3, && x_{[2,6]}=3, && x_{[4,5]}=0, && \quad \mbox{for }  \Gamma_{G(\T)}.
\end{align*}
One facet-defining arc at the boundary of $\textbf{P}_9$ is $2\to 3$, with $\area(2\to 3) =0$. This gives an external facet lying on $\langle Y 23 \rangle= 0$ for $Z_{\hat{G}(\T)}$ and $ x_2=0 $ for $\Gamma_{G(\T)}$. 
\end{example}

We can now prove \cref{conj:adjacency}, which extends the \emph{$m=2$ cluster adjacency conjecture}  of
{\L}ukowski--Parisi--Spradlin--Volovich  
\cite{Lukowski:2019sxw}.

\begin{theorem}[Cluster adjacency for $\AA_{n,k,2}$]\label{thm:cluster}
Let $Z_{\hat{G}(\T)}$ be a positroid tile of $\AA_{n,k,2}(Z)$.
	 $$\text{Set }\Facet(Z_{\hat{G}(\T)}):=\{p_{ij} \ \vert \ 
	       \text{ there is a facet of $Z_{\hat{G}(\T)}$ on 
	       the hypersurface $\lrangle{Yij}=0$.}\}. \text{ Then:}$$
	    	 \begin{enumerate}
	 \item 
		 $\Facet(Z_{\hat{G}(\T)})$ consists of compatible cluster variables
       for $\Gr_{2,n}$.
 	\item If $p_{h\ell}$ is
	 compatible with 
	$\Facet(Z_{\hat{G}(\T)})$, then
	$\lrangle{Yh\ell}$  has a fixed sign on $Z^{\circ}_{\hat{G}(\T)}$. 
 \end{enumerate}
\end{theorem}

\begin{proof}
The first part follows directly from Theorem \ref{th:facetsGT} as $Z_{\hat{G}(\T)}$ has a facet on $\{\langle Y ij\rangle=0\}$ if and only if $i \to j$ is a facet defining arc in $\T$, and the facet-defining arcs do not cross. The second part follows from Theorem \ref{thm:inequalitiesMatch}.
\end{proof}

Using Theorems \ref{thm:inequalitiesMatch} and \ref{th:facetsGT},
we can translate the cluster adjacency theorem for the $m=2$ amplituhedron into a 
cluster adjacency theorem for the hypersimplex.

\begin{theorem}[Cluster adjacency for $\Delta_{k+1,n}$]\label{thm:clusterHyp}
Let $\Gamma_{G(\T)}$ be a positroid tile of $\Delta_{k+1,n}$.
	 $$\text{Set }\Facet(\Gamma_{G(\T)}):=\{p_{ij} \ \vert \ 
	       \text{ there is a facet of $\Gamma_{G(\T)}$ on 
	       the hyperplane $x_{[i,j-1]}=a_{i,j}$}\},$$
	       where $a_{i,j}$ are some non-negative integers. Then: 
	       \begin{enumerate}
	       \item $\Facet(\Gamma_{G(\T)})$ consists of compatible cluster variables
       for $\Gr_{2,n}$. 
       \item  If $p_{h\ell}$ is
	 compatible with 
	$\Facet(\Gamma_{G(\T)})$, then $x_{[h,\ell-1]}>\area(h \to \ell)$ in $\Gamma^\circ_{G(\T)}$.
\end{enumerate}	        
\end{theorem}


\section{Eulerian numbers: $w$-simplices in $\Delta_{k+1, n}$ and $w$-chambers in $\AA_{n, k, 2}$}\label{sec:wsimplices}

In this section we study the amplituhedron chambers of $\AA_{n, k, 2}(Z)$. Because positroid tiles in $\AA_{n, k, 2}(Z)$ are defined by sign conditions, the decomposition of $\AA_{n, k, 2}(Z)$ into chambers refines every positroid tiling. 
Separately, 
the hypersimplex $\Delta_{k+1, n}$ has a well-known decomposition 
into simplices 
which refines every positroid tiling. 
Both 
decompositions have chambers/maximal simplices which are naturally
indexed by permutations
of $n-1$ with $k$ descents.  
We use this correspondence in \cref{sec:wsimplices2} to 
establish results on tilings. 

We begin by reviewing the decomposition of the hypersimplex $\Delta_{k+1, n}$. 
It is well-known that the volume of the hypersimplex  $\Delta_{k+1,n}$
is the \emph{Eulerian number} $E_{k, n-1}$ \cite{stanley_ec1}, which counts the 
 permutations on 
$n-1$ letters with $k$ descents.
A triangulation
of $\Delta_{k+1,n}$ into unit simplices
indexed by such permutations
was first discovered by Stanley \cite{StanleyTriangulation}. Sturmfels \cite{SturmfelsGrobner} later gave an \emph{a priori} different triangulation of $\Delta_{k+1, n}$. Lam and Postnikov \cite{LamPost} then gave two other triangulations, and showed that all four triangulations coincide. 
After defining some permutation statistics, we will 
define this triangulation.

\begin{defn}
Let $w \in S_n$. We call a letter $i\geq 2$ in $w$ a \emph{left descent}
 (or  a \emph{left descent top}) if $i$ occurs to the left of $i-1$ in $w$.  In other words,
 $w^{-1}(i) < w^{-1}(i-1)$.
 And we say that $i\in [n]$ in $w$ is a \emph{cyclic left descent} if either
 $i\geq 2$ is a left descent of $w$ or if $i=1$ and $1$ occurs to the left of $n$ in $w$, 
 that is, $w^{-1}(1) < w^{-1}(n)$.
We let $\cdes(w)$ denote the set of cyclic left descents of $w$, and $\des(w)$ the set of left descents.
	We frequently refer to cyclic left descents as simply \emph{cyclic descents}.
\end{defn}

\begin{remark}
Left and right descents and descent sets are discussed extensively 
in \cite[Chapter 1]{bjornerbrenti}.
Left descents are sometimes called \emph{recoils} in the literature.
\end{remark}

Let $D_{k+1, n}$ be the set of permutations $w \in S_n$ with $k+1$ cyclic descents and $w_n=n$. Note that $|D_{k+1, n}|=E_{k+1, n-1}$.

\begin{defn}[$w$-simplices] \label{defn:wsimplexHSimplex}
 For $w \in D_{k+1, n}$, let $w^{(a)}$ denote the cyclic rotation of $w$ ending at $a$.
	We define 
	\[I_r=I_r(w):=\cdes(w^{(r-1)}).
	\]
The \emph{w-simplex} $\simp{w} \subseteq \Delta_{k+1, n}$ is the simplex with vertices $e_{I_1}, \dots, e_{I_n}$.
\end{defn}

\begin{example}\label{ex:descents}
	Let $w=324156$ in one-line notation. Then $w$ has cyclic descents $\{1, 2, 3\}=I_1$. The rotation of $w$ ending at $1$ is $563241$, which has cyclic descents $I_2=\{ 2, 3, 5\}$. The rotation of $w$ ending at $2$ is $415632$, which has cyclic descents $I_3=\{1,3,4\}$.
\end{example}
Notice that $r$ is always in $I_r$ and $r-1$ is never in $I_r$.

The following triangulation of the hypersimplex first appeared in \cite{StanleyTriangulation},
though the description there was slightly different.
\begin{proposition}[ \cite{StanleyTriangulation}]
	The $w$-simplices $\{\Delta_w: w \in D_{k+1, n}\}$ are the maximal simplices of a triangulation of the 
 hypersimplex $\Delta_{k+1,n}$.
Moreover,
projecting $\{\Delta_{w}: w\in S_n\}$ into $\RR^{n-1}$ (see \cref{rem:hypercube}),
we obtain the maximal simplices in a triangulation of the hypercube $\mbox{\mancube}_{n-1}$ 
which refines the subdivision of the hypercube into hypersimplices. 
\end{proposition}

\begin{rmk}\label{rmk:circuitDefn}
The $w$-simplex $\simp{w}$ as defined above agrees with the simplex denoted $\Delta_{(w)}$ in \cite[Section 2.4]{LamPost}. In particular, the directed circuit the authors use to define $\Delta_{(w)}$ is given by $e_{I_1} \to e_{I_{w_1+1}} \to  e_{I_{w_2+1}} \to \dots \to e_{I_{w_{n-1}+1}} \to e_{I_1}$. Another way to say this is $I_{w_{i}+1}$ is equal to $(I_{w_{i-1}+1} \setminus \{w_i\} )\cup \{w_i +1\}$.
\end{rmk}


It follows from the results of \cite{LamPost} that every full-dimensional positroid polytope also has a triangulation into $w$-simplices. Indeed, the triangulation of $\Delta_{k+1, n}$ given by $w$-simplices is the simultaneous refinement of all positroid subdivisions of $\Delta_{k+1, n}$.

We now turn to the amplituhedron side. We define some special chambers in $\AA_{n, k, 2}(Z)$ whose sign vectors are obtained from cyclic descents of permutations. We later will show that these are precisely the realizable sign chambers (\cref{thm:wSimpCover}, \cref{thm:nonempty}).


Recall that for $v\in \R^n$, $\flip(v)$ records where 
coordinates of $v$ change sign (\cref{def:flip}).

\begin{defn}[$w$-chambers]\label{def:ampchamber}
	Let $w \in D_{k+1,n}$ and let the vertices of $\simp{w}$ be $e_{I_1}, \dots, e_{I_n}$, as in Definition \ref{defn:wsimplexHSimplex}. Then the \emph{open amplituhedron $w$-chamber} $\asimpo{w}:=\AA_{n, k, 2}^w(Z)$ consists of $Y \in \AA_{n, k, 2}(Z)$ such that $\lrangle{Yij} \neq 0$ for $i \neq j$ and for $a=1, \dots, n$,
\[
	\flip(	\langle Y a \hat{1} \rangle , \langle Y a \hat{2} \rangle, \dots, \langle Y a \widehat{a-1} \rangle, \langle Y a a\rangle, \langle Y a a+1 \rangle, \dots, \langle Y a n \rangle)=
	I_a \setminus \{a\}.
\]
Equivalently, $\AA_{n, k, 2}^w(Z)$ consists of $Y \in \Gr_{k,k+2}$ such that
	\begin{align*} \sgn \langle Y a j \rangle& = (-1)^{|I_a \cap [a, j-1]|-1}
	\quad	\text{ for }j>a\\
		\sgn \langle Y a \hat{j} \rangle&= (-1)^{|I_a \cap [a, j-1]|-1}
	\quad	\text{ for }j<a.
	\end{align*}
 
 The \emph{closed amplituhedron $w$-chamber} is the closure $\asimp{w}:= \overline{\AA_{n, k, 2}^w(Z)}$. 
	Abusing notation, we will often refer to closed amplituhedron $w$-chambers
	as simply \emph{$w$-chambers}.
\end{defn}


\begin{rmk} \label{rmk:emptyWs}
One might hope that the structure of $\asimp{w}$ does not depend on the choice of $Z \in \Mat^{>0}_{n, k+2}$.  However, even the property  
	that $\asimp{w}$ is nonempty depends on $Z$.
	More precisely, while we know that 
	each $\asimp{w}$ is nonempty for some choice of $Z$ (\cref{thm:nonempty}), it  may be empty for other choices of $Z$ (see \cref{subsec:wChambersEmpty}). 
\end{rmk}

Because the positroid tiles of $\AA_{n, k, 2}(Z)$ can be described entirely in terms of signs of twistor coordinates and the signs of twistor coordinates in $\asimp{w}$ are constant, we have the following lemma. It is the analogue of the fact that for a tree positroid polytope $\Gamma_\pi$, either $\simp{w} \cap \Gamma_\pi^\circ =\emptyset$ or $\simp{w} \subseteq \Gamma_\pi$.

\begin{lem}\label{lem:disjointOrContained}
	Let $\gto{\pi}$ be a positroid tile for $\AA_{n, k, 2}(Z)$ and let $\asimp{w}$ be a non-empty $w$-chamber. Then either $\asimp{w} \cap \gto{\pi}= \emptyset$ or $\asimp{w} \subset \gt{\pi}$.
\end{lem}

Despite the subtleties regarding the nonemptiness of $\asimp{w}$, the closed $w$-chambers always cover $\AA_{n, k, 2}(Z)$, in direct analogy to $w$-simplices in $\Delta_{k+1, n}$.

\begin{thm}[$\AA_{n, k, 2}$ is the union of $w$-chambers] \label{thm:wSimpCover}
	Fix $k<n$ and $Z \in \Mat^{>0}_{n, k+2}$. Then 
	\[\AA_{n, k, 2}(Z)= \bigcup_{w \in D_{k+1,n}} \asimp{w}.\]
 \end{thm}
To prove \cref{thm:wSimpCover},  we use a characterization of the simplices $\simp{w}$  given by Sturmfels \cite{SturmfelsGrobner}, involving sorted collections. We follow the presentation of \cite[Section 2.2]{LamPost}.

\begin{defn}
	Let $(J_1, \dots, J_t)$ be a tuple of distinct elements of $\binom{[n]}{k+1}$, where
	we write  $J_s=\{j_{s1} < j_{s2}< \dots < j_{s(k+1)}\}$. We call $(J_1, \dots, J_t)$  a \emph{sorted collection} if $j_{11}\leq j_{21} \leq \cdots \leq j_{t1} \leq j_{12} \leq j_{22} \leq \dots \leq j_{t(k+1)}$. If $(J_1, J_2)$ is a sorted collection, we call them a \emph{sorted pair}.
\end{defn}

The $w$-simplices of $\Delta_{k+1, n}$ are exactly the simplices with vertices $e_{J_1}, \dots, e_{J_n}$ for $(J_1, \dots, J_n)$ a sorted collection. To see if a collection is sorted, one need only check pairs of elements.

\begin{lem}
\label{lem:sortedPairwiseEnough}
Given $\{J_1, \dots, J_t\} \subset \binom{[n]}{k+1}$, suppose that for all $a \neq b$, either $(J_a, J_b)$  or $(J_b, J_a)$ is a sorted pair. Then $J_1, \dots, J_t$ can be ordered to give a sorted collection.
\end{lem} 

\begin{proof}
First, notice that if $(J_a, J_b)$ is a sorted pair and $(J_b, J_c)$ is a sorted pair, then $(J_a, J_c)$ is a sorted pair. Indeed, if $a \neq b$, 
then there exists $i$ such that $j_{ai}<j_{bi} \leq j_{ci}$. 
It follows that $(J_c, J_a)$ is not a sorted pair, so $(J_a, J_c)$ must be. 
So on $\{J_1, \dots, J_t\}$, the property of being a sorted pair is reflexive, antisymmetric, 
and transitive, which means it is a partial order. 
We've assumed every pair is comparable, so we have a total order.  The result follows.
\end{proof}

\begin{proof}[Proof of \cref{thm:wSimpCover}]
	Let $Y \in \AA_{n, k, 2}$ be a point whose twistor coordinates are all nonzero. We will show that $Y$ lies in $\asimp{w}$ for some $w$. The points with nonzero twistor coordinates form a dense subset of $\AA_{n, k, 2}$ (their complement, a union of hypersurfaces, has codimension 1), so this will show the desired equality.
	\[
		\text{Set }L_a:=\flip(\langle Y a \hat{1} \rangle , \langle Y a \hat{2} \rangle, \dots, \langle Y a ~\widehat{a-1} \rangle, \langle Y a ~a\rangle, \langle Y a ~a+1 \rangle, \dots, \langle Y a ~n \rangle).
	\]
	By \cref{cor:cyclic}, we have $|L_a| = k$.
	Choose $a<b$. We will show that $I_a:=L_{a} \cup \{a\}$ and $I_b:=L_{b} \cup \{b\}$ are distinct and, for some ordering, form a sorted pair. 
We temporarily abuse notation by omitting the $Y$'s and hats from our notation;
	if $a >i$, we write $\lrangle{a i}$ for $\lrangle{Y a \hat{i}}$.
	
	Certain 3-term Pl\"ucker relations constrain sign flips, as noted in \cite[Section 5]{ATT}. For $j \in [a-2] \cup [b+1, n]$, we have the relation
	\begin{equation} \label{eqn:smallj}
		\lrangle{j ~j+1}\lrangle{a ~b}= \lrangle{a ~j}\lrangle{b ~j+1}- \lrangle{b ~j}\lrangle{a ~j+1}
	\end{equation}
and for $j\in [a+1, b-2]$ we have
\begin{equation} \label{eqn:medj}
	\lrangle{j ~j+1}\lrangle{b ~ a}= \lrangle{b ~j}\lrangle{a ~j+1} - \lrangle{a ~j}\lrangle{b ~j+1}.
\end{equation}

Because $\sgn \lrangle{j ~ j+1}= +$ for all $j$, the sign of the left hand sides of \cref{eqn:smallj,eqn:medj} does not depend on $j$. This means that if $\sgn \lrangle{a~ b}= +$, then for $j \in [a-2] \cup [b+1, n]$
\begin{equation}\label{eqn:forbidSign1}
	\begin{pmatrix}
	\sgn\lrangle{a ~ j}&  \sgn\lrangle{a ~ j+1} \\
	\sgn\lrangle{b ~ j}& \sgn\lrangle{b ~ j+1} \\
\end{pmatrix}\neq
\begin{pmatrix}
	\delta & \epsilon\\
	\epsilon & - \delta
\end{pmatrix}
\end{equation}
for any $\delta, \epsilon \in \{+, -\}$. Similarly, if $\sgn \lrangle{a~ b}= -$, then for $j \in [a-2] \cup [b+1, n]$ 
\begin{equation}\label{eqn:forbidSign2}
	\begin{pmatrix}
	\sgn\lrangle{a ~ j}&  \sgn\lrangle{a ~ j+1} \\
	\sgn\lrangle{b ~ j}& \sgn\lrangle{b ~ j+1} \\
\end{pmatrix}\neq 
\begin{pmatrix}
	\delta & -\epsilon\\
	\epsilon & \delta
\end{pmatrix}.
\end{equation}
If $\sgn\lrangle{b ~ a}=+$ (respectively, $-$), then for $j \in [a+1, b-1]$, the sign pattern in \cref{eqn:forbidSign2} (respectively, \cref{eqn:forbidSign1}) never occurs.

Suppose $j$ is a value where the sign pattern in \cref{eqn:forbidSign1} is forbidden. If there is a sign flip after $\lrangle{a ~ j}$ and not after $\lrangle{b ~j}$, then $\sgn\lrangle{a ~j}= \sgn\lrangle{b ~j}$; if there is a sign flip after $\lrangle{b ~ j}$ and not after $\lrangle{a ~j}$, then $\sgn\lrangle{a ~j} \neq \sgn\lrangle{b ~j}$. When the sign pattern in \cref{eqn:forbidSign2} is forbidden, there are analogous statements with conclustions swapped. This means that for any interval $I \subset [n]$ where one of the patterns is forbidden, $| L_a \cap I|$ and $|L_b \cap I|$ differ by at most one and either $(L_a \cap I,L_b \cap I)$ or $(L_b \cap I,L_a \cap I)$ is sorted\footnote{We extend the definition of sorted in the obvious way to sets whose sizes differ by at most one. In particular, if $(I, J)$ are sorted, we must have $|I|\geq |J|$.}. 

Which one of $(L_a \cap I,L_b \cap I)$ and $(L_b \cap I,L_a \cap I)$ is sorted gives us additional information.

Let $P:= L_a \cap [a-2]$, $Q:=L_b \cap [a-2]$. Suppose $P \neq Q$, and consider the smallest $j$ so that there is a sign flip after exactly one of $\lrangle{a ~  j}$ and $\lrangle{b ~j} $. Clearly, $(P, Q)$ is sorted if and only if there is a sign flip after $\lrangle{a ~ j}$ and not after $\lrangle{b ~j}$. If the latter occurs, then $\sgn\lrangle{a b}=+$ and $\sgn\lrangle{a ~1}= \sgn\lrangle{b ~1}$, or $\sgn\lrangle{a b}=-$ and $\sgn\lrangle{a ~1}\neq \sgn\lrangle{b ~1}$; in short, $ \sgn\lrangle{a~b} \cdot \sgn \lrangle{a~ 1} =\sgn \lrangle{b~1}$. Analogously, if $(Q, P)$ is sorted, then $ \sgn\lrangle{a~b}\cdot \sgn \lrangle{a~ 1} \neq \sgn \lrangle{b~1}$.

Let $T:= L_a \cap [a+1, b-2]$ and $U = L_b \cap [a+1, b-2]$, and suppose $T \neq U$. By essentially identical reasoning as in the previous paragraph, if $(T, U)$ is sorted, then $\sgn\lrangle{b~a} \cdot \sgn\lrangle{a ~a+1} \neq \sgn \lrangle{b~a+1}$. Since $\lrangle{a ~ a+1}>0$ by assumption, the latter condition implies we have a sign flip between $\lrangle{b ~a}$ and $\lrangle{b ~ a+1}$, so $a \in L_b$. Similar reasoning gives that if $(U, T)$ is sorted, then $a \notin L_b$.

Let $V:=L_a \cap [b+1, n]$ and $W:=L_b \cap [b+1, n]$ and suppose that $V \neq W$. Repeating the arguments of the previous paragraphs gives that if $(V, W)$ is sorted, then $b \notin L_a$, and if $(W, V)$ is sorted, then $b \in L_a$.

Now, there are two cases: $|L_a \cap [b, n]|$ and $|L_b \cap [b,n]|$ have the same parity or they have opposite parity. They are similar, so we will assume we are in the first case, and leave the second to the reader.

Suppose $|L_a \cap [b, n]|$ and $|L_b \cap [b,n]|$ have the same parity. Note that $(-1)^{|L_i \cap [i+1, j-1]|}$ is $\sgn \lrangle{i ~ j}$, and, since $b \notin L_b$, $L_b \cap [b, n]$ and $L_b \cap [b+1, n]$ are equal. So
\begin{align*}
	\sgn \lrangle{b ~n}&= (-1)^{|L_a \cap [b, n]|}\\
	&=(-1)^{|L_a \cap [a+1, b-1]|}(-1)^{|L_a \cap [a+1, n]|}\\
	&= \sgn \lrangle{a ~b} \cdot \sgn \lrangle{a ~1}
\end{align*}
and thus $(P,Q)$ is sorted. We will show that $(I_a, I_b)$ is sorted and $I_a \neq I_b$.

If $b \in L_a$, then $|V|$ and $|W|$ have different parity. In particular, $V$ and $W$ are not equal, so $(W, V)$ is sorted and $|W|=|V|+1$. The two sets interweave like 
\[w_1 \leq v_1 \leq w_2 \leq v_2 \leq \cdots \leq w_r \leq v_r \leq w_{r+1}.\]
Since $V=I_a \cap [b+1, n]$ and $W=I_b \cap [b+1, n]$, we also have that $I_a$ and $I_b$ are distinct. Note that $I_b \cap [b, n]= W \cup \{b\}$ and $I_a \cap [b, n]= V \cup \{b\}$ and we have
\[b\leq b<w_1 \leq v_1 \leq w_2 \leq v_2 \leq \cdots \leq w_r \leq v_r \leq w_{r+1}, \]
so $(I_b \cap [b, n],I_a \cap [b, n])$ form a sorted pair.
If $b \notin L_a$, then $|V|$ and $|W|$ have the same parity. The pair $(V, W)$ is sorted, and $I_a \cap [b, n]=V$, while $I_b \cap [b,n]= W \cup \{b\}$. So $(I_b \cap [b, n],I_a \cap [b, n])$ are a sorted pair in this case as well, since we have
\[b< v_1 \leq w_1 \leq v_2 \leq w_2 \leq \cdots \leq v_r \leq w_r.\]
Note that $b$ is in $I_b$ but not in $I_a$, so we also have that $I_a$ and $I_b$ are distinct in this case.

Now we turn to the sets $T$ and $U$. Because $\sgn\lrangle{a ~b} = (-1)^{k} \sgn \lrangle{b ~ a}$, we have 
\begin{align*}
	(-1)^{|L_a \cap [a, b-1]|} &= (-1)^{k}(-1)^{|L_b \cap [b, a-1]|}\\
	&=(-1)^{|L_b \cap [1, n]| + |L_b \cap [b, a-1]| }\\
&=	(-1)^{|L_b \cap [a, b-1]|}.
\end{align*}

Note that $|T| \leq |L_a \cap [a, b-1]|\leq 1+ |T|$, since $a \notin L_a$. If $a \in L_b$, then $(T, U)$ is sorted and $|L_b \cap [a, b-1]|= 1 + |U|$, since $b-1 \notin L_b$. If $T$ and $U$ have the same cardinality, then $|L_a \cap [a, b-1]|$ must be equal to $1+ |T|$ in order to have the same parity as $1+|U|$. Thus $b-1 \in L_a $. This means that $(I_a \cap [a, n], I_b \cap [a, n])$ is a sorted pair, as we have
\[a \leq a < t_1 \leq u_1 \leq \cdots \leq t_j \leq u_j<b-1<b \leq \dots \leq v_r \leq w_r.\]
If $|T|=|U|+1$, we conclude by similar reasoning that $b-1 \notin L_a$, and again $(I_a \cap [a, n], I_b \cap [a, n])$ is a sorted pair, as we have
\[a \leq a < t_1 \leq u_1 \leq \cdots \leq t_j \leq u_j\leq t_{q+1}<b \leq \dots \leq v_r \leq w_r.\]

If $a \notin L_b$, then $(U, T)$ is sorted and $L_b \cap [a, b-1] = U$, since $a$ and $b-1$ are not in $L_b$. A parity argument as in the last paragraph shows that if $|U|=|T|$, then $b-1 \notin L_a$; if $|U|= |T|+1$, then $b-1 \in L_a$. Either way, $(I_a \cap [a, n], I_b \cap [a, n])$ is a sorted pair; we see 
\[a \leq u_1 \leq t_1 \leq \cdots \leq u_j \leq t_j<b \leq \dots \leq v_r \leq w_r\]
in the first case and 
\[a \leq u_1 \leq t_1 \leq \cdots \leq u_j \leq t_j \leq u_{q+1}<b-1<b \leq \dots \leq v_r \leq w_r\]
in the second.

Finally, we deal with $P$ and $Q$. Recall that $(P, Q)$ are sorted. Since $|L_a \cap [b, n]|$ and $|L_b \cap [b,n]|$ have the same parity and $|L_a \cap [a, b-1]|$ and $|L_b \cap [a,b-1]|$ have the same parity, $|L_a \cap [1, a-1]|$ and $|L_b \cap [1,a-1]|$ have the same parity. Since $a-1 \notin L_a$, we have that $P= L_a \cap [1, a-1]$. On the other hand $|Q| \leq |L_b \cap [1, a-1]| \leq |Q|+1$. If $|P|=|Q|$, then for parity reasons $Q=L_b \cap [1, a-1]$ and thus $a-1 \notin L_b$. So $(I_a \cap [1, a-1], I_b \cap [1, a-1])$ are a sorted pair, as we have
\[r_1 \leq s_1 \leq \cdots \leq r_i \leq s_i . \]
Similarly, if $|P|=|Q|+1$, then $a-1 \in L_b$ and $(I_a \cap [1, a-1], I_b \cap [1, a-1])$ again are a sorted pair, since we have
\[r_1 \leq s_1 \leq \cdots \leq r_i \leq s_i \leq r_{i+1} \leq a-1. \]

Since $(I_a \cap [1, a-1], I_b \cap [1, a-1])$ is a sorted pair ending in an element of $I_b$ and $(I_a \cap [a, n], I_b \cap [a, n])$ is a sorted pair, it follows that $(I_a, I_b)$ is a sorted pair.
\end{proof}

Using \cref{thm:wSimpCover}, we can conclude that positroid tiles are unions of amplituhedron $w$-chambers, just as tree positroid polytopes are unions of hypersimplex $w$-simplices. More precisely, we have the following corollary, which we sharpen further in \cref{prop:simplexContainment}.

\begin{cor}[Positroid tiles are unions of $w$-chambers] \label{cor:GTunionWSimp}
	Let $\gt{\pi}$ be a positroid tile for $\AA_{n, k, 2}(Z)$. Then 
	
	\[\gt{\pi}= \bigcup_{\substack{{\asimp{w}: }\\{\asimpo{w} \cap \gto{\pi} \neq \emptyset}}} \asimp{w}.
	\]
\end{cor}



\section{T-duality and positroid tilings}\label{sec:wsimplices2}

In this section we show one of our main results: we 
prove that a collection $\{\Gamma_{\pi}\}$ 
of positroid polytopes is a positroid
tiling of $\Delta_{k+1,n}$ if and only if 
for all $Z\in \Mat_{n,k+2}^{>0}$, the collection of 
T-dual Grasstopes $\{Z_{\hat{\pi}}\}$ is a positroid tiling
of $\AA_{n,k,2}(Z)$.
Along the way  we show that realizable amplituhedron chambers are exactly counted by Eulerian numbers.  
We also explore the phenomena that $w$-chambers can be empty, 
and  define the $\mathcal{G}$-amplituhedron -- a $Z$-independent analogue of the amplituhedron in $\Gr_{2,n}$. Finally, we introduce the 
\emph{total amplituhedron} $\mathcal{G}_n^{(2)} \subset \Gr_{2,n}$ which is
the amplituhedron-analogue of the hypercube, and discuss positroid tilings based on descents/sign-flips. 

\subsection{Positroid tilings of $\Delta_{k+1, n}$ and $\AA_{n,k,2}$.}
Recall that $w$-simplices in $\Delta_{k+1, n}$ are indexed by $D_{k+1,n}$.
One main tool is the following.

\begin{prop}\label{prop:simplexContainment} Fix $k<n$ and $Z \in \Mat^{>0}_{n,k+2}$. Suppose $w \in D_{k+1,n}$ and that $\asimp{w} \neq \emptyset$. For any tree positroid polytope $\Gamma_{\pi}$, $\simp{w} \subset \Gamma_{\pi}$ if and only if $\asimp{w} \subset \gt{\td{\pi}}$.
\end{prop}

\begin{proof}
	Fix a bicolored triangulation $\T$ so that $G(\T)$ is a plabic tree with trip permutation $\pi$ and $\hatG(\T)$ has trip permutation $\td{\pi}$. 
	From \cref{thm:inequalitiesMatch}, $\gto{\td{\pi}}$ consists of $Y \in \Gr_{k, k+2}$ such that for all arcs $a \to b$ of $\T$ 
	\[ \begin{cases}
			\sgn\lrangle{Ya b }= (-1)^{\area(a \to b)}  &\text{if }a < b\\
			\sgn\lrangle{Ya \hat{b} }= (-1)^{\area(a \to b)}  &\text{if }a > b
	\end{cases}
\] 
and $\Gamma_\pi$ consists of the points $x \in \R^n$ satisfying \[\area(a \to b) \leq x_{[a, b-1]}  \leq \area(a \to b)+1 \]
for all arcs $a \to b$ of $\T$. (In fact, to cut out $\gto{\td{\pi}}$, it suffices to consider arcs with $a<b$.)
	
Suppose $\simp{w} \subset \Gamma_\pi$. Then the vertices $e_{I_1}, \dots, e_{I_n}$ of $\simp{w}$ satisfy the defining inequalities of $\Gamma_\pi$. In particular, for each arc $a \to b$ of $\T$,
\[\area(a\to b)\leq |I_a \cap [a, b-1]|\leq \area(a \to b) +1.\]

By \cref{rmk:circuitDefn}, there is another vertex $e_{I_r}$ of $\simp{w}$ satisfying $I_r = I_a \setminus \{a\} \cup \{a-1\}$. This vertex also satisfies the defining inequalities of $\Gamma_\pi$. Moreover, $|I_r \cap [a, b-1]|$ is 1 smaller than $|I_a \cap [a, b-1]|$, so we must have 
\[|I_a \cap [a, b-1]|=\area(a \to b) +1.\]

Consider $Y \in \asimpo{w}$. By definition, for $a<b$, $\sgn \lrangle{Ya b}= (-1)^{|I_a \cap [a, b-1]|-1}$. By the above computation, $\sgn \lrangle{Ya b}=(-1)^{\area( a \to b)}$ for every arc $a \to b$, so we have shown $\asimpo{w} \subset \gto{\td{\pi}}$. Taking closures gives the desired containment.

Now, suppose $\asimp{w} \subset \gt{\td{\pi}}$. This means that for all arcs $a \to b$ of $\T$, $\area(a \to b)+1$ is the same parity as $|I_a \cap [a, b-1]|$. We will show that for all $q$, 
\[ \area(a \to b) \leq |I_q \cap [a, b-1]| \leq \area(a \to b) +1.\]

From the alcove description of $w$-simplices in \cite[Section 2.3]{LamPost}, there is some $d$ so that $\simp{w}$ lies between the hyperplanes $\{x_{[a, b-1]}= d-1\}$ and $ \{x_{[a, b-1]} = d\}$. As noted above, there is a vertex $e_{I_r}$ of $\simp{w}$ satisfying $I_r = I_a \setminus \{a\} \cup \{a-1\}$. Since $|I_r \cap [a, b-1]|=|I_a \cap [a, b-1]| -1 $, we conclude that $d$ is $|I_a \cap [a, b-1]|$. Thus, it suffices to show that
\begin{equation}\label{eq:areaEquality}
	\area(a \to b) + 1= |I_a \cap [a, b-1]|.
\end{equation}
This is proved in the following lemma.
\end{proof}

\begin{lem}
	Let $\T$ be a bicolored triangulation of type $(k,n)$ and let $\simp{w} \subset \Delta_{k+1, n}$ be a $w$-simplex with vertices $e_{I_1}, \dots, e_{I_n}$. Suppose for all arcs $a \to b$ of $\T$, \[\area(a \to b) + 1 \equiv |I_a \cap [a, b-1]| \pmod{2}.\]
	Then $\area(a \to b)+1 = |I_a \cap [a, b-1]| $ for all arcs $a\to b$ of $\T$.
\end{lem}

\begin{proof}
	We use induction on $n$. The base cases are $n=3$ and $k=0, 1$, which are clear.
	
	Without loss of generality, we may assume that $\T$ contains the arc $1 \to (n-1)$. Indeed, $\T$ contains some arc $(r+1) \to (r-1)$. We can rotate $\T$ by $r$ to obtain a new triangulation  with an arc $1 \to (n-1)$.  We can also apply the corresponding cyclic shift $e_i \mapsto e_{i -r}$ to $\simp{w}$ to obtain a new simplex $\simp{u}$. The vertex $e_{I_p}$ of $\simp{w}$ is mapped to vertex $e_{J_{p-r}}$ of $\simp{u}$, where $J_{p-r} =\{i -r: i \in I_p\}$. If the proposition is true for the new triangulation and $\simp{u}$, it is easy to see (by shifting back) that it is true for $\T$ and $\simp{w}$.
	
	Let $\T'$ be the bicolored triangulation of type $(k', n-1)$ obtained by chopping the triangle with vertices $1, n-1, n$ off of $\T$. Note that $k'=k$ if this triangle is white, and $k'=k-1$ otherwise. Let $v \in S_{n-1}$ be the permutation obtained from $w$ by deleting $w_n=n$ and moving $n-1$ to the end.

	\textbf{Case I:} Suppose the triangle deleted from $\T$ is white, so $k'=k$. Then $\area_\T(1 \to (n-1))$ is equal to $\area_\T(1 \to n)$, so the assumption on parities means that $I_1 \cap [1, n-1]$ has the same size as $I_1 \cap [1, n-2]$. That is, $n-1 \notin I_1$, which means that $n-1$ appears to the right of $n-2$ in $w$. Deleting $w_n$ and moving $n-1$ to the end results in a permutation with the same number of cyclic descents as $w$, meaning that $\simp{v} \subset \Delta_{k', n-1}$. 
	
	The vertices of $\simp{v}$ are $e_{J_1}, \dots, e_{J_{n-1}}$, where
	\[J_a = \begin{cases}
		I_a & \text{if } n \notin I_a\\
		I_a \setminus \{n \}\cup \{n-1\} & \text{if } n \in I_a.
	\end{cases}
	\]
	
	For the moment, we will denote cyclic intervals in $[n-1]$ by $[a, b]'$.  
	
	Let $a \to b$ be an arc of $\T'$. Because $b \neq n$, $[a, b-1]$ either contains both $n-1$ and $n$, or neither. So $J_a \cap [a, b-1]'$ and $I_a \cap [a, b-1]$ have the same cardinality. Also, $\area_{\T'}(a \to b)$ is equal to $\area_{\T}(a \to b)$, so $\T'$ and $\asimp{v}$ satisfy the assumptions of the proposition. By induction, we can conclude that $|J_a \cap [a, b-1]'|= \area_{\T'}(a \to b)+1$. In light of the equalities in this paragraph, this means that for all arcs $a \to b$ of $\T$ where $a, b$ are not $n$, we have $|I_a \cap [a, b-1]|= \area_{\T}(a \to b)+1$. It remains to check that a similar equality for the arcs $1 \to n$, $(n-1) \to n$ and their reverses, which are trivial.
	
	\textbf{Case II:} Suppose the triangle deleted from $\T$ is black, so $k'=k-1$. Then $\area_\T(1 \to (n-1))$ is equal to $\area_\T(1 \to n)-1$. The assumption on parities implies that $I_1 \cap [1, n-1]$ and $I_1 \cap [1, n-2]$ are different sizes, so $n-1 \in I_1$. This means that $n-1$ appears to the left of $n-2$ in $w$, and $v$ has one fewer left descent than $w$. So $\simp{v} \subset \Delta_{k', n-1}$ as desired.
	
	The vertices of $\simp{v}$ are $e_{J_1}, \dots, e_{J_{n-1}}$, where
	\[J_a = \begin{cases}
		I_a \setminus \{n \}& \text{if } n \in I_a\\
		I_a \setminus \{n-1\} & \text{if } n-1 \in I_a, n \notin I_a.
	\end{cases}
	\]
	
	Let $a \to b$ be an arc of $\T'$. Again, the cyclic interval $[a, b-1]$ either contains both $n-1$ and $n$, or contains neither. If $[a, b-1]$ contains neither, then clearly $|J_a \cap [a, b-1]'|= |I_a \cap [a, b-1]|$; in this case, $\area_{\T'}(a \to b)=\area_{\T}(a \to b)$ as well. If $[a, b-1]$ contains both, then $|J_a \cap [a, b-1]'|= |I_a \cap [a, b-1]|-1$ and $\area_{\T'}(a \to b)=\area_{\T}(a \to b)-1$. So again, $\T'$ and $\simp{v}$ satisfy the assumptions of the proposition. As in Case I, we can conclude that for all arcs $a \to b$ of $\T$ where $a, b$ are not $n$, we have $|I_a \cap [a, b-1]|= \area_{\T}(a \to b)+1$. The equalities for the arcs $1 \to n$, $(n-1) \to n$, and their reverses are clear.
\end{proof}

\begin{rmk} 
	\cref{prop:simplexContainment} motivates the intuition that the $w$-simplex $\Delta_w \subset \Delta_{k+1,n}$ and the
	$w$-chamber $\hat{\Delta}_w(Z) \subset \AA_{n,k,2}(Z)$ are `T-dual' to each other. In \cref{prop:wSimplChambInt} we will show that any $w$-simplex is the intersection of $n$ distinguished positroid polytopes $\{\Gamma_{\pi}\}$, and the corresponding $w$-chamber is the intersection of the $n$  T-dual Grasstopes $\{Z_{\hat{\pi}}\}$. 
\end{rmk}

To prove the correspondence between positroid tilings, we also need the following crucial result, whose proof we delay to the following subsection.

\begin{theorem}[$w$-chambers are realizable]\label{thm:nonempty}
	For each $w\in D_{k+1,n}$, there exists some $Z\in \Mat_{n,k+2}^{>0}$ such that the 
	amplituhedron $w$-chamber
	$\asimp{w}$  in $\AA_{n,k,2}(Z)$ is nonempty.
\end{theorem}

We can now show the main result of this section.

\begin{thm}[Tilings of $\Delta_{k+1,n}$ and $\mathcal{A}_{n,k,2}$ are T-dual] \label{thm:simpTriangGiveAmp}
	The collection $\mathcal{C}=\{\Gamma_{\pi}\}$ is a positroid tiling of $\Delta_{k+1, n}$ if and only if for all $Z \in \Mat_{n,k+2}^{>0}$, the collection of T-dual Grasstopes $\td{\mathcal{C}}=\{\gt{\td{\pi}}\}$ is a positroid tiling of $\AA_{n, k, 2}(Z)$.
\end{thm}

\begin{proof}
	$(\implies):$ Suppose $\mathcal{C}$ is a positroid tiling of $\Delta_{k+1, n}$ and choose $Z\in \Mat_{n,k+2}^{>0}$. We already know that $\gt{\td{\pi}}$ is a positroid tile from \cref{cor:GTsInBijection}.
	
	We first show that the Grasstopes in $\hat{\mathcal{C}}$ are dense in the amplituhedron. Consider a nonempty amplituhedron $w$-chamber $\asimp{w}$. Since $\mathcal{C}$ is a positroid tiling, there exists a tree positroid polytope $\Gamma_{\pi} \in \mathcal{C}$ which contains $\simp{w}$. By \cref{prop:simplexContainment}, $\asimpo{w} \subset \gto{\td{\pi}}$, where the latter is by definition in $\td{\mathcal{C}}$. So we have
	\[\bigcup_w \asimpo{w} \subseteq \bigcup_{\td{\mathcal{C}}} \gto{\td{\pi}} \subseteq \AA_{n, k, 2}(Z).
	\]
	By \cref{thm:wSimpCover}, the closure of the left-most set is equal to the right, so the closure of the middle set is $\AA_{n, k, 2}(Z)$, as desired.
	
	Now, suppose for the sake of contradiction that two distinct $\gto{\td{\pi}},\gto{\td{\mu}} \in\td{\mathcal{C}}$ are not disjoint. They are open, so their intersection is open, and thus their intersection contains a point in $\asimpo{w}$ for some $w$. \cref{lem:disjointOrContained} implies that in fact the entire $w$-simplex $\asimpo{w}$ is contained in their intersection. But then by \cref{prop:simplexContainment}, $\simp{w}$ is contained in $\Gamma_\pi \cap \Gamma_\mu$, a contradiction.
	
	$(\Longleftarrow)$: Suppose that for all $Z \in \Mat_{n,k+2}^{>0}$, $\hat{\mathcal{C}}$ is a positroid tiling of $\AA_{n, k, 2}(Z)$. By \cref{thm:nonempty}, for all $w \in D_{k+1, n}$, we can choose $Z$ so that $\asimp{w}$ is nonempty. In particular, $\asimpo{w}$ must intersect one of the positroid tiles $\gto{\hat{\pi}}$ and thus by \cref{lem:disjointOrContained}, $\asimp{w} \subset \gt{\hat{\pi}}$. Because $\hat{\mathcal{C}}$ is a positroid tiling, $\asimp{w}$ is not contained in any other positroid tile in $\hat{\mathcal{C}}$. Using \cref{prop:simplexContainment}, we see that every $w$-simplex is contained in precisely one positroid polytope in $\mathcal{C}$, and thus $\mathcal{C}$ is a positroid tiling of $\Delta_{k+1, n}$.
\end{proof}

In \cite{Karp:2017ouj}, they conjectured there are $\binom{n-2}{k}$
Grasstopes in a positroid tiling of $\mathcal{A}_{n,k,2}$. As noted in \cite{LPW}, this is also the number of positroid polytopes in a regular positroid tiling of $\Delta_{k+1,n}$ \cite{SW}, which are those arising from the \emph{tropical positive Grassmannian} $\Trop^+\Gr_{k+1,n}$ \cite{LPW}. 
A positroid tiling of $\AA_{n,k,2}$ is \emph{regular} if it is T-dual to a regular positroid tiling of $\Delta_{k+1,n}$. By Theorem \ref{thm:simpTriangGiveAmp}, we have:

\begin{corollary}
	There are $\binom{n-2}{k}$ Grasstopes in any regular positroid tiling of $\AA_{n,k,2}(Z)$.
\end{corollary}

\begin{rmk}
\cite{LPW} showed that all \emph{BCFW} tilings of $\AA_{n,k,2}(Z)$ contain $\binom{n-2}{k}$ Grasstopes; there are BCFW tilings which are not regular and regular tilings which are not BCFW.
\end{rmk}

\subsection{The $\mathcal{G}$-amplituhedron, the hypercube and the total amplituhedron}

In this subsection, we embed $\AA_{n,k,2}(Z)$ into a full-dimensional subset of $\Gr_{2,n}$ -- the `$\mathcal{G}$-amplituhedron' $\mathcal{G}_{n,k,2}$-- which does not depend on $Z$. We use sign chambers in the $\mathcal{G}$-amplituhedron to prove that all $w$-chambers of $\AA_{n, k, 2}$ are realizable (\cref{thm:nonempty}). We also draw another parallel between the hypersimplex and the ampliltuhedron. In \cref{rem:hypercube} we saw that the union of the (projected) hypersimplices $\tilde{\Delta}_{k+1,n}$ is the hypercube $\mbox{\mancube}_{n-1}$. Analogously, we take the union of $\mathcal{G}$-amplituhedra varying over all $k$ to obtain the \emph{total amplituhedron} $\mathcal{G}_{n}^{(2)}$, which is the amplituhedron-analogue of $\mbox{\mancube}_{n-1}$.

The following definition is intended to be a $Z$-independent version of the amplituhedron, inspired by \cref{cor:G}. 
\begin{definition}[The $\mathcal{G}$-Amplituhedron]
	Fix $k<n$ and let 
	\begin{align*}
		\mathcal{G}_{n,k,2}^{\circ}:= \{z\in \Gr_{2,n} \ &\vert \  
		p_{i,i+1}(z)>0 \text{ for }1 \leq i \leq n-1,
		\text{ and }  p_{n,\hat{1}}(z) >0,\\
		&\text{ and } 
		\var((p_{12}(z), 
		p_{13}(z),  \dots 
		p_{1n}(z))=k\},
	\end{align*}
	The closure $\mathcal{G}_{n,k,2}:=\overline{\mathcal{G}_{n,k,2}^{\circ}}$ in $\Gr_{2,n}$ is
	the \emph{$\mathcal{G}$-amplituhedron}. 
\end{definition}

\begin{rmk} Following the sign-flip descriptions from
	\cite{ATT,karpwilliams}, one can 
	generalise most of the definitions in this section for any $m$. We leave this to future work.
\end{rmk}
Comparing with  \cref{cor:G} we  have:
\begin{prop}\label{prop:Bamplchar}
	Fix $k<n$, and $W \in \Gr_{k+2,n}^{>0}$.
	Then 
	\begin{equation*}
		\mathcal{G}^{\circ}_{n,k,2}(W) = \{z \in \mathcal{G}^{\circ}_{n,k,2} \ \vert \ z \subset W \} =
		\mathcal{G}^{\circ}_{n,k,2} \cap \Gr_2(W) \quad \text{ and } \quad 
		\mathcal{B}_{n, k, 2}(W)= \overline{\mathcal{G}^{\circ}_{n,k,2} \cap \Gr_2(W)}.
	\end{equation*}
\end{prop}

\begin{rmk}
	Note that $\mathcal{G}_{n,k,2}$ is full-dimensional in $\Gr_{2,n}$, i.e. it has dimension $2(n-2)$, whereas $\mathcal{G}^\circ_{n,k,2}(W)$ and $\mathcal{B}_{n,k,2}(W)$ are full-dimensional in $\Gr_2(W)$, i.e.  have dimension $2k$. 
\end{rmk}

Motivated by the decomposition of $\AA_{n,k,2}(Z)$ into $w$-chambers, we analogously define $w$-chambers for $\mathcal{G}_{n,k,2}$.
\begin{defn}\label{def:ampchamber2}
	Let $w \in D_{k+1,n}$ and let $I_a:=\cdes(w^{(a-1)})$.
	Then the \emph{open $\mathcal{G}$-amplituhedron $w$-chamber} $\hat{\Delta}_w^{\circ}(\mathcal{G})$ consists of 
	$z \in \mathcal{G}_{n,k,2}$ with all nonzero Pl\"ucker coordinates such that for $a=1, \dots, n$,
	\begin{equation*}
		\flip(	p_{a \hat{1}}(z) , p_{a \hat{2}}(z), \dots, p_{a \widehat{a-1}}(z), p_{a a}(z), p_{a a+1}(z), \dots, p_{a n}(z))=
		I_a \setminus \{a\}.
	\end{equation*}	 	
	Equivalently, $\Delta^{\circ}_w(\mathcal{G})$ consists of 
	$z \in \Gr_{2,n}$ such that
	\begin{equation}
		\sgn p_{ a j}(z)= (-1)^{|I_a \cap [a, j-1]|-1}  \text{ for } j>a \quad \text{ and } \quad 
		\sgn p_{a \hat{j}}(z)= (-1)^{|I_a \cap [a, j-1]|-1}  \text{ for } j< a. \label{eq:sgnCond}
	\end{equation}
	The \emph{closed $\mathcal{G}$-amplituhedron $w$-chamber} is the closure $\hat{\Delta}_w(\mathcal{G}):= \overline{\hat{\Delta}^{\circ}_w(\mathcal{G})}$.
	Abusing notation, we will often omit `closed' when referring to closed $\mathcal{G}$-amplituhedron $w$-chambers.
\end{defn}

The situation for $\mathcal{G}$-amplituhedron $w$-chambers is quite straightforward. We will see that 
the second part of \eqref{eq:sgnCond} follows from the first part, 
so each $\hat{\Delta}_w^{\circ}(\mathcal{G})$ is an oriented matroid stratum, whose underlying matroid is the rank 2 uniform matroid on $[n]$.

\begin{prop} \label{prop:GchamberNonempty}
	Let $w \in D_{k+1, n}$. Then 
	$\hat{\Delta}_w^{\circ}(\mathcal{G})$ 
	is nonempty and is contractible.
\end{prop}

\begin{proof} 
	Consider $n$ vectors $v_1, v_2, \dots, v_n$ in $\RR^2$ so that the matrix
	\[\begin{bmatrix}
		v_1 & v_{w_1 +1}& v_{w_2+1} &\dots& v_{w_{n-1}+1}
	\end{bmatrix}\]
	has all maximal minors positive. In particular, drawing the vectors in the plane and going counterclockwise, we see $v_1, v_{w_1 +1}, v_{w_2+1}, \dots, v_{w_{n-1}+1}$ in that order.
	
	Now, set $z_1:=v_1$ and $z_b := (-1)^{|I_1 \cap [1, b-1]|-1}v_b$ for $b \geq 2$. We claim that 
	\begin{equation*}
		z=\begin{bmatrix}
		z_1 & z_2& z_3 &\dots& z_n
	\end{bmatrix}\end{equation*}
	represents a point in $\hat{\Delta}_w^{\circ}(\mathcal{G})$.
	
	Clearly $p_{1b}(z)$ has the correct sign. Consider $1 \neq a<j$. We will assume $\det [v_a v_j]>0$; the other case is similar. Note that $p_{aj}(z)$ has sign $(-1)^{|I_1 \cap [a, j-1]|}$; we would like to show that this is equal to $(-1)^{|I_a \cap [a, j-1]|-1}$. Because $\det[v_a v_j]>0$, $a-1$ occurs before $j-1$ in $w$, written in one-line notation. Recall from \cref{rmk:circuitDefn} that $I_{w_i +1}= I_{w_{i-1}+1} \setminus \{w_i\} \cup \{w_i+1\}$. That is, $I_a$ can be obtained from $I_1$ by removing $w_1$ and adding $w_1+1$, then removing $w_2$ and adding $w_2+1$, and so on until one removes $w_q=a-1$ and adds $a$. Note that for $c=w_1, \dots, w_{q-1}$, the numbers $c$ and $c+1$ are either both in $[a, j-1]$ or both not in $[a, j-1]$, so $|I_1 \cap [a, j-1]|=|I_{c+1} \cap [a, j-1]|$. Removing $a-1$ from $I_{w_{q-1}+1}$ and adding $a$ increases the size of the intersection with $[a, j-1]$ by one, so $|I_1 \cap [a, j-1]|= |I_a \cap [a, j-1]|-1$. This shows $p_{aj}(z)$ has the correct sign for $a<j$; a similar argument shows that for $a>j$, $p_{a \hat{j}}(z)$ has the desired sign so long as $p_{ja}(z)$ does.
	
	So $\hat{\Delta}_w^{\circ}(\mathcal{G})$ is an oriented matroid stratum for a rank 2 oriented matroid. By \cite[Corollary 8.2.3]{OrientedMatroidBook}, all rank 2 oriented matroid strata are contractible.
\end{proof}

\begin{example} \label{ex:realizable}
	Let $w=(2,6,1,4,5,3,7)\in D_{k+1,n}$ with $k=3$ and $n=7$.
	We have $I_1 = \{1,2,4,6\}$.
	Following the proof of \cref{prop:GchamberNonempty}, we can choose 
	\begin{align*}
		(v_1, v_{w_1 +1}, v_{w_2 +1}, v_{w_3+1}, v_{w_4+1}, v_{w_5+1}, v_{w_6+1})
		&= (v_1, v_3, v_7, v_2, v_5, v_6, v_4) \\ 
		&= \begin{pmatrix}
			1 & 1 & 1 & 1 & 1 & 1 & 1 \\
			1 & 2 & 3 & 4 & 5 & 6 & 7
		\end{pmatrix}.
	\end{align*}
	We then get 
	\begin{equation*}
		z = \begin{pmatrix}
			1 & 1 & -1 & -1 & 1 & 1 & -1 \\
			1 & 4 & -2 & -7 & 5 & 6 & -3
		\end{pmatrix}.
	\end{equation*}
	One can check that $z$ lies in $\hat{\Delta}_w^{\circ}(\mathcal{G})$.
	Also note that both row vectors $z^{(1)}$ and $z^{(2)}$ of $z$ have 
	$\var(z^{(1)}) = \var(z^{(2)})=k$ by construction.
\end{example}

\begin{rmk} \label{rk:wchambersasslices}
	The $w$-chambers of the $\mathcal{G}$-amplituhedron do \emph{not} depend on $Z$. Roughly speaking, the amplituhedron $w$-chambers are linear slices of $\mathcal{G}$-amplituhedron $w$-chambers. More precisely, for
	$Z\in \Mat^{>0}_{n,k+2}$ with column span $W\in \Gr_{k+2,n}^{>0}$, we have
	\begin{equation*}
		f_Z(\hat{\Delta}^\circ_w(\mathcal{G})\cap 
		\Gr_{2}(W))= \hat{\Delta}^\circ_w(Z),
	\end{equation*}
	where $f_Z$ is the homeomorphism from \cref{prop:twistor}.	
\end{rmk}

Our next goal is to use
\cref{prop:GchamberNonempty}
and the connection with the $\mathcal{B}$-amplituhedron from 
\cref{prop:Bamplchar}
to deduce \cref{thm:nonempty} on realizability of $w$-chambers. We start by proving the following lemma.
\begin{lemma}\label{lem:extend}
	Given a $2\times n$ matrix $z$ as constructed in the proof of \cref{prop:GchamberNonempty}, we can construct 
	a $(k+2) \times n$ matrix $A'$ representing a point $W\in \Gr_{k+2,n}^{\geq 0}$ which contains
	$\rowspan(z)$ as a subspace.
\end{lemma}
\begin{proof}
	Let $z^{(1)}=(z^{(1)}_1,\dots, z^{(1)}_n)$ and $z^{(2)}=(z^{(2)}_1,\dots,z^{(2)}_n)$ denote the rows of $z$.
	By construction, $\var(z^{(1)})=\var(z^{(2)})=k$ and moreover we can partition $[n]$ 
	into disjoint consecutive intervals $H_1 \sqcup \dots \sqcup H_{k+1}$ such that 
	the entries of $z^{(1)}$ and $z^{(2)}$ in positions $H_i$ are positive if $i$ is odd and negative if $i$ is even.
	
	By \cite[Lemma 4.1]{karp},  since $\var(z^{(2)})=k$,
	we can construct a $(k+1) \times n$ matrix $A$ with maximal minors nonnegative
	whose row sum  is $z^{(2)}$.  More explicitly, we define the $i$th row of $A$ to be the vector 
	$(a_{i1},\dots,a_{i2})$ such that $a_{ij}= z^{(2)}_j$ for $j\in H_i$ and $a_{ij}=0$ for $j \notin H_i$.
	Therefore the nonvanishing Pl\"ucker coordinates of $A$
	are precisely the $p_B(A)$ such that $B=\{b_1<b_2<\dots<b_{k+1}\}$ with $b_i \in H_i$.
	
	Let $A'$ be the matrix obtained from $A$ by adding $z^{(1)}$ as a new top ($0$th) row. 
	We will label the rows of $A'$ from $0$ to $k+1$.
	The nonvanishing Pl\"ucker coordinates of $A'$ 
	are precisely the $p_{B'}(A')$ where $B'=\{b_1<b_2<\dots <b_{k+1}\} \cup \{b'_j\}$
	with $b_i\in H_i$ and both $b_j, b'_j$ lie in $H_j$.
	
	Now we can compute the Pl\"ucker coordinates of $A'$ in terms of Pl\"ucker coordinates of $z$ and minors of $A$.
	Let $B'=\{b_1<b_2<\dots <b_{k+1}\} \cup \{b'_j\}$ as above.  Then 
	we have 
	\begin{align*}
		p_{B'}(A') &= (-1)^{j-1} \Delta_{0j, b_j b'_j}(A') \cdot \Delta_{[k+1]\setminus {j}, B'\setminus \{b_j, b'_j\}}(A')\\
		&= (-1)^{j-1} p_{b_j b'_j}(z) 
		\prod_{i \neq j} z^{(2)}_{b_i}
	\end{align*}
	where $\Delta_{R, C} (A')$ denotes the minor of $A'$ on rows $R$ and columns $C$. 
	Now it follows from the construction of $z$ that 
	since both $b_j, b'_j$ lie in $H_j$, we have 
	$p_{b_j b'_j}(z)>0$.  Additionally, we have that 
	the sign of $\prod_{i \neq j} z^{(2)}_{b_i}$ is $(-1)^{j+1}.$
	Therefore $p_{B'}(A')$ is positive, as desired.

\end{proof}

\begin{example} 
	\label{ex:realizable2}
	We illustrate the proof of 
	\cref{lem:extend} using 
	our running example from \cref{ex:realizable}.
	We have \begin{equation*}
		z = \begin{pmatrix}
			1 & 1 & -1 & -1 & 1 & 1 & -1 \\
			1 & 4 & -2 & -7 & 5 & 6 & -3
		\end{pmatrix}
	\end{equation*}
	so \begin{equation*}
		A = \begin{pmatrix}
			1 & 4 & 0 & 0 & 0 & 0 & 0\\
			0 & 0 & -2 & -7 & 0 & 0 & 0\\
			0 & 0 & 0 & 0 & 5 & 6 & 0\\
			0 & 0 & 0 & 0 & 0 & 0 & -3
		\end{pmatrix} \text{ and }
		A' = \begin{pmatrix}
			1 & 1 & -1 & -1 & 1 & 1 & 1\\
			1 & 4 & 0 & 0 & 0 & 0 & 0\\
			0 & 0 & -2 & -7 & 0 & 0 & 0\\
			0 & 0 & 0 & 0 & 5 & 6 & 0\\
			0 & 0 & 0 & 0 & 0 & 0 & -3
		\end{pmatrix}.
	\end{equation*}
	Both matrices have maximal minors nonnegative.
	If $B' = \{2,3,5,6,7\}$ then $2\in H_1, 3\in H_2, 5,6\in H_3, 7\in H_4$ 
	and we have 
	$$p_{B'}(A') = \Delta_{03,56}(A') \Delta_{124, 237}(A') = 
	p_{56}(z) \cdot (4 \cdot (-2) \cdot (-3)).$$
\end{example}

\begin{proof}[Proof of \cref{thm:nonempty}]
	By \cref{prop:twistor}, we know that $\mathcal{B}_{n,k,2}(W)$ is homeomorphic to 
	$\mathcal{A}_{n,k,2}(Z)$, where $W\in \Gr_{k+2,n}^{>0}$ is the column span of $Z$.
	Moreover the Pl\"ucker coordinates of the former agree with the twistor coordinates of the latter.
	\cref{prop:GchamberNonempty} gives an explicit construction of a $2 \times n$ matrix
	$z$ representing a point in 
	$\hat{\Delta}_w^{\circ}(\mathcal{G})$, and by \cref{prop:Bamplchar} we have 
	$	\mathcal{B}_{n, k, 2}(W)= \overline{\mathcal{G}^{\circ}_{n,k,2} \cap \Gr_2(W)},$
	so to prove the theorem, 
	we just need to realize $z$ as a two-dimensional subspace contained in some  $(k+2)$-plane
	$W\in \Gr_{k+2,n}^{>0}$.
	
	By \cref{lem:extend}, we can realize $z$ as a two-dimensional subspace contained in 
	a $(k+2)$-plane $W\in \Gr_{k+2,n}^{\geq 0}$.  (Here $W = \rowspan(A')$.)
	We want to now slightly deform $A'$ to make it totally positive.
	
	We claim that $A'\in \Gr_{k+2,n}^{\geq 0}$ is the limit of a sequence of 
	points $\{\tilde{A}_t\} \in \Gr_{k+2,n}^{>0}$ where 
	$\rowspan(\tilde{A}_t)$ contains a $2$-plane $z(t)$
	which 
	lies in the same sign-chamber as $z$. 
	To see this, 
	we use the fact that $\Gr_{k+2,n}^{\geq 0} = \overline{\Gr_{k+2,n}^{>0}}$ (see 
	\cref{rem:PL}).
	We can therefore write
	$A'$ as the limit of a sequence of matrices of the form 
	$A'+(\epsilon_{ij}(t)) \in \Gr_{k+2,n}^{>0}$, where 
	$(\epsilon_{ij}(t))$ is a $(k+2) \times n$ matrix, and each 
	$\epsilon_{ij}(t)$ is a function of $t$ with small absolute value and 
	$\epsilon_{ij}(t) \to 0$ as $t\to 0$.
	
	We denote the rows of 
	$A'+(\epsilon_{ij}(t))$ by $r_i(t)$ for $0\leq i \leq k+1$.
	Let $z^{(1)}(t): = r_0(t)$, 
	let $z^{(2)}(t): = r_1(t)+r_2(t)+\dots + r_{k+1}(t),$ and 
	let $z(t)$ be the matrix with rows 
	$z^{(1)}(t)$ and $z^{(2)}(t)$. 
	
	Then when $t=0$, we have $z=z(t)$.  Moreover for small $t$,
	the Pl\"ucker coordinates of $z(t)$ have the same signs as the Pl\"ucker 
	coordinates of $z$, so $z(t)$ lies in
	the same $w$-chamber $\hat{\Delta}_w^{\circ}(\mathcal{G})$ as $z$.
	But now by construction, $\rowspan(z(t))$ lies in the positive $(k+2)$-plane 
	$W = \rowspan(A'+(\epsilon_{ij}(t)))$.
	This completes the proof of the theorem.
\end{proof}

Recall the definition of realizable amplituhedron chamber from \cref{def:realizable}.
\begin{cor}[Amplituhedron chambers and Eulerian numbers]\label{cor:realizableChambers}
	The realizable amplituhedron chambers $\AA^{\sigma}_{n,k,2}$ are exactly the $w$-chambers 
	$\hat{\Delta}_w^{\circ}$ where $w \in D_{k+1, n}$.
\end{cor}
\begin{proof}
	\cref{thm:nonempty} shows that each $w$-chamber is realizable. \cref{thm:wSimpCover} shows that no other sign chambers are realizable.
\end{proof}

We now turn to the 
$\mathcal{G}$-amplituhedron. The proof of Theorem \ref{thm:wSimpCover} implies the following.
\begin{thm}\label{thm:GAmplWSimpl} Fix $k<n$, then
	\begin{equation*}
		\mathcal{G}_{n,k,2}= \bigcup_{w \in D_{k+1, n}} \hat{\Delta}_w(\mathcal{G}).
	\end{equation*}
\end{thm}

	Using the sign characterization of a positroid tile $Z_{\mathcal{T}}$ of $\mathcal{A}_{n,k,2}(Z)$ (\cref{thm:surjectivity}), one can define a \emph{positroid tile} $\mathcal{G}_{\mathcal{T}}$ in $\mathcal{G}_{n,k,2}$ as (the closure of) the region in $\Gr_{2,n}$ whose Pl\"ucker coordinates satisfy the same sign conditions as the twistor coordinates of $Z_{\mathcal{T}}$. Analogously to \cref{cor:GTunionWSimp}, $\mathcal{G}_{\T}$ is a union of $\mathcal{G}$-amplituhedron $w$-chambers. Moreover, $Z_{\T}$ is a linear slice of $\mathcal{G}_{\T}$ (analogously to \cref{rk:wchambersasslices}). We say a \emph{positroid tiling} of $\mathcal{G}_{n,k,2}$ is a collection of positroid tiles which cover $\mathcal{G}_{n,k,2}$ and have disjoint interiors. 
	Since all $\mathcal{G}$-amplituhedron $w$-chambers are non-empty, the analogue of \cref{thm:simpTriangGiveAmp} holds for the $\mathcal{G}$-amplituhedron (without any dependence on $Z$): T-duality gives a bijection between positroid tilings of $\Delta_{k+1, n}$ and positroid tilings of $\mathcal{G}_{n, k, 2}$.

\begin{definition}[Total Amplituhedron]
	The \emph{total  amplituhedron} $\mathcal{G}_{n}^{(2)}$ is
	\begin{equation*}
		\mathcal{G}_{n}^{(2)}:= \bigcup_{k=0}^{n-2} \mathcal{G}_{n,k,2}.
	\end{equation*}
\end{definition}
%
%
%
Note that $\mathcal{G}_{n}^{(2)}$ has top dimension $2(n-2)$ in $\Gr_{2,n}$,  
and it does \emph{not} depend on $Z$.

Recall that the 
hypercube $\mbox{\mancube}_{n-1}\subset \R^{n-1}$ 
can be decomposed into $(n-1)!$ $w$-simplices in a way which is compatible with 
its slicing into (projected) hypersimplices 
$\tilde{\Delta}_{1,n}, \tilde{\Delta}_{2,n}, \ldots,\tilde{\Delta}_{n-1,n}$. Each 
$\tilde{\Delta}_{k+1,n}$ is a union of exactly $E_{k,n-1}$ simplices,
where $E_{k,n-1}$ is the Eulerian number.


Analogously, 
by \cref{thm:GAmplWSimpl},  
the total amplituhedron $\mathcal{G}_n^{(2)} \subset \Gr_{2,n}$
can be decomposed into $(n-1)!$ $w$-chambers 
in a way which is 
compatible with its decomposition into 
the $\mathcal{G}$-amplituhedra $\mathcal{G}_{n,0,2},$ $\mathcal{G}_{n,1,2}, \ldots,\mathcal{G}_{n,n-2,2}$. Each $\mathcal{G}_{n,k,2}$ is a union of exactly $E_{k,n-1}$ $w$-chambers.
This is the `$m=2$' equivalent of encoding all helicity sectors at once for tree-level scattering amplitudes of $\mathcal{N}=4$ SYM for $m=4$. 
A related space was discussed in the context of the $\mathcal{B}$-amplituhedron \cite[Section 3.4]{karpwilliams}.

\subsection{Empty $w$-chambers and tilings of $\AA_{n,k,2}$}\label{subsec:wChambersEmpty}
 In this section we provide algorithms to find \emph{all} positroid tilings of the hypersimplex and the amplituhedron using $w$-simplices and $w$-chambers. 
 As mentioned in \cref{rmk:emptyWs}, $\asimp{w}$ may be empty for some choices of $Z\in \Mat^{>0}_{n,k+2}$. 
 We take a closer look at this phenomenon and give some examples.
 
 \begin{rmk}
 	 It is \emph{a priori} possible for an amplituhedron $\AA_{n, k, 2}(Z)$ to have a positroid tiling $\hat{\mathcal{C}}$ which is not T-dual to a hypersimplex positroid tiling. However, \cref{thm:simpTriangGiveAmp} tells us that the collection of Grasstopes $\hat{\mathcal{C}}$ will fail to be a tiling for some other amplituhedron $\AA_{n, k, 2}(Z')$. We have not found any instances of such ``sporadic" tilings.
 \end{rmk}

\begin{prop}[Algorithm for positroid tilings of $\Delta_{k+1,n}$]
In order to find all positroid tilings of $\Delta_{k+1,n}$ proceed as follows. Call two positroid tiles $\Gamma_{\pi_1}$ and $\Gamma_{\pi_2}$ \emph{compatible} if they do not contain any common $w$-simplex.
\begin{itemize}
\item[Step 1.] Define a graph $\mathcal{G}$ whose vertices are positroid tiles of $\Delta_{k+1,n}$ and edges connect compatible positroid tiles.
\item[Step 2.] Compute the set $Cl(\mathcal{G})$ of all maximal cliques of $\mathcal{G}$;
\item[Step 3.] For each clique $\mathcal{C} \in Cl(\mathcal{G})$, compute the list $\mathcal{L}_\mathcal{C}$ of all $w$-simplices contained in any positroid tile $\Gamma_\pi \in C$;
\item[Step 4.] If $\mathcal{L}_\mathcal{C}$ consists of all $w$-simplices of $\Delta_{k+1,n}$, then $\mathcal{C}$ is a positroid tiling of   $\Delta_{k+1,n}$. Otherwise it is not.
\end{itemize}
\end{prop}

\begin{prop}[Algorithm for positroid tilings of $\AA_{n,k,2}$] \label{prop:triangamplalg}
In order to find all positroid tilings of $\AA_{n,k,2}(Z)$ proceed as follows. Let $\mathcal{E}_Z$ be the list of all $w$-simplices $\Delta_w$ in $\Delta_{k+1,n}$ such that $\asimp{w}=\emptyset$.
 Call two positroid tiles $Z_{\hat{\pi}_1}, Z_{\hat{\pi}_2}$ \emph{compatible} if and only if $\Gamma_{\pi_1} \cap \Gamma_{\pi_2}$ is empty or is the union of $w$-simplices which are in $\mathcal{E}_Z$.
\begin{itemize}
\item[Step 1.] Make a graph $\mathcal{\hat{G}}$ whose vertices are positroid tiles of $\AA_{n,k,2}(Z)$ and edges connect compatible positroid tiles;
\item[Step 2.] Compute the set $Cl(\hat{\mathcal{G}})$ of all maximal cliques of $\hat{\mathcal{G}}$;
\item[Step 3.] For each clique $\hat{\mathcal{C}} \in Cl(\hat{\mathcal{G}})$, consider the collection $\mathcal{C}$ of T-dual positroid tiles in $\Delta_{k+1,n}$. Compute the list $\mathcal{L}_{\mathcal{C}}$ of all $w$-simplices in $\Delta_{k+1,n}$ contained in any positroid tile $\Gamma_\pi \in \mathcal{C}$;
\item[Step 4.] If the (possibly empty) complement of $\mathcal{L}_\mathcal{C}$ is contained in $\mathcal{E}_Z$, then $\hat{\mathcal{C}}$ is a positroid tiling of $\AA_{n,k,2}(Z)$. Otherwise it is not.
\end{itemize}
\end{prop}
\begin{rmk}\label{rmk:counterexample}
If we would like to find a positroid tiling $\hat{\mathcal{C}}$ of the amplituhedron $\AA_{n,k,2}(Z)$ which is \emph{not} a positroid tiling of $\Delta_{k+1,n}$, then after Step 3 we need check that either: i) the complement of $\mathcal{L}_\mathcal{C}$ is \emph{nonempty} and contained in $\mathcal{E}_Z$; or ii) $\mathcal{L}_\mathcal{C}$ is the set of all $w$-simplices of $\Delta_{k+1,n}$ and there is a pair of positroid tiles $\Gamma_{\pi_1},\Gamma_{\pi_2}$ in $\mathcal{C}$ which both contain a $w$-simplex in $\mathcal{E}_Z$.
\end{rmk}

Below, we report some results on empty $w$-chambers in the cases $k=1,2$. 

{\bf $k=1$ Case.} The amplituhedron $\AA_{n,1,2}(Z)$ is just an $n$-gon  $\mathbf{P}_n(Z)$ in $\mathbb{P}^2$ with vertices $Z_1, \dots, Z_n$ going clockwise. Let $i \rightarrow j$ be a side or a diagonal of $\mathbf{P}_n(Z)$, with $i<j$. The twistor coordinate $\langle Y ij \rangle$ is positive, negative or zero if $Y$ lies to the right, left, or on the diagonal $i \rightarrow j$ respectively. Then the nonempty $w$-chambers $\asimp{w}$ are the connected components of the complement of all diagonals of $\mathbf{P}_n(Z)$ (see Figure~\ref{fig:EmptyWs}).
If no three diagonals of $\mathbf{P}_n(Z)$ intersect at a point in the interior, it is well known the number of connected components is given by:
\begin{equation*}
N_n= \sum_{r=2}^4 {n-1 \choose r}={n \choose 4}+{n-1 \choose 2}.
\end{equation*}
The number of empty $w$-chambers in this case is show in \cref{tab:emptyWk1}.

If three diagonals of $\mathbf{P}_n(Z)$ intersect at a point in its interior, then the number of empty $w$-chambers is larger (as the number of regions realized is smaller).

\begin{table}[h]
\caption{Empty $w$-chambers vs. Eulerian numbers for $k=1$.} 
\centering 
\begin{tabular}{|r||c|c|c|c|c|c|c|} 
\hline\hline 
 $n$ & $3$ & $4$ & $5$ & $6$ & $7$ & $8$ & $9$\\
\hline 
 $N_n$ & $1$ & $4$ & $11$ & $25$ & $50$ & $91$ & $154$\\
\hline 
 $E_{1, n-1}$ & $1$ & $4$ & $11$ & $26$ & $57$ & $120$ & $247$\\
\hline 
 $\#$ Empty $\hat{\Delta}_w$ & $0$ & $0$ & $0$ & $1$ & $7$ & $29$ & $93$\\
\hline 
\end{tabular}
\label{tab:emptyWk1}
\end{table}

\begin{example}
Consider $\AA_{6,1,2}(Z)$, which is an hexagon. 
Let us consider the permutations $w^{(+)}=145236$ and $w^{(-)}=341256$. Points in $\asimp{w^{(+)}}$ and $\asimp{w^{(-)}}$ have all twistor coordinates with the same sign, except for $\{\langle Y 14\rangle, \langle Y 25\rangle, \langle Y 36\rangle \}$, whose signs are $\{+ \, - \, +\}$ and $\{- \, + \, -\}$, respectively. Let $Z^*$ be the intersection of the diagonals $(1,4)$ and $(2,5)$. Then 
$\asimp{w^{(+)}}$ (respectively, $\asimp{w^{(-)}}$) is non-empty if and only if $Z^*$ is to the right (respectively, left) of the diagonal $3 \rightarrow 6$. This happens when 
\begin{equation}\label{eq:badfcn}
 \langle Z_1,Z_2,Z_5 \rangle \langle Z_4,Z_3,Z_6 \rangle-\langle Z_1,Z_3,Z_6 \rangle\langle Z_4,Z_2,Z_5 \rangle
\end{equation}
is positive (respectively, negative), see Figure~\ref{fig:EmptyWs}. So for any choice of $Z$, either $\asimp{w^{(+)}}= \emptyset$ or $\asimp{w^{(-)}} = \emptyset$, and both are empty if \eqref{eq:badfcn} vanishes. 
\begin{figure}[h]
\includegraphics[width=0.8\textwidth]{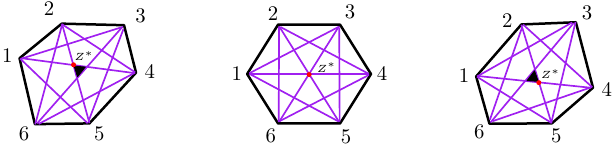}
\caption{From left to right: For $\langle Z^*,Z_3,Z_6 \rangle>0$, $\hat{\Delta}_{w^{(+)}}$ is nonempty (in black) but $\hat{\Delta}_{w^{(-)}}$ is empty;  if $\langle Z^*,Z_3,Z_6 \rangle=0$ then $\hat{\Delta}_{w^{(+)}}$ and $\hat{\Delta}_{w^{(-)}}$ are both empty; for $\langle Z^*,Z_3,Z_6 \rangle<0$ then $\hat{\Delta}_{w^{(+)}} $ is empty but $\hat{\Delta}_{w^{(-)}}$ is not (shown in black).}
	\label{fig:EmptyWs}
\end{figure}
\end{example}
If a collection of Grasstopes covers $\AA_{n, k, 2}(Z)$, the T-dual positroid polytopes may not cover $\Delta_{k+1,n}$.
\begin{example} Let $Z_0$ be the point in $Gr^{>0}_{3,6}$ invariant under cyclic symmetry, so $\AA_{6,1,2}(Z_0)$ is a regular hexagon in $\mathbb{P}^2$. Consider the positroid tiles $Z_{\hat{\pi}_1}, \ldots,Z_{\hat{\pi}_6}$ below. 
	\vspace{12pt}
	
	\begin{center}
		\includegraphics[width=\textwidth]{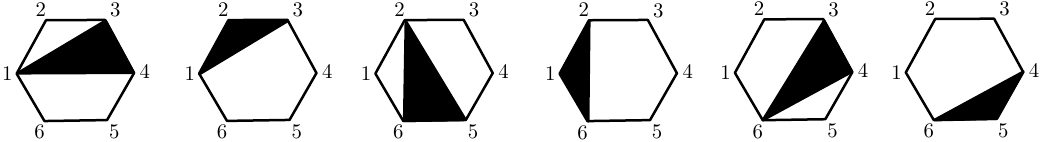}
	\end{center}

\vspace{12pt}
	
	 Clearly, they \emph{do} cover $\AA_{6,1,2}(Z_0)$ (and overlap). However, $\Delta_{w^{(+)}}$ is not contained in $\Gamma_{\pi_1} \cup  \cdots \cup \Gamma_{\pi_6} \subset \Delta_{2,6}$. Therefore the T-dual positroid tiles do \emph{not} cover $\Delta_{2,6}$.
\end{example}

Despite the presence of empty $w$-chambers, for \emph{any} $Z$ in $\Mat^{>0}_{k+2,n}$, positroid tilings of $\Delta_{2,n}$ and $\AA_{n,1,2}(Z)$ are still in bijection:
\begin{prop}\label{prop:Triangsk1}
	A collection of tree positroid polytopes $ 
	\{\Gamma_{\pi}\}$ is a positroid tiling of $\Delta_{2, n}$ if and only if 
	$\{\gt{\td{\pi}}\}$ is a positroid tiling of $\AA_{n, 1, 2}(Z)$. All such tilings are regular.
\end{prop}
\begin{proof}
The forward direction comes from Theorem \ref{thm:simpTriangGiveAmp}. The other direction
	comes from the fact that $\AA_{n, 1, 2}(Z)$ is just an $n$-gon.  Its positroid tilings
 are in bijection with the regular positroid tilings of $\Delta_{2,n}$ described in \cite[Proposition 10.7]{LPW} (of `Catalan' type).
\end{proof}

{\bf $k=2$ Case.}
We used 
Mathematica and the package `\texttt{positroid}' \cite{Bourjaily:2012gy}. 

For $n=6$, 
 there are choices of $Z$
such that all $w$-chambers of 
$\AA_{6,2,2}(Z)$ are nonempty.

For $n=7$ and some choices of $Z$, there are empty $w$-chambers $\hat{\Delta}_{w}$ for which 
$\Delta_{w}$ is the intersection of just $2$ positroid tiles of $\Delta_{3,7}$. This  implies that in general the compatibility graph $\mathcal{G}$ of positroid tiles of 
$\Delta_{k+1,n}$  differs from the one $\hat{\mathcal{G}}$ of positroid tiles of $\mathcal{A}_{n,k,2}(Z)$ (cf. \cref{prop:triangamplalg}). For example, if $w=1645237$, then $\Delta_{w}=\Gamma_{\pi_1} \cap \Gamma_{\pi_2}$, with $\pi_1=2 3 7 1 6 4 5$ and  $\pi_2=6 7 4 5 2 3 1$. The positroid polytopes $\{\Gamma_{\pi_1}, \Gamma_{\pi_2}\}$ are \emph{not} compatible in $\Delta_{3,7}$, but there are choices of $Z=Z^*$ for which the T-dual Grasstopes $\{Z_{\pi_1}, Z_{\pi_2}\}$ are compatible in $\AA_{7,2,2}(Z^*)$, as $Z_{\pi_1} \cap Z_{\pi_2} = \hat{\Delta}_{w}=\emptyset$.
Nevertheless, the $3073$ positroid tilings of $\AA_{7,2,2}(Z^*)$ are still in bijection with the $3073$ positroid tilings of $\Delta_{3,7}$.

For $n=8$, we checked only a few choices of $Z$, but found that 
there are more than 100 $w$-chambers which can be empty depending on $Z$.
As in the $n=7$ case, the compatibility graph  of $\Delta_{3,8}$
differs from that of $\mathcal{A}_{8,2,2}(Z)$. 
Nevertheless, for all such such choices of $Z$,
the $6443460$ positroid tilings of $\AA_{8,2,2}(Z)$
are  in bijection with the positroid tilings of $\Delta_{3,8}$.

\subsection{Descent and sign-flip tilings}

Recall that permutations and their cyclic descents were used to
define both the $w$-simplices in $\Delta_{k+1, n}$ and the $w$-chambers in $\AA_{n, k, 2}(Z)$. 
In the same spirit, by refining the set of permutations based on the 
\emph{positions} of the  descents, we will obtain a 
distinguished positroid tiling of $\Delta_{k+1,n}$ and a distinguished positroid
tiling of  $\AA_{n,k,2}(Z)$.  These tilings are T-dual to each other. 

Recall that
\[\Delta_{k+1, n}=\bigcup_{w \in D_{k+1, n}} \simp{w}.\]
Since 1 is always a cyclic descent of $w \in D_{k+1, n}$, we have that
$D_{k+1, n}$ is the set of permutations $w\in S_n$ with $k$ left descents and $w(n)=n$. 
The Eulerian numbers have a very natural refinement by descent set $\des(w)$. 
If $w\in S_n$ has $w(n)=n$, then neither 1 nor $n$ is a left descent of $w$, so we have
\begin{equation*} \label{eq:EulerianNumbersId}
	E_{k, n-1}=\sum_{I \in {[2, n-1] \choose k}} \# \{w \in S_{n}: w(n)=n, \des(w)=I \}.
\end{equation*}

This inspires the following decomposition of $\Delta_{k+1, n}$. For $I \in \binom{[2, n-1]}{k}$, let 
\[\Gamma_I:= \bigcup_{\substack{{w \in D_{k+1, n}}\\{\des(w)=I}}} \simp{w}.\]
Clearly, the collection of $\Gamma_I$ cover the hypersimplex and their interiors are pairwise disjoint. There are also $\binom{n-2}{k}$ of them, which is exactly the number of full-dimensional positroid polytopes in a regular positroid tiling of $\Delta_{k+1,n}$ \cite{SW}.
We will show that each $\Gamma_I$ is in fact a positroid polytope, and that $\{\Gamma_I\}$ is a (regular) positroid tiling of $\Delta_{k+1,n}$. We will refer to it as the \emph{descent tiling}.

On the other hand, given the sign-flip characterization
of the amplituhedron from \cref{thm:main1}, 
it is natural to subdivide $\AA_{n,k,2}(Z)$ into regions based on where 
the sequence 
$(\langle Y 1a \rangle)_{a=1}^n$
has 
sign flips.  That is, for each $I \in {[2, n-1] \choose k}$, we define\footnote{Because of our conventions regarding sign flips, $Z_I$ would be empty if $1, n \in I$.}
$$\gto{I}:=\{Y\in \AA_{n,k,2}(Z) \ \vert \ 
\flip(\lrangle{Y11},\langle Y 12 \rangle, \langle Y 13 \rangle, \ldots, \langle Y 1n \rangle) = I\}$$
and define $\gt{I}$ to be the closure of $\gto{I}$.

\cite[Section 7]{ATT} conjectured that 
$\{Z_I\}$ is a positroid tiling of 
$\AA_{n,k,2}(Z)$. The authors referred to $\{Z_I\}$ as a \emph{sign-flip (or kermit\footnote{For $k=2$, $\AA_{n,2,2}(Z)$ provides the integrand for the $1$-loop $n$-point scattering amplitude in $\mathcal{N}=4$ SYM.  The name `kermit' comes from the resemblance of the pictorial expansion of such amplitude (e.g. see \cite[pg. 18]{ArkaniHamed:2010kv}) with the Muppet character `Kermit the Frog'.}) tiling}.
In this section  we prove this conjecture.  Moreover we show that sign-flip tilings of 
$\AA_{n,k,2}(Z)$ and descent tilings of the hypersimplex $\Delta_{k+1,n}$ are T-dual to each other and also regular.

\begin{definition}[Bicolored triangulations of kermit type]\label{def:kermit}
	Let $I=\{i_1,\ldots,i_k\} \in {[2, n-1] \choose k}$ and let $\T_I$ be the bicolored triangulation whose black triangles have vertices $\{1,i_{\ell},i_{\ell}+1\}$ for $\ell=1, \dots k$. We say $\T_I$ is \emph{kermit type} and denote the plabic graph $\hat{G}(\T_I)$ by $K_I$. We also denote the plabic graph $G(\T_I)$ by $C_I$, and call it a \emph{caterpillar} tree.
\end{definition}
\vspace{-.4cm}
\begin{figure}[h]
	\includegraphics[width=0.25\textwidth]{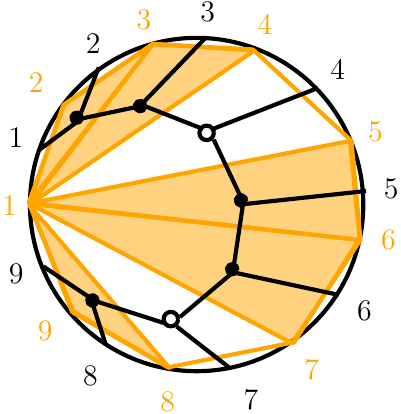}
	\caption{In orange, the bicolored triangulation $\T_I$ of kermit-type for $I=\{2,3,5,6,8\}$. In black, the dual caterpillar tree $C_I$. }
	\label{fig:kermitExample}
\end{figure}
\vspace{-.4cm}
\begin{prop}[Descent and sign-flips tilings are T-dual]\label{prop:eulersigntriangulations} Let $I$ run over ${[2, n-1] \choose k}$. The collections $\{\Gamma_I\}$ and $\{Z_I\}$ are T-dual regular positroid tilings of $\Delta_{k+1, n}$ and $\AA_{n,k,2}(Z)$. 
	Furthermore, $\Gamma_I=\Gamma_{C_I}$ and $\gt{I}=\gt{K_I}$ where $C_I$ and $K_I$ are as in \cref{def:kermit}.
\end{prop}

\begin{proof}
	By the sign description of $\gto{K_I}$ in \cref{thm:surjectivity}, it is straightforward that $Z_I=Z_{K_I}$. Moreover, using \cref{def:ampchamber}
	and Corollary \ref{cor:GTunionWSimp}, we have
	\begin{equation} \label{eq:kermitDecomp}
		Z_{K_I}=\bigcup_{w: \des(w)=I} \asimp{w}.
	\end{equation}
	Using \cref{prop:treefacet}, it is not hard to check that the positroid polytope $\Gamma_{C_I}$ satisfies
	\begin{equation}\label{eq:caterpillarDecomp}
		\Gamma_{C_I}=\bigcup_{w: \des(w)=I} \Delta_w.
	\end{equation}
	But this is exactly $\Gamma_I$. 
	
	Finally, it is easy to check that $\{\Gamma_{C_I}\}_{I \in {[2, n-1] \choose k}}$ is a positroid tiling of $\Delta_{k+1,n}$ of the sort appearing in \cite[Proposition 10.7]{LPW} (`Catalan type'), hence is a regular positroid tiling. It follows that $\{Z_{K_I}\}_{I \in {[2, n-1] \choose k}}$ is the T-dual regular positroid tiling.
\end{proof}

\begin{rmk}
	Sign-flip tilings of $\AA_{n,k,2}(Z)$ and descent tilings of $\Delta_{k+1,n}$ are of BFCW type (in particular, of `Catalan type', see \cite[Proposition 10.7]{LPW}).
\end{rmk}

We end this section by describing each $w$-simplex (resp. $w$-chamber) as an intersection of cyclically shifted caterpillar positroid polytopes (resp. kermit Grasstopes). For $I \subset [n]$ and $a \in [n]$, let $I^{(a)}$ denote the cyclic shift of $I$ such that $1 \mapsto a$. Similarly, for $G$ a plabic graph, let $G^{(a)}$ denote the cyclic shift of $G$ such that $1 \mapsto a$.

\begin{prop} \label{prop:wSimplChambInt}
	Let $\simp{w}, 
	\asimp{w}$ 
	be a $w$-simplex and a $w$-chamber in $\Delta_{k+1,n}$ and $\AA_{n,k,2}(Z)$ respectively. Let $I_1, \dots, I_n$ give the vertices of $\simp{w}$, and let $J_a:= (I_a \setminus \{a\})^{(2-a)}$.
	Then:
	\begin{equation*}
		\simp{w}=\bigcap_{a \in [n]} \Gamma_{C^{(a)}_{J_a}}\quad \text{ and } \quad \asimp{w}=\bigcap_{a \in [n]} Z_{K^{(a)}_{J_a}}.
	\end{equation*}
\end{prop}

\begin{proof}
	First, $C^{(a)}_{J_a}$ is dual to a kermit-type bicolored triangulation whose black triangles all use vertex $a$.
	
	To see the statement about $\simp{w}$, note that another way to phrase \eqref{eq:caterpillarDecomp} is that $\Gamma_{C_I}$ is the union of all $w$-simplices with 1st vertex given by $I \cup \{1\}$. Using the cyclic shift on the hypersimplex, it is not hard to see that $\Gamma_{C^{(a)}_{J_a}}$ is the union of all $w$-simplices with $a$th vertex given by $I_a$. So taking the intersection gives exactly the $w$-simplex with vertices $e_{I_1},\dots, e_{I_a}$.

	The statement about $\asimp{w}$ follows from a similar argument, using \eqref{eq:kermitDecomp} and the cyclic shift on $\Gr_{k, n}$.
	
\end{proof}

\begin{example}
	Let us consider $w=324156$ from Example \ref{ex:descents}. We have:
	\begin{align*}
		&I_1=\{1,2,3\},&& I_2=\{2,3,5\},&& I_3=\{1,3,4\},&& I_4=\{1,2,4\},&& I_5=\{1,3,5\},&& I_6=\{2,3,6\}; \\
		&J_1=\{2, 3\},&& J_2=\{2, 4\}, &&J_3=\{2, 5\}, &&J_4=\{4, 5\},&& J_5=\{3, 5\}, &&J_6=\{3, 4\}.
	\end{align*}
	Then $\Delta_w$ is the intersection of $\Gamma_{C^{(1)}_{J_1}}, \ldots, \Gamma_{C^{(6)}_{J_6}}$ and $\hat{\Delta}_w$ is the intersection of $Z_{K^{(1)}_{J_1}}, \ldots, Z_{K^{(6)}_{J_6}}$. The cyclically rotated kermit-type bicolored triangulations $\T^{(1)}_{J_1}, \ldots,\T^{(6)}_{J_6}$ are displayed below.
	\begin{center}
		\includegraphics[width=\textwidth]{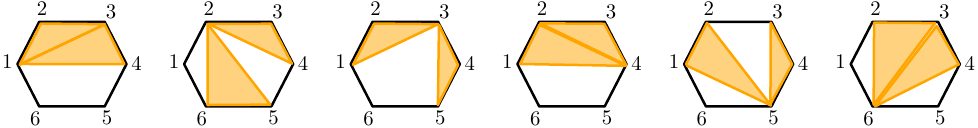}
	\end{center}
	Notice that $\T^{(1)}_{J_1}$ is equivalent to $\T^{(4)}_{J_4}.$
\end{example}



\section{Schr\"oder numbers: separable permutations and positroid tiles}
\label{sec:appendixB}

Recall from \cref{cor:GTsInBijection}
that positroid tiles for both $\AA_{n,k,2}(Z)$ 
and $\Delta_{k+1,n}$
are in bijection with 
bicolored subdivisions of type $(k,n)$ and tree positroids in $\Gr_{k+1,n}^{\geq 0}$.
\cite{Lukowski:2019sxw} provided experimental evidence that the number $R_{k,n-2}$ of positroid tiles for $\AA_{n,k,2}$ is given by  ~\cite[A175124]{OEIS}, a refinement of
the \emph{large Schr\"oder numbers} (see Table \ref{tab:Schroeder}). 
 In this section we prove this statement by 
giving a bijection between tree positroids in 
$\Gr_{k+1,n}^{\geq 0}$ and 
\emph{separable} permutations on $[n-1]$ with $k$ descents (enumerated by $R_{k,n-2}$).


\begin{definition}
	A permutation 
	$w=w_1 \ldots w_{n}$ (in one-line notation) 
	is  \emph{separable} if it is 3142- and 2413-avoiding, i.e. 
there are not four indices $i_1<i_2<i_3<i_4$ such that 
$w_{i_3}<w_{i_1}<w_{i_4}<w_{i_2}$ or  $w_{i_2}<w_{i_4}<w_{i_1}<w_{i_3}$. 
\end{definition}

\begin{table}[h] 
\centering 
\begin{tabular}{|l||c|c|c|c|c|c||r|} 
\hline 
\diagbox{$n$}{$k$} & $0$ & $1$ & $2$ & $3$ & $4$ & $5$  & $R_{n-2}$\\
\hline\hline 
 $2$ & $1$ & $ $ & $ $ & $ $ & $ $ & $ $ &  $1$ \\
 \hline 
 $3$ & $1$ & $1$ & $ $ & $ $ & $ $ & $ $ & $2$\\
\hline 
 $4$ & $1$ & $4$ & $1$ & $ $ & $ $ & $ $ & $6$\\
\hline 
 $5$ & $1$ & $10$ & $10$ & $1$ & $ $ & $ $ & $22$\\
 \hline 
 $6$ & $1$ & $20$ & $48$ & $20$ & $1$ & $ $ & $90$\\
  \hline 
 $7$ & $1$ & $35$ & $161$ & $161$ & $35$ & $1$ &  $394$\\
   \hline 
\end{tabular}
\caption{Large Schr\"oder numbers $R_{n-2}$ and their refinement $R_{k,n-2}$ which count the number of bicolored subdivisions of type $(k,n)$.} 
\label{tab:Schroeder}
\end{table}

\vspace{-14pt}

\begin{definition}
Let $\pi$ and $\nu$ be permutations on $[k]$ and $[l]$, respectively. The \emph{direct sum} $\pi  \oplus \nu$ and the \emph{skew sum}  $\pi \ominus \nu$ of $\pi$ and $\nu$ are permutations on $[k+l]$ defined by:
\begin{equation*}
(\pi  \oplus \nu)_i =
\begin{cases}
\pi_i, \quad i \in [1,k]\\
\nu_{i-k}+k,  \quad i \in [k+1,k+l]
\end{cases} ,
\quad (\pi  \ominus \nu)_i =
\begin{cases}
\pi_i+l, \quad i \in [1,k]\\
\nu_{i-k},  \quad i \in [k+1,k+l]
\end{cases}.
\end{equation*}
\end{definition}
For example, $123 \oplus 21=12354$ and $123 \ominus 21=34521$.

\begin{prop}[\cite{kitaev2011patterns}]
A permutation is separable if and only if $w$ can be built from the permutation $1$ by repeatedly applying $\oplus$ and $\ominus$.
\end{prop}
For example, the permutation $w=231654$ can be written as  
\begin{equation*}
\left((1 \oplus 1) \ominus 1 \right) \oplus \left((1 \ominus 1) \ominus 1 \right)=\left(12 \ominus 1 \right) \oplus \left(21 \ominus 1 \right)= 231 \oplus 321=231654.
\end{equation*}

\begin{prop} Let $\beta$ be the map sending a permutation 
	$w=w_1 \ldots w_{n-1}$ in one-line notation 
	to the permutation $\beta(w)=(w_1, \ldots ,w_{n-1} ,n)$ in cycle notation. Then $\beta$ is a bijection between separable permutations on $[n-1]$ with $k$ descents and trip permutations of tree positroids in $\Gr^{\geq 0}_{k+1,n}$.
\end{prop}
\begin{proof}
We use strong induction on $n$; 
the base case $n=2$ is trivial. 
It is enough to show that $\beta$ is well-defined and surjective.
	Suppose that $w \in S_{n-1}$ is separable. Then either $w=u \oplus v$ or  $w=u \ominus v$, for some $u \in S_{\ell-1}, v \in S_{r-1}$ separable, with $\ell-1+r-1=n-1$. By the induction hypothesis, $\beta(u)\in S_{\ell}$ and $\beta(v)\in S_r$ are the trip permutations 
	of  tree plabic graphs $S$ and $T$. We now ``glue" together $S$ and $T$ in order to obtain a tree plabic graph with boundary vertices
	$\{1,2,\dots,n\}$ with trip permutation $\beta(w)\in S_{n}$ (see \cref{fig:separableBij}).
	
	\begin{figure}[h]
		\includegraphics[width=0.8\linewidth]{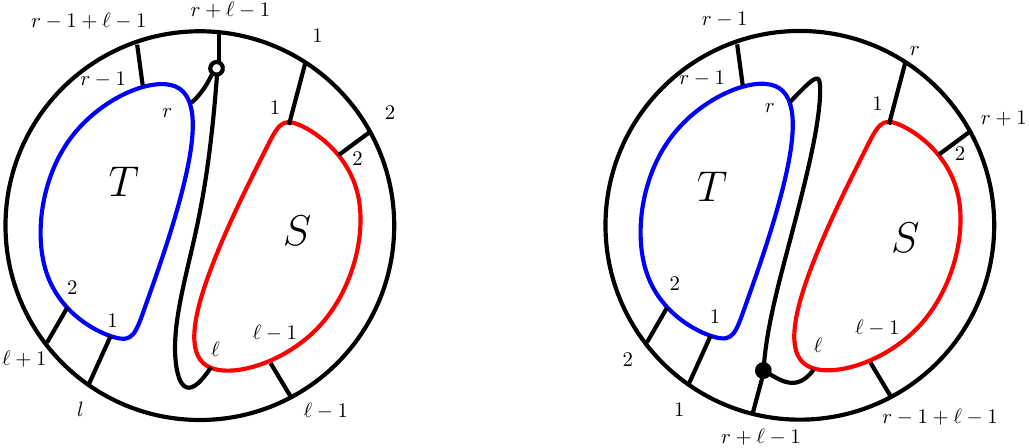}
		\caption{\label{fig:separableBij} How to glue $S, T$ together when $w= u \oplus v$ (on the left) and when $w= u \ominus v$ (on the right). }
	\end{figure}

It is straighforward to check that the trip permutation of the resulting tree is $\beta(w)$. This shows that $\beta$ is well-defined.

For surjectivity, consider a trivalent tree plabic graph $G$ on $[n]$. 
Let $v$ be the internal vertex adjacent to the boundary vertex $n$. 
Then deleting $v$ gives two trees: $S$ on $[\ell]$ and $T$ on $[\ell+1,n-1]$. Let $\pi$ be the trip permutation of $S$. Subtract $\ell$ from the boundary labels of $T$ to get a tree $T'$ on $[1,n-\ell-1]$ and let $\nu$ be its trip permutation. Then define $w$ to be either $\beta^{-1}(\pi) \oplus \beta^{-1}(\nu)$ or $\beta^{-1}(\pi) \ominus \beta^{-1}(\nu)$, based on whether $v$ is white or black. By the argument used above to show well-definedness, $\beta(w)$ is the trip permutation of $G$.
\end{proof}

\begin{rmk}
If $S$ and $T$ are tree plabic graphs, 
the positroids associated to $S \oplus T$ and $S\ominus T$ 
are the 
\emph{parallel-connection} and \emph{series-connection} of the 
matroids associated to $S, T$.
\end{rmk} 

The large Schr\"oder number $R_{n-2}$ counts separable permutations on $[n-1]$ \cite{WEST1995247}
and $R_{k,n-2}$ counts separable permutations on $[n-1]$ with $k$ descents \cite[Theorem 1.1]{FU20182616}.

\begin{corollary} 
Positroid tiles of $\Delta_{k+1}$ and $\AA_{n,k,2}(Z)$
are in bijection with separable permutations on $[n-1]$ with $k$ descents. They are enumerated by $R_{k,n-2}$ from~\cite[A175124]{OEIS}.  
\end{corollary}

\appendix
\section{Combinatorics of the totally nonnegative Grassmannian}
\label{sec:appendix}

In \cite{postnikov}, Postnikov defined several families of combinatorial objects which are in bijection with cells of the positive Grassmannian, including \emph{decorated permutations},
and equivalence classes of \emph{reduced plabic graphs}. 
He also used these objects to give concrete descriptions of 
the cells.  Here we review some of this technology.


\begin{defn}\label{defn:decperm}
A \emph{decorated permutation} on $[n]$ is a bijection $\pi : [n] \to [n]$ whose fixed points are each coloured either black (loop) or white (coloop). We denote a black fixed point $i$ by $\pi(i) = \underline{i}$, and a white fixed point $i$ by $\pi(i) = \overline{i}$.
An \emph{anti-excedance} of the decorated permutation $\pi$ is an element $i \in [n]$ such that either $\pi^{-1}(i) > i$ or $\pi(i)=\overline{i}$.  We say that a decorated permutation
	on $[n]$ is of
	\emph{type $(k,n)$} if it has $k$ anti-excedances.
\end{defn}

For example, $\pi = (3,\underline{2},5,1,6,8,\overline{7},4)$ has 
a loop in position $2$, and a coloop in position $7$.
It has three anti-excedances $1, 4, 7$.

Decorated permutations can be equivalently thought of as affine permutations \cite{KLS}.
	An \emph{affine permutation} on $[n]$ is a bijection $\pi : \mathbb{Z} \to \mathbb{Z}$ such that $\pi(i+n)=\pi(i)+n$ and $i \leq \pi(i) \leq i+n$, for all $i \in \mathbb{Z}$.
	It is additionally \emph{$(k,n)$-bounded} if 
	$\sum_{i=1}^n (\pi(i)-i) = kn$.


There is a bijection between decorated permutations of type $(k,n)$ and $(k,n)$-bounded 
affine permutations. Given a decorated permutation $\pi_d$ we can define an affine permutation $\pi_a$ by the following procedure: if $\pi_d(i)>i$, then define $\pi_a(i):=\pi_d(i)$; if $\pi_d(i)<i$, then define $\pi_a(i):=\pi_d(i)+n$; if $\pi_d(i)$ is a loop then define $\pi_a(i):=i$; if $\pi_d(i)$ is a coloop then define $\pi_a(i):=i+n$. 
For example, under this map, the decorated permutation $\pi_d=(3,\underline{2},5,1,6,8,\overline{7},4)$ in the previous example gives rise to $\pi_a=(3,2,5,9,6,8,15,12)$.


Given a $k\times n$ matrix $C = (c_1,\dots,c_n)$ written as a list of its columns,
we associate a decorated permutation $\pi$ as follows.
Given $i,j\in [n]$, let 
$r[i,j]$ denote the rank of $\langle c_i, c_{i+1},\dots, c_j\rangle$, 
where we list the columns in cyclic order, going from $c_n$ to $c_1$ if 
$i>j$.
We set $\pi(i):=j$ to be the label of the first column $j$ such that 
$c_i \in \spn \{c_{i+1}, c_{i+2},\dots, c_j\}$.
If $c_i$ is the all-zero vector, we decorate $i$ as loop, and if $c_i$ is not in the span of the other column
vectors, we decorate $i$ as coloop.

The map $C \mapsto \pi$ extends to a map on positroid cells.  Moreover,
Postnikov showed that the positroids for $Gr_{k,n}^{\ge 0}$ are in bijection with decorated permutations of $[n]$ with exactly $k$ anti-excedances (equivalently, by $(k,n)$-bounded affine permutations) \cite[Section 16]{postnikov}.
One may read off the dimension of the cell $S_{\pi}$ from the affine permutation 
$\pi$ as follows.  Let $\inv(\pi)$ be the number of pairs $(i,j)$ such that 
$i\in [n], j\in \Z, i<j$, and $\pi(i)>\pi(j)$.  Then the dimension of 
$S_{\pi}$ equals $k(n-k)-\inv(\pi)$.

\begin{defn}\label{def:plabic}
	A {\it planar bicolored graph} (or ``plabic graph'')
is a planar graph $G$ properly embedded into a closed disk, such that 
			each internal vertex is colored black or white;
			each internal vertex is connected by 
			a path to some boundary vertex;
			there are (uncolored) vertices lying on the 
			boundary of the disk labeled $1,\dots, n$
			for some positive $n$;
and 			each of the boundary vertices is incident to a single 
			edge.
See Figure \ref{G25} for an example.
\end{defn}
\vspace{-.4cm}
\begin{figure}[h]
\centering
\includegraphics[height=1in]{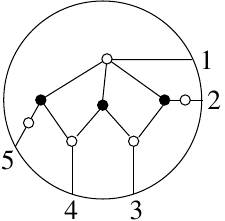}
\caption{A plabic graph}
\label{G25}
\end{figure}
\vspace{-.4cm}

If the connected component of $G$ attached to a boundary vertex $i$ is a 
path ending at a black (resp., white) leaf, we call this component a black (resp., white) 
\emph{lollipop}.  
\emph{We will require that our plabic graphs have no internal leaves
except for lollipops.}

There is a natural set of local transformations (moves) of plabic graphs:

(M1) \emph{Square move} (or \emph{urban renewal}).  If a plabic graph has a square formed by
four trivalent vertices whose colors alternate,
then we can switch the
colors of these four vertices.

(M2) \emph{Contracting/expanding a vertex}.
Two adjacent internal vertices of the same color can be merged.
This operation can also be reversed.  

(M3) \emph{Middle vertex insertion/removal}.
We can remove/add degree $2$ vertices.

See \cref{fig:M1} for depictions of these three moves.

\begin{figure}[h]
\centering
\includegraphics[height=.5in]{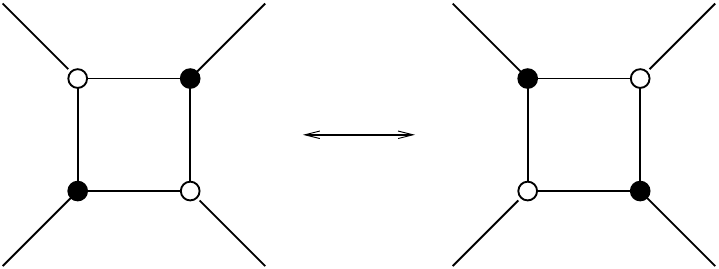}
\hspace{.3in}
\raisebox{6pt}{\includegraphics[height=.3in]{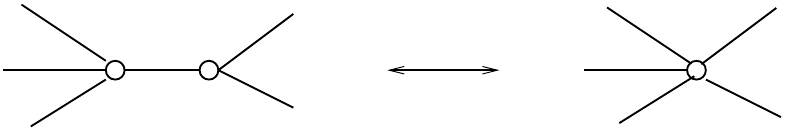}}
\hspace{.3in}
\raisebox{16pt}{\includegraphics[height=.07in]{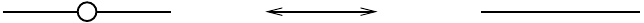}}
\caption{%
	Local moves (M1), (M2), (M3) on plabic graphs.}
\label{fig:M1}
\end{figure}
\begin{definition}\label{def:move}
Two plabic graphs are called \emph{move-equivalent} if they can be obtained
from each other by moves (M1)-(M3).  The \emph{move-equivalence class}
of a given plabic graph $G$ is the set of all plabic graphs which are move-equivalent
to $G$.
A plabic graph 
is called \emph{reduced} if there is no graph in its move-equivalence
in which two adjacent vertices $u$ and $v$ 
		are connected by 
			more than one edge
\end{definition}
	
 Note that given a plabic graph $G$,  we can always 
	apply moves to $G$ to obtain a new graph $G'$ which is bipartite.

\begin{defn}\label{def:rules}
Let $G$ be a reduced plabic graph as above with boundary vertices $1,\dots, n$. For each boundary vertex $i\in [n]$, we follow a path along the edges of $G$ starting at $i$, turning (maximally) right at every internal black vertex, and (maximally) left at every internal white vertex. This path ends at some boundary vertex $\pi(i)$. By \cite[Section 13]{postnikov}, the fact that $G$ is reduced implies that each fixed point of $\pi$ is attached to a lollipop; we color each fixed point by the color of its lollipop. In this way we obtain the \emph{decorated permutation} $\pi_G = \pi$ of $G$. We say that $G$ is of {\itshape type} $(k,n)$, where $k$ is the number of anti-excedances of $\pi_G$. 
\end{defn}
The decorated permutation of the plabic graph $G$ of 
\cref{G25}
is $\pi_G = (3,4,5,1,2)$, which has 
$k = 2$ anti-excedances.

\begin{definition}\label{def:matching}
	Let $G$ be a bipartite plabic graph.  Use move (M3) to 
ensure that each boundary vertex
is incident to a white vertex.
An {\it almost perfect matching} $M$ of a plabic graph $G$ 
is a
subset $M$ of edges such that each internal vertex is incident to
exactly one edge in $M$ (and each boundary vertex $i$ is incident
to either one or no edges in $M$). 
	We let $\partial M = \{i \ \vert \ i \text{ is incident to an edge of }M\}$.
\end{definition}

We associate to each  graph $G$ as above
a collection of 
subsets $\M(G) \subset [n]$ as follows.

\begin{proposition}\label{prop:positroidMatching}
\cite[Propostion~11.7, Lemma~11.10]{postnikov}
	Let $G$ be a plabic graph as in \cref{def:matching}, and let 
$\M(G) = \{\partial M \ \vert \ M \text{ an almost perfect matching of }G\}.$
	Then $\M(G)$ is the set of bases of a positroid on $[n]$. Its rank is 
	$\#\{\text{white vertices of }G\}-\#\{\text{black vertices of }G\},$
	which is the size of $\partial M$ for any almost perfect matching $M$ of $G$.
\end{proposition}


Postnikov used plabic graphs to give parameterizations of cells
of $\Grk$.  These parameterizations of cells 
can be recast as a variant of a theorem of Kasteleyn,
as was made explicit in \cite{Speyer}.  We follow the exposition 
there.

\begin{theorem}\cite{Speyer}\label{thm:Kasteleyn}
Let $G$ be a bipartite graph with boundary embedded in a disk, such 
that all of the boundary vertices are black.  Suppose there
are $N+k$ white vertices $W_1,\dots, W_{N+k}$,
$N$ internal black vertices $B_1,\dots, B_N$, and 
$n$ boundary vertices $B_{N+1},\dots, B_{N+n}$, labeled in 
clockwise order.  
Let $w:\Edges(G) \to \R_{>0}$ be any weighting function; if there
is an edge between vertices $i$ and $j$, we denote the weight on 
this edge by $w_{ij}$.
For a perfect matching $M$, define $w(M) = \prod_{e\in M} w(e)$ and define $\partial M$ to be the indices of the boundary vertices covered by an edge in $M$.
For a subset $I$ of $\{W_{N+1},\dots, W_{N+n}\}$, define
$\DD(G,I,w) = \sum_{\partial M = I} w(M)$. 

Then there is a real
	$k\times n$ \emph{Kasteleyn matrix} $L$ such that 
for each $k$-element subset $I$ of $\partial G$, the 
determinant $\det L_I$ of the $k\times k$ submatrix of $L$ using
the columns indexed by $I$ is $\det L_I = \DD(G,I,w).$
In particular, all Pl\"ucker coordinates of $L$ are non-negative.
\end{theorem}

The \emph{positroid cell} $S_G \subset \Gr_{k,n}$ associated to the plabic graph $G$
is the set of all $k$-planes in $\R^n$ spanned by matrices $L$ 
as in 
\cref{thm:Kasteleyn}.  If $G$ is a tree, we call $S_G$  a 
\emph{tree positroid}.

\begin{remark}\label{rm:Kasteleyn}
The Kasteleyn matrix $L$ is constructed as follows.
First construct an $(N+k)\times (N+n)$ matrix $K$, with 
rows indexed by white vertices and columns indexed by black vertices,
with $K_{ij} = \pm w_e$ if there is an edge $e$
between vertices $i$ and $j$ (otherwise $K_{ij}=0$).
Then, assuming $G$ has at least one perfect matching,
	we can apply row operations to transform $K$ into a 
matrix of block form 
$	\begin{pmatrix}
		\Id_N & \star \\
		0 & L
	\end{pmatrix}.$

\end{remark}

\bibliographystyle{alpha}
\bibliography{bibliography}

\end{document}